\documentclass{amsart}
\usepackage{amssymb}
\usepackage{amscd}
\usepackage{verbatim}
\usepackage{epsfig}
\begin{document}
\newcommand\Mand{\ \text{and}\ }
\newcommand\Mfor{\ \text{for}\ }
\newcommand\Real{\mathbb{R}}
\newcommand\RR{\mathbb{R}}
\newcommand\im{\operatorname{Im}}
\newcommand\re{\operatorname{Re}}
\newcommand\sign{\operatorname{sign}}
\newcommand\sphere{\mathbb{S}}
\newcommand\BB{\mathbb{B}}
\newcommand\HH{\mathbb{H}}
\newcommand\ZZ{\mathbb{Z}}
\newcommand\codim{\operatorname{codim}}
\newcommand\Sym{\operatorname{Sym}}
\newcommand\End{\operatorname{End}}
\newcommand\Span{\operatorname{span}}
\newcommand\Ran{\operatorname{Ran}}
\newcommand\ep{\epsilon}
\newcommand\Cinf{\cC^\infty}
\newcommand\dCinf{\dot \cC^\infty}
\newcommand\CI{\cC^\infty}
\newcommand\dCI{\dot \cC^\infty}
\newcommand\Cx{\mathbb{C}}
\newcommand\Nat{\mathbb{N}}
\newcommand\dist{\cC^{-\infty}}
\newcommand\ddist{\dot \cC^{-\infty}}
\newcommand\pa{\partial}
\newcommand\Card{\mathrm{Card}}
\renewcommand\Box{{\square}}
\newcommand\WF{\mathrm{WF}}
\newcommand\WFb{\mathrm{WF}_\bl}
\newcommand\Vf{\mathcal{V}}
\newcommand\Vb{\mathcal{V}_\bl}
\newcommand\Vz{\mathcal{V}_0}
\newcommand\Hom{\mathrm{Hom}}
\newcommand\Id{\mathrm{Id}}
\newcommand\sgn{\operatorname{sgn}}
\newcommand\ff{\mathrm{ff}}
\newcommand\supp{\operatorname{supp}}
\newcommand\vol{\mathrm{vol}}
\newcommand\Diff{\mathrm{Diff}}
\newcommand\Diffd{\mathrm{Diff}_{\dagger}}
\newcommand\Diffs{\mathrm{Diff}_{\sharp}}
\newcommand\Diffb{\mathrm{Diff}_\bl}
\newcommand\Diffz{\mathrm{Diff}_0}
\newcommand\Psib{\Psi_\bl}
\newcommand\Psibc{\Psi_{\mathrm{bc}}}
\newcommand\Tb{{}^{\bl} T}
\newcommand\Sb{{}^{\bl} S}
\newcommand\zT{{}^{0} T}
\newcommand\Tz{{}^{0} T}
\newcommand\zS{{}^{0} S}
\newcommand\dom{\mathcal{D}}
\newcommand\cA{\mathcal{A}}
\newcommand\cB{\mathcal{B}}
\newcommand\cE{\mathcal{E}}
\newcommand\cG{\mathcal{G}}
\newcommand\cH{\mathcal{H}}
\newcommand\cU{\mathcal{U}}
\newcommand\cO{\mathcal{O}}
\newcommand\cF{\mathcal{F}}
\newcommand\cM{\mathcal{M}}
\newcommand\cQ{\mathcal{Q}}
\newcommand\cR{\mathcal{R}}
\newcommand\cI{\mathcal{I}}
\newcommand\cL{\mathcal{L}}
\newcommand\cK{\mathcal{K}}
\newcommand\cC{\mathcal{C}}
\newcommand\Ptil{\tilde P}
\newcommand\ptil{\tilde p}
\newcommand\chit{\tilde \chi}
\newcommand\yt{\tilde y}
\newcommand\zetat{\tilde \zeta}
\newcommand\xit{\tilde \xi}
\newcommand\sigmah{\hat\sigma}
\newcommand\zetah{\hat\zeta}
\newcommand\loc{\mathrm{loc}}
\newcommand\compl{\mathrm{comp}}
\newcommand\reg{\mathrm{reg}}
\newcommand\GBB{\textsf{GBB}}
\newcommand\GBBsp{\textsf{GBB}\ }
\newcommand\bl{{\mathrm b}}
\newcommand{\sH}{\mathsf{H}}
\newcommand{\cte}{\digamma}
\newcommand\cl{\operatorname{cl}}
\newcommand\hsf{\mathcal{S}}
\newcommand\Div{\operatorname{div}}
\newcommand\hilbert{\mathfrak{X}}

\newcommand\xib{{\underline{\xi}}}
\newcommand\etab{{\underline{\eta}}}
\newcommand\zetab{{\underline{\zeta}}}

\newcommand\xibh{{\underline{\hat \xi}}}
\newcommand\etabh{{\underline{\hat \eta}}}
\newcommand\zetabh{{\underline{\hat \zeta}}}

\newcommand\psit{\tilde\psi}
\newcommand\rhot{\tilde\rho}

\setcounter{secnumdepth}{3}
\newtheorem{lemma}{Lemma}[section]
\newtheorem{prop}[lemma]{Proposition}
\newtheorem{thm}[lemma]{Theorem}
\newtheorem{cor}[lemma]{Corollary}
\newtheorem{result}[lemma]{Result}
\newtheorem*{thm*}{Theorem}
\newtheorem*{prop*}{Proposition}
\newtheorem*{cor*}{Corollary}
\newtheorem*{conj*}{Conjecture}
\numberwithin{equation}{section}
\theoremstyle{remark}
\newtheorem{rem}[lemma]{Remark}
\newtheorem*{rem*}{Remark}
\theoremstyle{definition}
\newtheorem{Def}[lemma]{Definition}
\newtheorem*{Def*}{Definition}

\newcommand{\mar}[1]{{\marginpar{\sffamily{\scriptsize #1}}}}
\newcommand\av[1]{\mar{AV:#1}}

\renewcommand{\theenumi}{\roman{enumi}}
\renewcommand{\labelenumi}{(\theenumi)}

\title[The wave equation on asymptotically Anti-de Sitter spaces]
{The wave equation on asymptotically\\
Anti-de Sitter spaces}
\author[Andras Vasy]{Andr\'as Vasy}
\date{December 24, 2010. Original version from November 28, 2009.}
\thanks{This work is partially supported by
the National Science Foundation under
grant DMS-0801226, and
a Chambers Fellowship from Stanford University.}
\address{Department of Mathematics, Stanford University, Stanford, CA
94305-2125, U.S.A.}
\email{andras@math.stanford.edu}

\begin{abstract}
In this paper we describe the behavior
of solutions of the
Klein-Gordon equation, $(\Box_g+\lambda)u=f$,
on Lorentzian manifolds $(X^\circ,g)$
which are anti-de Sitter-like (AdS-like) at infinity.
Such manifolds are Lorentzian analogues of the so-called
Riemannian conformally compact (or asymptotically hyperbolic) spaces, in
the sense that the metric is conformal to a smooth Lorentzian metric $\hat g$
on $X$, where $X$ has a non-trivial boundary, in the sense that
$g=x^{-2}\hat g$, with $x$ a boundary defining function.
The boundary is conformally time-like for these spaces,
unlike asymptotically de Sitter spaces studied in
\cite{Vasy:De-Sitter, Baskin:Parametrix}, which are similar but with
the boundary being conformally space-like.

Here we show local well-posedness for the Klein-Gordon equation, and also
global well-posedness under global assumptions on the
(null)bicharacteristic flow, for $\lambda$ below
the Breitenlohner-Freedman bound, $(n-1)^2/4$.
These have been known under additional
assumptions, \cite{Breitenlohner-Freedman:Positive,
Breitenlohner-Freedman:Stability,Holzegel:Massive}.
Further, we describe the propagation of
singularities of solutions and obtain the asymptotic behavior (at $\pa X$)
of regular solutions. We also define the scattering operator,
which in this case is an analogue of the hyperbolic Dirichlet-to-Neumann
map. Thus, it is shown that below the Breitenlohner-Freedman bound,
the Klein-Gordon equation behaves much like it would for the
conformally related metric, $\hat g$, with Dirichlet boundary conditions,
for which propagation of singularities was shown by Melrose, Sj\"ostrand
and Taylor \cite{Melrose-Sjostrand:I, Melrose-Sjostrand:II,
Taylor:Grazing, Melrose-Taylor:Kirchhoff},
though the precise form of the asymptotics is different.
\end{abstract}

\maketitle

\section{Introduction}
In this paper we consider asymptotically anti de Sitter (AdS) type
metrics on $n$-dimensional manifolds with boundary $X$, $n\geq 2$.
We recall the actual definition of AdS space below, but for our purposes
the most important feature is the asymptotic of the metric on these
spaces, so we start by making a bold general definition.
Thus, an asymptotically AdS type space is a manifold with boundary
$X$ such that
$X^\circ$ is equipped with a pseudo-Riemannian metric $g$ of signature
$(1,n-1)$ which near the boundary $Y$ of $X$ is of the form
\begin{equation}\label{eq:g-form}
g=\frac{-dx^2+h}{x^2},
\end{equation}
$h$ a smooth symmetric 2-cotensor on $X$ such that with respect to some
product decomposition of $X$ near $Y$, $X=Y\times[0,\ep)_x$,
$h|_Y$ is a section of $T^*Y\otimes T^*Y$ (rather than
merely\footnote{In fact,
even this most general setting would necessitate only minor changes,
except that the `smooth asymptotics' of Proposition~\ref{prop:asymp}
would have variable order, and
the restrictions on $\lambda$ that arise here, $\lambda<(n-1)^2/4$,
would have to be modified.}
$T^*_Y X\otimes T^*_Y X$) and is a
Lorentzian metric on $Y$ (with signature $(1,n-2)$). Note that
$Y$ is time-like with respect to the conformal metric
$$
\hat g=x^2g,\ \text{so}\ \hat g=-dx^2+h\ \text{near}\ Y,
$$
i.e.\ the dual metric $\hat G$ of $\hat g$ is negative definite
on $N^*Y$, i.e.\ on $\Span\{dx\}$, in contrast with the asymptotically 
de Sitter-like setting studied in
\cite{Vasy:De-Sitter} when the boundary is space-like. Moreover, $Y$ is
{\em not} assumed to be compact; indeed, under the assumption (TF) below,
which is useful for global well-posedness of the wave equation,
it never is.
Let the wave operator $\Box=\Box_g$ be the Laplace-Beltrami operator
associated to this metric, and let
$$
P=P(\lambda)=\Box_g+\lambda
$$
be the Klein-Gordon
operator, $\lambda\in\Cx$. The convention with the positive sign for
the `spectral parameter' $\lambda$ preserves the sign of $\lambda$
relative to the $dx^2$ component of the metric in both the Riemannian
conformally compact and the Lorentzian de Sitter-like cases, and
hence is convenient when describing the asymptotics. We remark
that if $n=2$ then up to a change of the (overall) sign of the metric,
these spaces are asymptotically de Sitter, hence the
results of \cite{Vasy:De-Sitter} apply. However, some of the
results are different even then, since in the two settings
the role of the time variable is reversed, so the formulation
of the results differs as the role of `initial' and `boundary' conditions
changes.

These asymptotically AdS-metrics are also analogues of the Riemannian
`conformally compact', or asymptotically hyperbolic, metrics,
introduced by Mazzeo and Melrose \cite{Mazzeo-Melrose:Meromorphic}
in this form, which
are of the form $x^{-2}(dx^2+h)$ with $dx^2+h$ smooth Riemannian on $X$,
and $h|_Y$ is a section of $T^*Y\otimes T^*Y$. These have been studied
extensively, in part due to the connection to AdS metrics (so some
phenomena might be expected to be similar for AdS and asymptotically
hyperbolic metrics) and their Riemannian signature, which makes the
analysis of related PDE easier. We point out that
hyperbolic space actually solves the Riemannian version of
Einstein's equations, while de Sitter and anti-de Sitter space satisfy
the actual hyperbolic Einstein equations.
We refer to
the works of Fefferman and Graham \cite{Fefferman-Graham:Conformal},
Graham and Lee \cite{Graham-Lee:Einstein}
and Anderson \cite{Anderson:Einstein} among others for analysis
on conformally compact spaces. We also refer to the works of
Witten \cite{Witten:AdS},
Graham and Witten \cite{Graham-Witten:Conformal} and
Graham and Zworski \cite{Graham-Zworski:Scattering},
and further references in these works,
for results in the Riemannian setting which are of physical relevance.
There is also
a large body of literature on asymptotically de Sitter spaces.
Among others, Anderson and Chru\'sciel
studied the geometry of asymptotically
de Sitter spaces
\cite{Anderson:Structure, Anderson:Existence,
Anderson-Chrusciel:asymptotically}, while in \cite{Vasy:De-Sitter}
the asymptotics of solutions of the Klein-Gordon equation were
obtained, and in \cite{Baskin:Parametrix} the forward fundamental
solution was constructed as a Fourier integral operator.
It should be pointed out that the de Sitter-Schwarzschild metric in fact
has many similar features with asymptotically de Sitter spaces
(in an appropriate sense, it simply has two de Sitter-like ends).
A weaker version of the asymptotics in this case
is contained in the part of works of Dafermos and
Rodnianski
\cite{Dafermos-Rodnianski:Price, Dafermos-Rodnianski:Red-shift,
Dafermos-Rodnianski:Sch-DS}
(they also
study a non-linear problem),
and local energy decay was studied by Bony and H\"afner
\cite{Bony-Haefner:Decay}, in part based on the stationary resonance
analysis of S\'a Barreto and Zworski \cite{Sa-Barreto-Zworski:Distribution};
stronger asymptotics (exponential decay to constants) was
shown in a series of papers with
Ant\^onio S\'a Barreto and Richard Melrose
\cite{Melrose-SaBarreto-Vasy:Semiclassical,Melrose-SaBarreto-Vasy:Asymptotics}.

For the universal cover of AdS space itself, the Klein-Gordon
equation was studied by Breitenlohner and Freedman
\cite{Breitenlohner-Freedman:Positive,
Breitenlohner-Freedman:Stability}, who showed its solvability for
$\lambda<(n-1)^2/4$, $n=4$, and uniqueness for $\lambda<5/4$,
in our normalization. Analogues of these results were extended to the Dirac
equation by Bachelot \cite{Bachelot:Dirac}; and on exact AdS space
there is an explicit solution due to Yagdjian and Galstian \cite{Yagdjian-Galstian:AdS}.
Finally, for a
class of perturbations of the universal cover of AdS, which still
possess a suitable Killing vector field,
Holzegel \cite{Holzegel:Massive}
recently showed well-posedness for $\lambda<(n-1)^2/4$ by
imposing a boundary condition, see \cite[Definition~3.1]{Holzegel:Massive}.
He also obtained certain estimates on the derivatives of the solution,
as well as pointwise bounds.

Below we consider solutions of $P u=0$, or indeed $Pu=f$ with $f$
given. Before describing our results, first we recall a formulation
of the conformal problem, namely $\hat g=x^2 g$, so $\hat g$ is
Lorentzian smooth on $X$, and $Y$ is time-like -- at the end of the
introduction we give a full summary of basic results
in the `compact' and `conformally compact'
Riemannian and Lorentzian settings, with
space-like as well as time-like boundaries in the latter case. Let
$$
\hat P=\Box_{\hat g};
$$
adding $\lambda$ to the operator makes no difference in this case
(unlike for $P$). Suppose that $\hsf$ is
a space-like hypersurface in $X$ intersecting $Y$
(automatically transversally). Then the Cauchy problem for the Dirichlet
boundary condition,
$$
\hat Pu=f,\ u|_{Y}=0,\ u|_{\hsf}=\psi_0,\ Vu|_{\hsf}=\psi_1,
$$
$f$, $\psi_0$, $\psi_1$ given, $V$ a vector field transversal
to $\hsf$, is locally well-posed (in appropriate
function spaces) near $\hsf$.
Moreover, under a global condition on the generalized broken bicharacteristic
(or \GBB)
flow and $\hsf$, which we recall below in Definition~\ref{def:gen-br-bich},
the equation is globally well-posed.

Namely, the global geometric assumption is
that
\begin{equation}\tag{TF}\begin{split}
&\text{there exists}\ t\in\CI(X)\ \text{such that for every
\GBB}\ \gamma,\ t\circ\rho\circ\gamma:\RR\to\RR\\
&\text{is either strictly increasing or strictly decreasing
and has range $\RR$,}
\end{split}\end{equation}
where $\rho:T^*X\to X$ is the bundle projection.
In the above formulation of the problem, we would assume
that $\hsf$ is a level set, $t=t_0$ -- note that locally
this is always true in view of the Lorentzian nature of the metric and
the conditions on $Y$ and $\hsf$. 
As is often the case in the presence
of boundaries, see e.g.\ \cite[Theorem~24.1.1]{Hor} and the subsequent remark,
it is convenient to consider the special case of the Cauchy
problem with vanishing
initial data and $f$ supported to one side of $\hsf$, say in $t\geq t_0$;
one can phrase
this as solving
$$
\hat Pu=f,\ u|_{Y}=0,\ \supp u\subset \{t\geq t_0\}.
$$
This forward
Cauchy problem is globally well-posed for
$f\in L^2_{\loc}(X)$, $u\in \dot H^1_{\loc}(X)$, and the analogous
statement also holds for the backward Cauchy problem.
Here we use H\"ormander's notation $\dot H^1(X)$, see \cite[Appendix~B]{Hor},
to avoid confusion
with the `zero Sobolev spaces' $H_0^s(X)$, which we recall momentarily.
In addition,
(without any global assumptions) singularities
of solutions, as measured by the b-wave front set, $\WFb$, relative
to either $L^2_{\loc}(X)$ or $\dot H^1_{\loc}(X)$,
propagate along \GBBsp as was shown by Melrose, Sj\"ostrand
and Taylor \cite{Melrose-Sjostrand:I, Melrose-Sjostrand:II,
Taylor:Grazing, Melrose-Taylor:Kirchhoff}, see also
\cite{Sjostrand:Propagation-I} in the analytic setting. Here recall that in
$X^\circ$, bicharacteristics are integral curves of the Hamilton
vector field $\sH_p$ (on $T^*X^\circ\setminus o$)
of the principal symbol $\hat p=\sigma_2(\hat P)$ inside
the characteristic set,
$$
\Sigma=\hat p^{-1}(\{0\}).
$$
We also recall
that the notion of a $\CI$ and an analytic \GBBsp is somewhat different
due to the behavior at diffractive points, with the analytic
definition being more permissive (i.e.\ weaker). Throughout this
paper we use the analytic definition, which we now recall.

First, we need the notion of the compressed characteristic set, $\dot\Sigma$
of $\hat P$. This can be obtained by replacing, in $T^*X$, $T^*_{Y}X$
by its quotient $T^*_YX/N^*Y$, where $N^*Y$ is the conormal bundle of
$Y$ in $X$. One denotes then by $\dot\Sigma$ the image $\hat\pi(\Sigma)$
of $\Sigma$ in this
quotient. One can give a topology to $\dot\Sigma$, making a set $O$
open if and only if $\hat\pi^{-1}(O)$ is open in $\Sigma$. This notion
of the compressed characteristic set is rather intuitive, since working
with the quotient encodes the law of reflection: points with the same
tangential but different normal momentum at $Y$ are identified, which,
when combined with the conservation of kinetic energy (i.e.\ working on the
characteristic set) gives the standard law of reflection. However, it
is very useful to introduce another (equivalent) definition already at
this point since it arises from structures which we also need.

The alternative point of view (which is what one needs in the proofs)
is that the analysis of solutions of the wave equation takes
place on the b-cotangent bundle, $\Tb^*X$  (`b' stands for boundary),
introduced by Melrose.
We refer to \cite{Melrose:Atiyah} for a very detailed description,
\cite{Vasy:Propagation-Wave} for a concise discussion.
Invariantly one can define $\Tb^*X$ as follows. First, let $\Vb(X)$
be the set of all $\CI$ vector fields on $X$ tangent to the boundary.
If $(x,y_1,\ldots,y_{n-1})$ are local coordinates on $X$,
with $x$ defining $Y$, elements of $\Vb(X)$ have the
form
\begin{equation}\label{eq:b-vf-in-coords}
a\,x\pa_x+\sum_{j=1}^{n-1} b_j \,\pa_{y_j},
\end{equation}
with $a$ and $b_j$ smooth. It follows immediately that $\Vb(X)$ is the set
of all smooth sections of a vector bundle, $\Tb X$: $x,y_j,a,b_j$, $j=1,
\ldots,n-1$, give
local coordinates in terms of \eqref{eq:b-vf-in-coords}.
Then $\Tb^*X$ is defined as the dual bundle of $\Tb X$.
Thus, points in the b-cotangent bundle, $\Tb^*X$, of $X$
are of the form
$$
\xib\,\frac{dx}{x}+\sum_{j=1}^{n-1}\zetab_j\,dy_j,
$$
so $(x,y,\xib,\zetab)$ give coordinates on $\Tb^*X$.
There is a natural map $\pi:T^*X\to\Tb^*X$ induced by the corresponding
map between sections
$$
\xi\,dx+\sum_{j=1}^{n-1} \zeta_j\, dy_j=(x\xi)\,\frac{dx}{x}
+\sum_{j=1}^{n-1} \zeta_j\,dy_j,
$$
thus
\begin{equation}\label{eq:pi-in-coords}
\pi(x,y,\xi,\zeta)=(x,y,x\xi,\zeta),
\end{equation}
i.e.\ $\xib=x\xi$, $\zetab=\zeta$. Over the interior of $X$ we can
identify $T^*_{X^\circ}X$ with $\Tb^*_{X^\circ}X$, but this identification
$\pi$ becomes singular (no longer a diffeomorphism) at $Y$.
We denote the image of
$\Sigma$ under $\pi$ by
$$
\dot\Sigma=\pi(\Sigma),
$$
called the compressed characteristic set. Thus, $\dot\Sigma$ is a subset
of the vector bundle $\Tb^*X$, hence is equipped with a topology which
is equivalent to the one define by the quotient, see
\cite[Section~5]{Vasy:Propagation-Wave}. The definition of {\em analytic}
\GBBsp
then becomes:

\begin{Def}\label{def:gen-br-bich}
{\em Generalized broken bicharacteristics,} or \GBB, are
continuous maps $\gamma:I\to\dot\Sigma$, where $I$ is an interval, satisfying
that
for all $f\in\Cinf(\Tb^*X)$ real valued,
\begin{equation*}\begin{split}
&\liminf_{s\to s_0}\frac{(f\circ\gamma)(s)-(f\circ\gamma)(s_0)}{s-s_0}\\
&\ \geq \inf\{\sH_p(\pi^* f)(q):
\ q\in\pi^{-1}(\gamma(s_0))\cap\Sigma\}.
\end{split}\end{equation*}
\end{Def}

Since
the map $p\mapsto \sH_p$ is a derivation, $\sH_{ap}=a\sH_p$ at $\Sigma$,
so bicharacteristics are merely reparameterized if $p$ is replaced
by a conformal multiple. In particular,
if $P$ is the Klein-Gordon operator, $\Box_g+\lambda$, for an
asymptotically AdS-metric $g$, the bicharacteristics over $X^\circ$
are, up to reparameterization, those of $\hat g$. We make this into
our definition of \GBB.

\begin{Def}
The compressed characteristic set $\dot\Sigma$ of $P$ is that
of $\Box_{\hat g}$.

Generalized broken bicharacteristics, or \GBB, of $P$ are \GBBsp 
{\em in the analytic sense} of
the smooth Lorentzian metric $\hat g$.
\end{Def}

We now give a formulation for the global problem. For this
purpose we need to recall one more class of differential operators
in addition to $\Vb(X)$ (which is
the set of $\CI$ vector fields {\em tangent to the boundary}).
Namely, we denote the set of $\CI$ vector fields
{\em vanishing at the boundary} by $\Vz(X)$. In local
coordinates $(x,y)$, these have the form
\begin{equation}\label{eq:z-vf-in-coords}
a\,x\pa_x+\sum_{j=1}^n b_j (x\pa_{y_j}),
\end{equation}
with $a,b_j\in\CI(X)$; cf.\ \eqref{eq:b-vf-in-coords}.
Again, $\Vz(X)$ is the set of all $\CI$ sections of a vector bundle,
$\Tz X$, which over $X^\circ$ can be naturally identified with $T_{X^\circ} X$;
we refer to \cite{Mazzeo-Melrose:Meromorphic} for a detailed discussion
of 0-geometry and analysis, and to \cite{Vasy:De-Sitter} for a summary.
We then let $\Diffb(X)$, resp.\ $\Diffz(X)$, be the set of
differential operators generated by $\Vb(X)$, resp.\ $\Vz(X)$, i.e\ they
are locally finite sums of products of these vector fields with
$\CI(X)$-coefficients. In particular,
$$
P=\Box_g+\lambda\in\Diffz^2(X),
$$
which explains the relevance of $\Diffz(X)$. This can be seen easily
from $g$ being in fact a non-degenerate smooth symmetric
bilinear form on $\Tz X$; the conformal factor $x^{-2}$ compensates
for the vanishing factors of $x$ in \eqref{eq:z-vf-in-coords},
so in fact this is {\em exactly} the same statement as $\hat g$
being Lorentzian on $T X$.

Let
$H^k_0(X)$ denote the zero-Sobolev space relative to
$$
L^2(X)=L^2_0(X)=L^2(X,dg)=L^2(X,x^{-n}d\hat g),
$$
so if $k\geq 0$ is an integer then
$$
u\in H^k_0(X)\ \text{iff for all}\ L\in\Diff^k_0(X),\ Lu\in L^2(X);
$$
negative values of $k$ give Sobolev spaces by dualization.
For our problem, we need a space of `very nice' functions corresponding
to $\Diffb(X)$. We obtain this by
replacing $\Cinf(X)$ with the space of
conormal functions to the boundary relative to a fixed space of functions,
in this case $H^k_0(X)$, i.e.\ functions
$v\in H^k_{0,\loc}(X)$
such that $Qv\in H^k_{0,\loc}(X)$ for every
$Q\in\Diffb(X)$ (of any order). The finite order regularity version of this
is $H^{k,m}_{0,\bl}(X)$, which is given for $m\geq 0$ integer by
$$
u\in H^{k,m}_{0,\bl}(X)\ \Longleftrightarrow\ 
u\in H^k_0(X)\ \text{and}\ \forall Q\in\Diffb^m(X),\ Qu\in H^k_0(X),
$$
while for $m<0$ integer, $u\in H^{k,m}_{0,\bl}(X)$ if
$u=\sum Q_j u_j$, $u_j\in H^{k,0}_{0,\bl}(X)$, $Q_j\in\Diffb^m(X)$. Thus,
$H^{-k,-m}_{0,\bl}(X)$ is the dual space of $H^{k,m}_{0,\bl}(X)$, relative
to $L^2_0(X)$. Note that in $X^\circ$, there is no distinction between $\Vb(X)$,
$\Vz(X)$, or indeed simply $\Vf(X)$ (smooth vector fields on $X$), so
over compact subsets $K$ of $X^\circ$, $H^{k,m}_{0,\bl}(X)$ is the same
as $H^{k+m}(K)$. On the other hand, at $Y=\pa X$, $H^{k,m}_{0,\bl}(X)$ distinguishes
precisely between regularity relative to $\Vz(X)$ and $\Vb(X)$.

Although the finite speed of propagation means that the wave equation
has a local character in $X$, and thus compactness of the slices
$t=t_0$ is immaterial, it is convenient to assume
\begin{equation}\tag{PT}
\text{the map}\ t:X\to\RR
\ \text{is proper.}
\end{equation}
Even as stated, the propagation of singularities results (which form
the heart of the paper) do not assume this, and the assumption
is made elsewhere merely to make the formulation and proof of the energy
estimates and existence slightly simpler, in that one does not have
to localize in spatial slices this way.

Suppose $\lambda<(n-1)^2/4$.
Suppose
\begin{equation}\label{eq:data-spaces}
f\in H^{-1,1}_{0,\bl,\loc}(X),\ \supp f\subset \{t\geq t_0\}.
\end{equation}
We want to find $u\in H^1_{0,\loc}(X)$ such that
\begin{equation}\label{eq:mixed-problem}
Pu=f,\ \supp u\subset \{t\geq t_0\}.
\end{equation}
We show that this is locally well-posed near $\hsf$.
Moreover, under the previous global assumption on \GBB, this problem
is globally well-posed:

\begin{thm}(See Theorem~\ref{thm:well-posed}.)
Assume that (TF) and (PT) hold. Suppose $\lambda<(n-1)^2/4$.
The forward Dirichlet problem, \eqref{eq:mixed-problem}, has a unique
global solution $u\in H^1_{0,\loc}(X)$, and for all compact $K\subset X$
there exists a compact $K'\subset X$ and a constant
$C>0$ such that for all $f$ as in \eqref{eq:data-spaces},
the solution $u$ satisfies
$$
\|u\|_{H^1_0(K)}\leq C\|f\|_{H^{-1,1}_{0,\bl}(K')}.
$$
\end{thm}

\begin{rem}
In fact, one can be quite explicit about $K'$ in view of (PT), since $u|_{t\in [t_0,t_1]}$
can be estimated by $f|_{t\in I}$, $I$ open containing $[t_0,t_1]$.
\end{rem}

We also prove microlocal elliptic regularity and describe the
propagation of singularities of solutions,
as measured by $\WFb$ relative to $H^1_{0,\loc}(X)$. We define this notion
in Definition~\ref{def:WFb} and discuss it there in more detail.
However, we recall the
definition of the standard wave front set $\WF$ on manifolds
without boundary $X$
that immediately generalizes to the b-wave front set $\WFb$.
Thus, one says that $q\in T^*X\setminus o$ is {\em not} in the wave front set
of a distribution $u$ if there exists $A\in\Psi^0(X)$ such
$\sigma_0(A)(q)$ is invertible and
$QAu\in L^2(X)$ for all $Q\in\Diff(X)$ -- this is equivalent to
$Au\in\Cinf(X)$ by the Sobolev embedding theorem. Here $L^2(X)$ can be replaced
by $H^m(X)$ instead, with $m$ arbitrary. Moreover, $\WF^m$ can also
be defined analogously: we require $Au\in L^2(X)$ for $A\in\Psi^m(X)$
elliptic at $q$. Thus, $q\notin\WF(u)$ means that $u$ is `microlocally
$\CI$ at $q$', while $q\notin\WF^m(u)$ means that $u$ is `microlocally
$H^m$ at $q$'.

In order to microlocalize $H^{k,m}_{0,\bl}(X)$, we need
pseudodifferential operators, here extending $\Diffb(X)$ (as that
is how we measure regularity). These are the b-pseudodifferential
operators $A\in\Psib^m(X)$ introduced by Melrose, their principal symbol
$\sigma_{\bl,m}(A)$ is a homogeneous degree $m$ function on
$\Tb^*X\setminus o$; we again refer to
\cite{Melrose:Atiyah, Vasy:Propagation-Wave}.
Then we say that $q\in\Tb^*X\setminus o$
is {\em not} in $\WFb^{k,\infty}(u)$ if there exists $A\in\Psib^0(X)$ with
$\sigma_{\bl,0}(A)(q)$ invertible and such that $Au$ is $H^k_0$-conormal
to the boundary. One also defines $\WFb^{k,m}(u)$: $q\notin\WFb^m(u)$
if there exists $A\in\Psib^m(X)$ with
$\sigma_{\bl,0}(A)(q)$ invertible and such that $Au\in H^k_{0,\loc}(X)$.
One can also extend these definitions to $m<0$.

With this definition we have the following theorem:

\begin{thm} (See Proposition~\ref{prop:elliptic} and
Theorem~\ref{thm:prop-sing}.)
Suppose that $P=\Box_g+\lambda$, $\lambda<(n-1)^2/4$,
$m\in\RR$ or $m=\infty$.
Suppose $u\in H^{1,k}_{0,\bl,\loc}(X)$ for some $k\in\RR$.
Then
$$
\WFb^{1,m}(u)\setminus\dot\Sigma\subset\WFb^{-1,m}(Pu).
$$
Moreover,
$$
(\WFb^{1,m}(u)\cap\dot\Sigma)\setminus\WFb^{-1,m+1}(Pu)
$$
is a union of maximally extended generalized broken bicharacteristics
of the conformal metric $\hat g$ in
$$
\dot\Sigma\setminus\WFb^{-1,m+1}(Pu).
$$

In particular, if $P u=0$ then $\WFb^{1,\infty}(u)\subset\dot\Sigma$
is a union of maximally extended generalized broken bicharacteristics
of $\hat g$.
\end{thm}

As a consequence of the preceding theorem, we obtain the following more general, and
precise, well-posedness result.

\begin{thm}(See Theorem~\ref{thm:well-posed-precise}.)
Assume that (TF) and (PT) hold.
Suppose that $P=\Box_g+\lambda$, $\lambda<(n-1)^2/4$,
$m\in\RR$, $m'\leq m$.
Suppose $f\in H^{-1,m+1}_{0,\bl,\loc}(X)$. Then
\eqref{eq:mixed-problem} has a unique solution in $H^{1,m'}_{0,\bl,\loc}(X)$,
which in fact lies in $H^{1,m}_{0,\bl,\loc}(X)$,
and for all compact $K\subset X$
there exists a compact $K'\subset X$ and a constant
$C>0$ such that
$$
\|u\|_{H^{1,m}_0(K)}\leq C\|f\|_{H^{-1,m+1}_{0,\bl}(K')}.
$$
\end{thm}

While we prove this result using the propagation of singularities, thus
a relatively sophisticated theorem, it could also be derived without
full microlocalization, i.e.\ without localizing the propagation of
energy in phase space.

We also generalize propagation of singularities to the case $\im\lambda\neq 0$
($\re\lambda$ arbitrary), in which case we prove one sided propagation
depending on the sign of $\im\lambda$. Namely, if $\im\lambda>0$,
resp.\ $\im\lambda<0$,
$$
(\WFb^{1,m}(u)\cap\dot\Sigma)\setminus\WFb^{-1,m+1}(Pu)
$$
is
a union of {\em maximally forward, resp.\ backward}, extended
generalized broken bicharacteristics
of the conformal metric $\hat g$. There is no difference between
the case $\im\lambda=0$ and $\re\lambda<(n-1)^2/4$, resp.\ $\im\lambda\neq 0$,
at the elliptic set,
i.e.\ the statement
$$
\WFb^{1,m}(u)\setminus\dot\Sigma\subset\WFb^{-1,m}(Pu).
$$
holds even if $\im\lambda\neq 0$. We refer to Proposition~\ref{prop:elliptic}
and Theorem~\ref{thm:prop-sing-im} for details.

These results indicate already that for $\im\lambda\neq 0$ there are many
interesting questions to answer, and in particular that one cannot
think of $\lambda$ as `small'; this will be the focus of future work.

In particular, if $f$ is conormal relative to $H^1_0(X)$
then $\WFb^{1,\infty}(u)=\emptyset$. Let
$\sqrt{}$ denote the branch square root function on $\Cx\setminus(-\infty,0]$
chosen so that takes positive values on $(0,\infty)$. The simplest conormal
functions are those in $\CI(X)$
that vanish to infinite order (i.e.\ with all derivatives) at the
boundary; the set of these is denoted by $\dCI(X)$.
If we assume
$f\in\dCI(X)$ then
$$
u=x^{s_+(\lambda)}v,\ v\in\CI(X),
\ s_+(\lambda)=\frac{n-1}{2}+\sqrt{\frac{(n-1)^2}{4}-\lambda},
$$
as we show in Proposition~\ref{prop:asymp}.
Since the indicial roots of $\Box_g+\lambda$ are
\begin{equation}\label{eq:spm-def}
s_\pm(\lambda)=\frac{n-1}{2}\pm\sqrt{\frac{(n-1)^2}{4}-\lambda},
\end{equation}
this
explains the interpretation of this problem as a `Dirichlet problem',
much like it was done in the Riemannian conformally compact case by
Mazzeo and Melrose \cite{Mazzeo-Melrose:Meromorphic}: asymptotics
corresponding to the growing
indicial root, $x^{s_-(\lambda)}v_-$, $v_-\in\CI(X)$, is ruled out.

For $\lambda<(n-1)^2/4$,
one can then easily solve the problem with inhomogeneous `Dirichlet' boundary
condition, i.e.\ given $v_0\in\CI(Y)$ and $f\in\dCI(X)$, both supported in
$\{t\geq t_0\}$,
$$
Pu=f,\ u|_{t<t_0}=0,\ u=x^{s_-(\lambda)}v_-
+x^{s_+(\lambda)}v_+,\ v_\pm\in\CI(X),\ v_-|_Y=v_0,
$$
if $s_+(\lambda)-s_-(\lambda)=2\sqrt{\frac{(n-1)^2}{4}-\lambda}$
is not an integer.
If $s_+(\lambda)-s_-(\lambda)$
is an integer, the same conclusion holds if we replace
$v_-\in\CI(X)$ by $v_-=\CI(X)
+x^{s_+(\lambda)-s_-(\lambda)}\log x\,\CI(X)$; see
Theorem~\ref{thm:inhomog-Dirichlet}.

The operator $v_-|_Y\to v_+|_Y$ is the analogue of the Dirichlet-to-Neumann
map, or the scattering operator. In the De Sitter setting
the setup is somewhat different as both pieces of
scattering data are specified either at
past or future infinity, see \cite{Vasy:De-Sitter}. Nonetheless, one
expects that the result of \cite[Section~7]{Vasy:De-Sitter},
that the scattering
operator is a Fourier integral
operator associated to the \GBBsp relation can be extended to the present
setting, at least if the boundary is totally geodesic
with respect
to the conformal metric $\hat g$, and the metric is even with respect to the boundary
in an appropriate sense.
Indeed, in an ongoing project, Baskin
and the author are extending Baskin's construction of
the forward fundamental solution on asymptotically De Sitter spaces,
\cite{Baskin:Parametrix},
to the even totally geodesic asymptotically AdS setting.
In addition, it is an interesting question what the `best' problem
to pose is when $\im\lambda\neq 0$; the results of this paper suggest
that the global problem (rather than local, Cauchy data versions)
is the best behaved.
One virtue of the parametrix construction is that we expect to be able
answer Lorentzian analogues of questions related to the work of
Mazzeo and Melrose \cite{Mazzeo-Melrose:Meromorphic}, which would
bring the Lorentzian world of AdS spaces significantly
closer (in terms of results) to the Riemannian world of conformally
compact spaces.
We singled out the totally geodesic
condition and evenness since they hold on actual AdS space, which we now discuss.

We now recall the structure of the actual AdS space to justify
our terminology. Consider $\RR^{n+1}$ with the pseudo-Riemannian
metric of signature $(2,n-1)$ given by
$$
-dz_1^2-\ldots-dz_{n-1}^2+dz_n^2+dz_{n+1}^2,
$$
with $(z_1,\ldots,z_{n+1})$ denoting coordinates on $\RR^{n+1}$,
and the hyperboloid
$$
z_1^2+\ldots+z_{n-1}^2-z_n^2-z_{n+1}^2=-1
$$
inside it.
Note that $z_n^2+z_{n+1}^2\geq 1$ on the hyperboloid, so we can
(diffeomorphically)
introduce polar coordinates in these two variables, i.e.\ we
let $(z_n,z_{n+1})=R\theta$, $R\geq 1$, $\theta\in\sphere^1$.
Then the hyperboloid is of the form
$$
z_1^2+\ldots+z_{n-1}^2-R^2=-1
$$
inside $\RR^{n-1}\times (0,\infty)_R\times\sphere^1_\theta$. As $dz_j$,
$j=1,\ldots,n-1$, $d\theta$ and $d(z_1^2+\ldots+z_{n-1}^2-R^2)$ are
linearly independent at the hyperboloid,
$$
z_1,\ldots,z_{n-1},\theta
$$
give local coordinates on it, and indeed these are global in the
sense that the hyperboloid $X^\circ$
is identified with $\RR^{n-1}\times\sphere^1$
via these. A straightforward calculation shows that the metric on
$\RR^{n+1}$ restricts to give a Lorentzian metric $g$ on the hyperboloid.
Indeed, away from $\{0\}\times\sphere^1$, we obtain a convenient
form of the metric by using polar coordinates $(r,\omega)$
in $\RR^{n-1}$, so $R^2=r^2+1$:
$$
g=-(dr)^2-r^2\, d\omega^2+(dR)^2+R^2\,d\theta^2
=-(1+r^2)^{-1}\,dr^2-r^2\,d\omega^2+(1+r^2)\,d\theta^2,
$$
where $d\omega^2$ is the standard round metric; a similar description
is easily obtained near $\{0\}\times\sphere^1$ by using the standard
Euclidean variables.

We can compactify the hyperboloid by compactifying $\RR^{n-1}$ to a ball
$\overline{\BB^{n-1}}$
via inverse polar coordinates $(x,\omega)$, $x=r^{-1}$,
$$
(z_1,\ldots,z_{n-1})=x^{-1}\omega,\ 0<x<\infty,\ \omega\in\sphere^{n-2}.
$$
Thus, the interior of $\overline{\BB^{n-1}}$ is identified with
$\RR^{n-1}$, and the boundary $\sphere^{n-2}$ of $\overline{\BB^{n-1}}$
is added at $x=0$ to compactify $\RR^{n-1}$. We let
$$
X=\overline{\BB^{n-1}}\times\sphere^1
$$
be this compactification of $X^\circ$; a collar neighborhood of $\pa X$
is identified with
$$
[0,1)_x\times \sphere^{n-2}_\omega\times\sphere^1_\theta.
$$
In this collar neighborhood the Lorentzian metric takes the form
$$
g=\frac{1}{x^2}\Big(-(1+x^2)^{-1}\,dx^2-d\omega^2+(1+x^2)\,d\theta^2\Big),
$$
which is of the desired form, and the conformal metric is
$$
\hat g=-(1+x^2)^{-1}\,dx^2-d\omega^2+(1+x^2)\,d\theta^2
$$
with respect to which the boundary, $\{x=0\}$, is indeed time-like.
Note that the induced metric on the boundary is $-d\omega^2+d\theta^2$, up
to a conformal multiple.

As already remarked, $\hat g$ has the special feature that $Y$ is
totally geodesic, unlike e.g.\ the case of $\BB^{n-1}\times\sphere^1$
equipped with a product Lorentzian metric, with $\BB^{n-1}$ carrying
the standard Euclidean metric.

For global results, it is useful to work on the universal cover
$\tilde X=
\overline{\BB^{n-1}}\times\RR_t$ of $X$, where $\RR_t$ is the universal
cover of $\sphere^1_\theta$; we use $t$ to emphasize the time-like nature
of this coordinate. The local geometry is
unchanged, but now $t$ provides a global parameter along generalized broken
bicharacteristics, and satisfies the assumptions (TF) and (PT) for our theorems.

We use this opportunity to summarize the results, already referred
to earlier, for analysis on conformally compact Riemannian
or Lorentzian spaces, including a comparison with the conformally
related problem, i.e.\ for $\Delta_{\hat g}$ or $\Box_{\hat g}$.
We assume Dirichlet boundary condition (DBC) when relevant for the
sake of definiteness, and global hyperbolicity for the hyperbolic
equations, and do not state the function spaces or optimal forms
of regularity results.

\begin{enumerate}
\item
Riemannian: $(\Delta_{\hat g}-\lambda) u=f$, with
DBC is well-posed for $\lambda\in\Cx\setminus[0,\infty)$; moreover, if $f\in\dCI(X)$, then
$u\in\CI(X)$. (Also works outside a discrete set of poles $\lambda$
in $[0,\infty)$.)
\item
Lorentzian, $\pa X=Y_+\cup Y_-$ is spacelike, $f$ supported in $t\geq t_0$,
$\lambda\in\Cx$:
$(\Box_{\hat g}-\lambda)u=f$,
$u$ supported in $t\geq t_0$, 
is well-posed.
If $f\in\dCI(X)$, the solution is $\CI$ up to $Y_\pm$.
\item
Lorentzian, $\pa X$ is timelike, $f$ supported in $t\geq t_0$,
$\lambda\in\Cx$:
$(\Box_{\hat g}-\lambda)u=f$,
with DBC at $Y$, $u$ supported in $t\geq t_0$,
is well-posed. If $f\in\dCI(X)$, the solution is $\CI$ up to $Y_\pm$.
\end{enumerate}

\begin{figure}[ht]
\begin{center}
\mbox{\epsfig{file=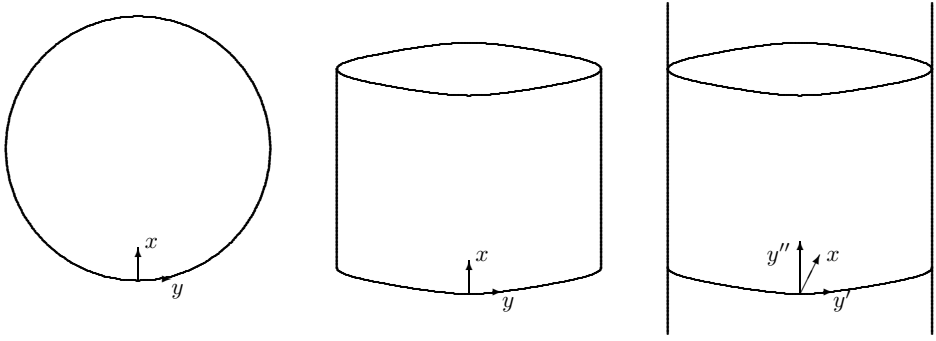, scale=0.88}}
\end{center}
\caption{On the left, a Riemannian example, $\overline{\BB^2}$, in the middle,
an example of spacelike
boundary, $[0,1]_x\times\sphere^1_y$ with $x$ timelike,
on the right, the case of timelike boundary, $\overline{\BB^2_{x,y'}}
\times\RR_{y''}$, with $y''$ timelike.}
\label{fig:conformal-spaces}
\end{figure}

We now go through the original problems. Let $s_\pm(\lambda)$ be as in
\eqref{eq:spm-def}.

\begin{enumerate}
\item
Asymptotically hyperbolic, $\lambda\in\Cx\setminus[0,+\infty)$:
There is a unique solution of
$(\Delta_{g}-\lambda) u=f$, $f\in\dCI(X)$, such that
$u=x^{s_+(\lambda)}v$, $v\in\CI(X)$. (Analogue of DBC; Mazzeo and Melrose
\cite{Mazzeo-Melrose:Meromorphic}.)
(Indeed, $u=(\Delta_g-\lambda)^{-1}f$, and this can be extended
to $\lambda\in[0,+\infty)$, apart from finitely many poles
in $[0,(n-1)^2/4]$, and analytically continued further.)
\item
Asymptotically de Sitter, $\lambda\in\Cx$:
For $f$ supported in $t\geq t_0$, there is a unique solution of
$(\Box_{g}-\lambda) u=f$ supported in $t\geq t_0$. Moreover,
for $f\in\dCI(X)$,
$u=x^{s_+(\lambda)}v_++x^{s_-(\lambda)}v_-$, $v_\pm\in\CI(X)$ and
$v_\pm|_{Y_-}$ is specified, provided that $s_+(\lambda)-s_-(\lambda)\notin
\ZZ$. (See \cite{Vasy:De-Sitter}.)
\item
Asymptotically Anti de Sitter, $\lambda\in\RR\setminus[(n-1)^2/4,+\infty)$:
For $f\in\dCI(X)$ supported in $t\geq t_0$, there is a unique solution of
$(\Box_{g}-\lambda) u=f$ such that
$u=x^{s_+(\lambda)}v$, $v\in\CI(X)$ and
$\supp u\subset\{t\geq t_0\}$.
\end{enumerate}

The structure of this paper is the following. In Section~\ref{sec:Poincare}
we prove a Poincar\'e inequality that we use to allow the
sharp range $\lambda<(n-1)^2/4$ for $\lambda$ real.
Then in Section~\ref{sec:energy}
we recall the structure of energy estimates on manifolds without
boundary as these are then adapted to our `zero geometry' in
Section~\ref{sec:Diff0-Diffb}. In Section~\ref{sec:Diff0-Psib}
we introduce microlocal tools to study operators such as $P$, namely
the zero-differential-b-pseudodifferential
calculus, $\Diffz\Psib(X)$. In Section~\ref{sec:GBB} the structure
of \GBBsp is recalled. In Section~\ref{sec:elliptic} we study
the Dirichlet form and prove microlocal elliptic regularity.
Finally, in Section~\ref{sec:prop-sing},
we prove the propagation of singularities for $P$.

I am very grateful to Dean Baskin, Rafe Mazzeo and Richard Melrose for helpful
discussions. I would also like to thank the careful referee whose comments
helped to improve the exposition significantly and also led to the removal of
some very confusing typos.

\section{Poincar\'e inequality}\label{sec:Poincare}
Let $h$ be a conformally compact Riemannian metric, i.e.\ a positive
definite inner product on $\zT X$, hence by duality on $\zT^*X$; we
denote the latter by $H$.
We denote the corresponding space of $L^2$ sections of $\zT^*X$ by
$L^2(X;\zT^*X)=L^2_0(X;\zT^*X)$.
While the inner product on $L^2(X;\zT^*X)$ depends on
the choice of $h$, the corresponding norms are independent of $h$,
at least over compact subsets $K$ of $X$. We first prove a Hardy-type
inequality:

\begin{lemma}\label{lemma:local-sharp-Poincare}
Suppose $V_0\in\Vf(X)$ is real with $V_0 x|_{x=0}=1$, and let $V\in\Vb(X)$
be given by $V=xV_0$. Given any compact subset $K$ of $X$ and
$\tilde C<\frac{n-1}{2}$, there exists $x_0>0$ such that
if $u\in\dCI(X)$ is supported in $K$ then for $\psi\in\CI(X)$
supported in $x<x_0$,
\begin{equation}\label{eq:local-sharp-Poincare}
\tilde C\|\psi u\|_{L^2_0(X)}\leq \|\psi Vu\|_{L^2_0(X)}.
\end{equation}
\end{lemma}

Recall here that $\dCI(X)$ denotes elements of $\CI(X)$ that vanish at
$Y=\pa X$ to infinite order, and the subscript $\compl$ on
$\dCI_{\compl}(X)$ below indicates that in addition the support of the function
under consideration is compact.

\begin{proof}
For any $V\in\Vb(X)$ real, and $\chi\in\CI_\compl(X)$, $u\in\dCI_{\compl}(X)$,
we have, using $V^*=-V-\Div V$,
\begin{equation*}\begin{split}
&\langle (V\chi)u,u\rangle=\langle [V,\chi]u,u\rangle
=\langle \chi u, V^* u\rangle-\langle Vu,\chi u\rangle\\
&\qquad=-\langle \chi u, Vu\rangle-\langle Vu,\chi u\rangle
-\langle \chi u,(\Div V)u\rangle.
\end{split}\end{equation*}
Now, if $V=xV_0$, $V_0\in\Vf(X)$ transversal to $\pa X$, and if
we write $dg=x^{-n} d\hat g$, $d\hat g$ a smooth non-degenerate
density then in local coordinates $z_j$ such that $d\hat g=J|dz|$,
$V_0=\sum V_0^j\pa_j$,
\begin{equation*}\begin{split}
\Div V&=x^nJ^{-1}\sum\pa_j(x^{-n}J xV_0^j)\\
&=-(n-1)\sum_j V_0^j(\pa_j x)
+xJ^{-1}\sum\pa_j(J V_0^j)=-(n-1) (V_0 x)+x\Div_{\hat g} V_0,
\end{split}\end{equation*}
where the subscript $\hat g$ in $\Div_{\hat g} V_0$ denotes that the
divergence is with respect to $\hat g$.
Thus, assuming $V_0\in\Vf(X)$ with $V_0 x|_{x=0}=1$,
$$
\Div V=-(n-1)+xa,\ a\in\CI(X).
$$
Let $x'_0>0$ be such that $V_0 x>\frac{1}{2}$ in $x\leq x'_0$.
Thus, if $0\leq\chi_0\leq 1$, $\chi_0\equiv 1$ near $0$, $\chi_0'\leq 0$,
$\chi_0$ is supported in $x\leq x'_0$, $\chi=\chi_0\circ x$, then
$$
V\chi=x(V_0x)(\chi_0'\circ x)\leq 0,
$$
hence $\langle (V\chi)u,u\rangle\leq 0$
$$
\langle \chi((n-1)+xa)u,u\rangle\leq 2\|\chi^{1/2}u\| \|\chi^{1/2}Vu\|,
$$
and thus given any $\tilde C<(n-1)/2$ there is $x_0>0$ such that
for $u$ supported in $K$,
\begin{equation*}
\tilde C\|\chi^{1/2}u\|\leq \|\chi^{1/2}Vu\|,
\end{equation*}
namely we take $x_0<x'_0/2$ such that
$(n-1)/2-\tilde C>(\sup_{K}|a|) x_0$, choose $\chi_0\equiv 1$ on $[0,x_0]$,
supported in $[0,2x_0)$. This completes the proof of the lemma.
\end{proof}

The basic Poincar\'e estimate is:

\begin{prop}\label{prop:nonlocal-Poincare-no-weight}
Suppose $K\subset X$ compact, $K\cap\pa X\neq\emptyset$,
$O$ open with $K\subset O$, $O$ arcwise connected to $\pa X$,
$K'=\overline{O}$ compact.
There exists $C>0$
such that for $u\in H^1_{0,\loc}(X)$ one has
\begin{equation}\label{eq:sharp-Poincare}
\|u\|_{L^2_0(K)}\leq C\|du\|_{L^2_0(O;\zT^*X)},
\end{equation}
where the norms are relative to the metric $h$.
\end{prop}

\begin{proof}
It suffices to prove the estimate for $u\in\dCI(X)$, for then the
proposition follows by the density of $\dCI(X)$ in $H^1_{0,\loc}(X)$ and
the continuity of both sides in the $H^1_{0,\loc}(X)$ topology.

Let $V_0,V$ be as in Lemma~\ref{lemma:local-sharp-Poincare}, and let
$\phi_0\in\CI_{\compl}(Y)$ identically $1$ on a neighborhood
of $K\cap Y$, supported in $O$, and let $x_0>0$ be as in the Lemma
with $K$ replaced by $K'$.
We pull back $\phi_0$ to a function $\phi$ defined on a neighborhood
of $Y$ by the $V_0$-flow; thus, $V_0\phi=0$.
By decreasing $x_0$ if needed, we may assume
that $\phi$ is defined and is $\CI$ in $x<x_0$, and $\supp\phi\cap\{x<x_0\}
\subset O$.
Now, let
$\psi\in\CI(X)$ identically $1$ where $x<x_0/2$, supported where
$x<3x_0/4$, and let $\psi_0\in\CI(X)$ be identically $1$
where $x<3x_0/4$, supported in $x<x_0$; thus
$\psi_0 \phi\in\CI_{\compl}(X)$.
Then, by Lemma~\ref{lemma:local-sharp-Poincare} applied
to $\psi_0\phi u$,
\begin{equation}\label{eq:sharp-Poincare-localized}
\tilde C\|\psi \phi u\|_{L^2_0(X)}=\tilde C\|\psi \psi_0\phi u\|_{L^2_0(X)}
\leq \|\psi V(\psi_0\phi u)\|_{L^2_0(X)}
=\|\psi \phi Vu\|_{L^2_0(X)}.
\end{equation}
The full proposition follows
by the standard Poincar\'e estimate and arcwise connectedness of $K$ to
$Y$ (hence to $x<x_0/2$), since one can estimate
$u|_{x>x_0/2}$ in $L^2$ in terms of $du|_{x>x_0/2}$ in $L^2$ and
$u|_{x_0/4<x<x_0/2}$.
\end{proof}

We can get a more precise estimate of the constants if we restrict
to a neighborhood of a space-like hypersurface $\hsf$; it is convenient
to state the result under our global assumptions. {\em Thus,
(TF) and (PT) are assumed to hold from here on in this section.}

\begin{prop}\label{prop:Poincare-no-weight}
Suppose $V_0\in\Vf(X)$ is real with $V_0 x|_{x=0}=1$, $V_0 t\equiv 0$ near
$Y$
and let $V\in\Vb(X)$
be given by $V=xV_0$. Let $I$ be a compact interval.
Let $C<(n-1)/2$, $\gamma>0$.
Then there exist $\ep>0$, $x_0>0$ and $C'>0$ such that the following holds.

For $t_0\in I$, $0<\delta<\ep$
and for $u\in H^1_{0,\loc}(X)$
one has
\begin{equation}\begin{split}\label{eq:sharp-const-Poincare}
&\|u\|_{L^2_0(\{p:\ t(p)\in [t_0,t_0+\ep])\}}\\
&\leq
C^{-1}\|V u\|_{L^2_0(\{p:\ t(p)\in[t_0-\delta,t_0+\ep],\ x(p)\leq x_0\})}
+\gamma\|du\|_{L^2_0(\{p:\ t(p)\in[t_0-\delta,t_0+\ep]\})}\\
&\qquad\qquad+C'\|u\|_{L^2_0(\{p:\ t(p)\in[t_0-\delta,t_0]\})},
\end{split}\end{equation}
where the norms are relative to the metric $h$.
\end{prop}

\begin{proof}
We proceed as in the proof of
Proposition~\ref{prop:nonlocal-Poincare-no-weight}, using
that the $t$-preimage of the enlargement of the interval by distance $\leq 1$
points is still compact by (PT); we always use $\ep<1$ correspondingly.
We simply let
$\phi=\tilde\phi\circ t$, where $\tilde\phi$ is the characteristic
function of $[t_0,t_0+\ep]$. Thus $V_0 \phi$ vanishes near $Y$;
at the cost of possibly decreasing $x_0$ we may assume that
it vanishes in $x<x_0$.
By \eqref{eq:sharp-Poincare-localized}, with $C=\tilde C<(n-1)/2$,
$\psi\equiv 1$
on $[0,x_0/4)$, supported in $[0,x_0/2)$,
\begin{equation}\label{eq:cutoff-Mellin}
\|\psi\phi u\|_{L^2_0(X)}\leq C^{-1}\|\psi V \phi u\|
=C^{-1}\|\psi \phi Vu\|.
\end{equation}
Thus, it remains to give a bound for
$\|(1-\psi)u\|_{L^2_0(\{p:\ t(p)\in [t_0,t_0+\ep])\}}$.

Let $\hsf$ be the space-like hypersurface in $X$ given by $t=t_0$,
$t_0\in I$.
Now let $W\in\Vb(X)$ be transversal to $\hsf$.
The standard Poincar\'e estimate (whose weighted version we prove
below in Lemma~\ref{lemma:weighted-Poincare}) obtained by integrating from
$t=t_0-\delta$ yields that for $u\in\dCI(X)$ with $u|_{t=t_0-\delta}=0$,
\begin{equation}\label{eq:standard-Poincare}
\|u\|_{L^2_0(\{p:\ t(p)\in[t_0-\delta,t_0+\ep]\})}\leq
C'(\ep+\delta)^{1/2}\|W u\|_{L^2_0(\{p:\ t(p)\in[t_0-\delta,t_0+\ep]\})},
\end{equation}
with $C'(\ep+\delta)\to 0$ as $\ep+\delta\to 0$.
Applying this with $u$ supported where $x\in(x_0/8,\infty)$
\begin{equation}\label{eq:standard-Poincare-mod}
\|u\|_{L^2_0(\{p:\ t(p)\in[t_0-\delta,t_0+\ep]\})}\leq
C''(\ep+\delta)^{1/2}\|xW u\|_{L^2_0(\{p:\ t(p)\in[t_0-\delta,t_0+\ep]\})},
\end{equation}
with $C''(\ep+\delta)\to 0$ as $\ep+\delta\to 0$.
As we want $0<\delta<\ep$, we choose $\ep>0$
such that
$$
C''(2\ep)^{1/2}<\gamma.
$$
Let $\chi\in\CI_{\compl}(\RR;[0,1])$
be identically $1$ on $[t_0,\infty)$, and be supported in
$(t_0-\delta,\infty)$. Applying \eqref{eq:standard-Poincare} to
$\chi(t) u$,
\begin{equation*}\begin{split}
&\|u\|_{L^2_0(\{p:\ t(p)\in[t_0,t_0+\ep]\})}\\
&\qquad\leq
C''(\ep+\delta)^{1/2}\|xW u\|_{L^2_0(\{p:\ t(p)\in[t_0-\delta,t_0+\ep]\})}\\
&\qquad\qquad+C''(\ep+\delta)^{1/2}\|x\chi'(t) (Wt) u\|_{L^2_0(\{p:\ t(p)\in[t_0-\delta,t_0]\})}.
\end{split}\end{equation*}
In particular, this can be applied with $u$ replaced by
$(1-\psi)u$.
This completes the proof.
\end{proof}

We also need a weighted version of this result.
We first recall a Poincar\'e inequality with weights.

\begin{lemma}\label{lemma:weighted-Poincare}
Let $C_0>0$.
Suppose that $W\in\Vb(X)$ real, $|\Div W|\leq C_0$,
$0\leq \chi\in\CI_{\compl}(X)$, and $\chi\leq -\gamma (W\chi)$ for $t\geq t_0$,
$0<\gamma<1/(2C_0)$. Then there exists $C>0$
such that for $u\in H^1_{0,\loc}(X)$ with $t\geq t_0$ on $\supp u$,
$$
\int |W\chi|\,|u|^2\,dg\leq C\gamma\int\chi |Wu|^2\,dg.
$$
\end{lemma}

\begin{proof}
We compute, using $W^*=-W-\Div W$,
\begin{equation*}\begin{split}
&\langle (W\chi)u,u\rangle=\langle [W,\chi]u,u\rangle
=\langle \chi u, W^* u\rangle-\langle Wu,\chi u\rangle\\
&\qquad=-\langle \chi u, Wu\rangle-\langle Wu,\chi u\rangle
-\langle \chi u,(\Div W)u\rangle,
\end{split}\end{equation*}
so
\begin{equation*}\begin{split}
&\int |W\chi| \,|u|^2\,dg= -\langle (W\chi)u,u\rangle
\leq 2\|\chi^{1/2}u\|_{L^2} \|\chi^{1/2} Wu\|_{L^2}
+C_0\|\chi^{1/2} u\|^2_{L^2}\\
&\leq 2\left(\int \gamma |W\chi|\,|u|^2\,dg\right)^{1/2}
\|\chi^{1/2} Wu\|_{L^2}
+C_0\int \gamma|W\chi| \,|u|^2\,dg.
\end{split}\end{equation*}
Dividing through by $(\int |W\chi|\,|u|^2\,dg)^{1/2}$ and rearranging yields
\begin{equation*}
(1-C_0\gamma) \left(\int |W\chi| \,|u|^2\,dg\right)^{1/2}\leq 2\gamma^{1/2}
\|\chi^{1/2} Wu\|_{L^2},
\end{equation*}
hence the claim follows.
\end{proof}

Our Poincar\'e inequality (which could also be named Hardy, in view of
the relationship of \eqref{eq:local-sharp-Poincare} to the Hardy
inequality) is then:

\begin{prop}\label{prop:Poincare}
Suppose $V_0\in\Vf(X)$ is real with $V_0 x|_{x=0}=1$, $V_0 t\equiv 0$ near
$Y$, and let $V\in\Vb(X)$
be given by $V=xV_0$. Let $I$ be a compact interval.
Let $C<(n-1)/2$.
Then there exist $\ep>0$, $x_0>0$,
$C'>0$, $\gamma_0>0$ such that the following holds.

Suppose $t_0\in I$, $0<\gamma<\gamma_0$.
Let $\chi_0\in\CI_{\compl}(\RR)$, $\chi=\chi_0\circ t$ and
$0\leq \chi_0\leq -\gamma\chi_0'$ on $[t_0,t_0+\ep]$, $\chi_0$ supported
in $(-\infty,t_0+\ep]$, $\delta<\ep$.
For $u\in H^1_{0,\loc}(X)$
one has
\begin{equation}\begin{split}\label{eq:sharp-const-weighted-Poincare}
&\| |\chi'|^{1/2} u\|_{L^2_0(\{p:\ t(p)\in [t_0,t_0+\ep])\}}\\
&\leq
C^{-1}\| |\chi'|^{1/2} Vu\|_{L^2_0(\{p:\ t(p)\in[t_0-\delta,t_0+\ep],
\ x(p)\leq x_0\})}\\
&\qquad\qquad
+C'\gamma\| \chi^{1/2} du\|_{L^2_0(\{p:\ t(p)\in[t_0-\delta,t_0+\ep]\})}\\
&\qquad\qquad+C'\|u\|_{L^2_0(\{p:\ t(p)\in[t_0-\delta,t_0]\})},
\end{split}\end{equation}
where the norms are relative to the metric $h$.
\end{prop}

\begin{proof}
Let $\hsf$ be the space-like hypersurface in $X$ given by $t=t_0$, $t_0\in I$.
We apply Lemma~\ref{lemma:weighted-Poincare} with $W\in\Vb(X)$ transversal
to $\hsf$ as follows.

One has from \eqref{eq:cutoff-Mellin} applied with $\phi$ replaced
by $|\chi'|^{1/2}$ that
\begin{equation*}
\|\psi |\chi'|^{1/2}u\|_{L^2_0(X)}
\leq \tilde C^{-1}\|\psi |\chi'|^{1/2} V u\|.
\end{equation*}
We now use Lemma~\ref{lemma:weighted-Poincare} with $\chi$ replaced by
$\chi\rho^2$, $\rho\equiv 1$ on $\supp(1-\psi)$,
$\rho\in\CI_{\compl}(X^\circ)$,
to estimate
$\|(1-\psi) |W\chi|^{1/2}u\|_{L^2_0(X)}$. We choose $\rho$ so that in addition
$W\rho=0$; this can be done by pulling back a function $\rho_0$ from
$\hsf$ under the $W$-flow. We may also assume that $\rho$ is supported
where $x\geq x_0/8$ in view of $x\geq x_0/4$ on $\supp (1-\psi)$ (we
might need to shorten the time interval we consider, i.e.\ $\ep>0$,
to accomplish this).
Thus, $W(\rho^2\chi)=\rho^2W\chi$, and hence
$$
\int \rho^2|W\chi|\,|u|^2\,dg\leq C\gamma\int \rho^2\chi |Wu|^2\,dg.
$$
As
$x\geq x_0/8$ on $\supp\rho$,
one can estimate $\int \chi\rho^2 |Wu|^2\,dg$ in terms of
$\int \chi |du|^2_H\,dg$ (even though $h$ is a Riemannian 0-metric!),
giving the desired result.
\end{proof}

\section{Energy estimates}\label{sec:energy}
We recall energy estimates on manifolds without boundary
in a form that will be particularly convenient in the next sections.
Thus, we work on $X^\circ$, equipped with a Lorentz metric $g$,
and dual metric $G$; let $\Box=\Box_g$ be the d'Alembertian, so
$\sigma_2(\Box)=G$. We consider a `twisted commutator' with a vector
field $V=-\imath Z$, where $Z$ is a real vector field, typically of
the form $Z=\chi W$, $\chi$ a cutoff function.
Thus, we compute
$\langle -\imath(V^*\Box-\Box V) u,u\rangle$ -- the point being that the
use of $V^*$ eliminates zeroth order terms and hence is useful when we
work not merely modulo lower order terms.

Note that $-\imath(V^*\Box-\Box V)$ is
a second order, real, self-adjoint operator, so if its principal symbol
agrees with that of $d^* C d$ for some real self-adjoint
bundle endomorphism $C$, then
in fact both operators are the same as the difference is $0$th order
and vanishes on constants. Correspondingly, there are no $0$th order
terms to estimate, which is useful as the latter tend to involve higher
derivatives of $\chi$, which in turn tend to be large relative to $d\chi$.
The principal symbol in turn is easy to calculate, for the operator is
\begin{equation}\label{eq:V-twisted-commutator}
-\imath(V^*\Box-\Box V)=-\imath(V^*-V)\Box+\imath[\Box,V],
\end{equation}
whose principal symbol is
\begin{equation*}
-\imath\sigma_0(V^*-V)G+H_G \sigma_1(V).
\end{equation*}

In fact, it is easy to perform this calculation explicitly in local
coordinates $z_j$ and dual coordinates $\zeta_j$. Let
$dg=J\,|dz|$, so $J=|\det g|^{1/2}$.
We write the components of the metric tensors as $g_{ij}$ and
$G^{ij}$, and $\pa_j=\pa_{z_j}$ when this does not cause confusion.
We also write $Z=\chi W=\sum_j Z^j\pa_j$. {\em In the remainder of
this section only,
we adopt the standard summation convention.} Then
\begin{equation*}\begin{split}
&(-\imath Z)^*=\imath Z^*=-\imath J^{-1}\pa_j J Z^j,\\
&-\Box=J^{-1}\pa_i JG^{ij}\pa_j,
\end{split}\end{equation*}
so
\begin{equation*}\begin{split}
&-\imath (V^*-V)u=-\imath ((-\imath Z)^*+\imath Z)u=(Z^*+Z)u=(-J^{-1}\pa_j J Z^j+Z^j\pa_j)u\\
&\qquad\qquad=-J^{-1}(\pa_j JZ^j)u=-(\operatorname{div} Z) u,\\
&H_G= G^{ij}\zeta_i\pa_{z_j}+G^{ij}\zeta_j\pa_{z_i}
-(\pa_{z_k}G^{ij})\zeta_i\zeta_j\pa_{\zeta_k},
\end{split}\end{equation*}
(the first two terms of $H_G$
are the same after summation, but it is convenient
to keep them separate) hence
\begin{equation*}
H_G\sigma_1(V)=G^{ij}(\pa_{z_j}Z^k)\zeta_i\zeta_k
+G^{ij}(\pa_{z_i}Z^k)\zeta_j\zeta_k
-Z^k(\pa_{z_k}G^{ij})\zeta_i\zeta_j.
\end{equation*}
Relabelling the indices, we deduce that
\begin{equation*}\begin{split}
&-\imath\sigma_0(V^*-V)G+H_G \sigma_1(V)\\
&=(-J^{-1}(\pa_k JZ^k)G^{ij}
+G^{ik}(\pa_k Z^j)+G^{jk}(\pa_k Z^i)-Z^k\pa_k G^{ij})\zeta_i\zeta_j,
\end{split}\end{equation*}
with the first and fourth terms combining into $-J^{-1}\pa_k(JZ^k G^{ij})
\zeta_i\zeta_j$, so
\begin{equation}\begin{split}\label{eq:B_ij-formula}
&-\imath(V^*\Box-\Box V)=d^*Cd,\ C_{ij}=g_{i\ell}B_{\ell j}\\
&B_{ij}=-J^{-1}\pa_k(JZ^kG^{ij})
+G^{ik}(\pa_k Z^j)+G^{jk}(\pa_k Z^i),
\end{split}\end{equation}
where $C_{ij}$ are the matrix entries of $C$ relative to the basis $\{dz_s\}$
of the fibers of the cotangent bundle.

We now want to expand $B$ using $Z=\chi W$, and separate the terms
with $\chi$ derivatives, with the idea being that we choose the
derivative of $\chi$ large enough relative to $\chi$ to
dominate the other terms. Thus,
\begin{equation}\begin{split}\label{eq:B_ij-exp-gen}
B_{ij}&=
G^{ik}(\pa_k Z^j)+G^{jk}(\pa_k Z^i)-J^{-1}\pa_k(JZ^kG^{ij})\\
&=(\pa_k \chi) (G^{ik}W^j+G^{jk}W^i-G^{ij}W^k)\\
&\qquad\qquad\qquad+\chi
(G^{ik}(\pa_k Z^j)+G^{jk}(\pa_k Z^i)-J^{-1}\pa_k(JZ^kG^{ij}))
\end{split}\end{equation}
and multiplying the first term on the right hand side by
$\pa_i u\,\overline{\pa_j u}$ (and summing over $i,j$)
gives
\begin{equation}\begin{split}\label{eq:stress-energy}
E_{W,d\chi}(du)&=
(\pa_k \chi) (G^{ik}W^j+G^{jk}W^i-G^{ij}W^k)\pa_i u\,\overline{\pa_j u}\\
&=( du,d\chi)_G \,\overline{du(W)}
+du(W)\,( d\chi,du)_G-d\chi(W) (du,du)_G,
\end{split}\end{equation}
which is twice the sesquilinear stress-energy tensor
associated to the wave $u$. This is well-known to be positive definite in
$du$, i.e.\ for covectors $\alpha$,
$E_{W,d\chi}(\alpha)\geq 0$ vanishing if and only if $\alpha=0$,
when
$W$ and $d\chi$ are both forward time-like for smooth Lorentz metrics,
see e.g.\ \cite[Section~2.7]{Taylor:Partial-I} or \cite[Lemma~24.1.2]{Hor}.
In the present setting, the metric is degenerate at the boundary,
but the analogous result still holds, as we show below.

If we replace the wave operator by the Klein-Gordon operator $P=\Box+\lambda$,
$\lambda\in\Cx$,
we obtain an additional term
\begin{equation*}\begin{split}
&-\imath\lambda (V^*-V)+2\im\lambda V=
-\imath\re\lambda (V^*-V)+\im\lambda (V+V^*)\\
&\qquad=
-\imath\re\lambda\operatorname{div}V+\im\lambda (V+V^*)
\end{split}\end{equation*}
in
$$
-\imath(V^*P-P^*V)
$$
as compared to \eqref{eq:V-twisted-commutator}. With $V=-\imath Z$, $Z=\chi W$,
as above, this contributes
$-\re\lambda (W\chi)$
in terms containing derivatives of $\chi$ to $-\imath(V^*P-P^*V)$.
In particular,
\begin{equation}\begin{split}\label{eq:interior-energy-estimate}
&\langle -\imath (V^*P-P^*V)u,u\rangle\\
&\qquad=\int E_{W,d\chi}(du)\,dg-\re\lambda \langle(W\chi) u,u\rangle\\
&\qquad\qquad+\im\lambda
(\langle \chi Wu,u\rangle+\langle u,\chi W u\rangle)
+\langle \chi R\,du,du\rangle+\langle \chi R'u,u\rangle,
\end{split}\end{equation}
$R\in\CI(X^\circ;\End(T^*X^\circ))$, $R'\in\CI(X^\circ)$.

Now suppose that $W$ and $d\chi$ are either both time like (either forward
or backward; this merely changes an overall sign).
The point of \eqref{eq:interior-energy-estimate} is that one controls
the left hand side if one controls $Pu$ (in the extreme case, when
$Pu=0$, it simply vanishes), and one can regard all terms on the right
hand side after $E_{W,d\chi}(du)$ as terms one can control
by a small multiple of the positive definite quantity
$\int E_{W,d\chi}(du)\,dg$ due to
the Poincar\'e inequality if one arranges that $\chi'$ is large relative to
$\chi$, and thus one can control $\int E_{W,d\chi}(du)\,dg$ in terms
of $Pu$.

In fact, one does not expect that $d\chi$ will be
non-degenerate time-like everywhere:
then one decomposes the energy terms into a region
$\Omega_+$ where one has the desired definiteness,
and a region $\Omega_-$ where this need not hold, and then one
can estimate $\int
E_{W,d\chit}(du)\,dg$
in $\Omega_+$
in terms of its behavior in $\Omega_-$ and $Pu$: thus one propagates
energy estimates (from $\Omega_-$ to $\Omega_+$), provided one
controls $Pu$. Of course, if $u$ is supported in $\Omega_+$, then one
automatically controls $u$ in $\Omega_-$, so we are back to the
setting that $u$ is controlled by $Pu$.
This easily gives uniqueness of solutions,
and a standard functional analytic argument by duality gives solvability.

It turns out that in the asymptotically AdS case one
can proceed similarly, except that the term $\re\lambda \langle(W\chi) u,u\rangle$
is not negligible any more at $\pa X$, and neither is $\im\lambda
(\langle \chi Wu,u\rangle+\langle u,\chi W u\rangle)$. In fact,
the $\re\lambda$ term is the `same size' as the stress energy tensor at
$\pa X$,
hence the need for an upper bound for it, while the $\im\lambda$ term
is even larger, hence the need for the assumption $\im\lambda=0$
because although $\chi$ is not differentiated (hence in some sense
`small'), $W$ is a vector field that is too large compared to the
vector fields the stress energy tensor can estimate at $\pa X$:
it is a b-vector field, rather than a 0-vector field: we explain these
concepts now.

\section{Zero-differential operators and b-differential operators}
\label{sec:Diff0-Diffb}
We start by recalling that $\Vb(X)$ is the Lie algebra of $\CI$ vector
fields on $X$ tangent to $\pa X$, while $\Vz(X)$ is the Lie algebra
of $\CI$ vector fields vanishing at $\pa X$. Thus, $\Vz(X)$ is a Lie
subalgebra of $\Vb(X)$. Note also that both $\Vz(X)$ and $\Vb(X)$
are $\CI(X)$-modules under multiplication from the left, and they act
on $x^k\CI(X)$, in the case of $\Vz(X)$ in addition
mapping $\CI(X)$ into $x\CI(X)$.
The Lie subalgebra property can be strengthened as follows.

\begin{lemma}\label{lemma:Vz-ideal}
$\Vz(X)$ is an ideal in $\Vb(X)$.
\end{lemma}

\begin{proof}
Suppose $V\in\Vz(X)$, $W\in\Vb(X)$. Then, as $V$ vanishes at $\pa X$,
there exists $V'\in\Vf(X)$ such that $V=xV'$. Thus,
$$
[V,W]=[xV',W]=[x,W]V'+x[V',W].
$$
Now, as $W$ is tangent to $Y$, $[x,W]=-Wx\in x\CI(X)$, and as $V',W\in\Vf(X)$,
$[V',W]\in\Vf(X)$, so $[V,W]\in x\Vf(X)=\Vz(X)$.
\end{proof}

As usual, $\Diffz(X)$ is the algebra generated by $\Vz(X)$, while
$\Diffb(X)$ is the algebra generated
by $\Vb(X)$. We combine these in the following
definition, originally introduced in \cite{Vasy:De-Sitter}
(indeed, even weights $x^r$ were allowed there).

\begin{Def}
Let $\Diffz^k\Diffb^m(X)$ be the (complex) vector space
of operators on $\dCI(X)$ of the
form
$$
\sum P_j Q_j,\ P_j\in\Diffz^k(X),\ Q_j\in\Diffb^k(X),
$$
where the sum is locally finite, and let
$$
\Diffz\Diffb(X)=\cup_{k=0}^\infty\cup_{m=0}^\infty\Diffz^k\Diffb^m(X).
$$
\end{Def}

We recall that
this space is closed under composition, and that commutators have
one lower order in the 0-sense than products, see
\cite[Lemma~4.5]{Vasy:De-Sitter}:

\begin{lemma}\label{lemma:Diff-filtered-ring}
$\Diffz\Diffb(X)$ is a filtered ring under composition with
$$
A\in\Diffz^k\Diffb^m(X),\ B\in\Diffz^{k'}\Diffb^{m'}(X)
\Rightarrow AB\in\Diffz^{k+k'}\Diffb^{m+m'}(X).
$$
Moreover, composition is commutative to leading order in $\Diffz$,
i.e.\ for $A,B$ as above, with $k+k'\geq 1$,
$$
[A,B]\in\Diffz^{k+k'-1}\Diffb^{m+m'}(X).
$$
\end{lemma}

Here we need an improved property regarding commutators with
$\Diffb(X)$ (which would a priori only gain in the 0-sense by the preceding
lemma). It is this lemma that necessitates the lack of weights on
the $\Diffb(X)$-commutant.

\begin{lemma}\label{lemma:Diffb-commutant}
For $A\in\Diffb^s(X)$, $B\in\Diffz^k\Diffb^m(X)$, $s\geq 1$,
$$
[A,B]\in\Diffz^k\Diffb^{s+m-1}(X).
$$
\end{lemma}

\begin{proof}
We first note that only the leading terms
in terms of $\Diffb$ order in both commutants
matter for the conclusion, for otherwise the composition result,
Lemma~\ref{lemma:Diff-filtered-ring}, gives
the desired conclusion.
We again write elements of $\Diffz\Diffb(X)$ as
locally finite sums of products
of vector fields and functions, and then, using
Lemma~\ref{lemma:Diff-filtered-ring}
and expanding the commutators, we are reduced to checking
that
\begin{enumerate}
\item
$V\in\Vz(X)$, $W\in\Vb(X)$, $[W,V]=-[V,W]\in\Diffz^1(X)$, which
follows from Lemma~\ref{lemma:Vz-ideal},
\item
and for $W\in\Vb(X)$, $f\in\CI(X)$, $[W,f]=Wf\in\CI(X)=\Diffb^0(X)$.
\end{enumerate}
In both cases thus, the commutator drops b-order by 1 as compared
to the product, completing the proof of the lemma.
\end{proof}

We also remark the following:

\begin{lemma}\label{lemma:b-to-0-conversion}
For each non-negative integer $l$ with $l\leq m$,
$$
x^l\Diffz^k\Diffb^m(X)\subset \Diffz^{k+l}\Diffb^{m-l}(X).
$$
\end{lemma}

\begin{proof}
This result is an immediate consequence of $x\Vb(X)\subset x\Vf(X)=\Vz(X)$.
\end{proof}

Integer ordered Sobolev spaces, $H^{k,m}_{0,\bl}(X)$ were defined
in the introduction. It is immediate from our definitions that for
$P\in\Diffz^r\Diffb^s(X)$,
$$
P:H^{k,m}_{0,\bl}(X)\to H^{k-r,s-m}_{0,\bl}(X)
$$
is continuous.

A particular consequence of Lemma~\ref{lemma:Diffb-commutant}
is that if $V\in\Vb(X)$, $P\in\Diffz^m(X)$,
the $[P,V]\in\Diffz^m(X)$.

We also note that for $Q\in\Vb(X)$, $Q=-\imath Z$, $Z$ real,
we have $Q^*-Q\in\CI(X)$, where
the adjoint is taken with respect to the $L^2=L^2_0(X)$ inner product.
Namely:

\begin{lemma}\label{lemma:divergence}
Suppose $Q\in\Vb(X)$, $Q=-\imath Z$, $Z$ real. Then
$Q^*-Q\in\CI(X)$, and
with
\begin{equation*}
Q=a_0(xD_x)+\sum a_j D_{y_j},
\end{equation*}
$$
Q^*-Q=\operatorname{div} Q=J^{-1} (D_x (x a_0 J)+\sum D_{y_j}(a_j J)).
$$ with the metric density given
by $J\,|dx\,dy|$, $J\in x^{-n}\CI(X)$.
\end{lemma}

Combining these results we deduce:

\begin{prop}
Suppose $Q\in\Vb(X)$, $Q=-\imath Z$, $Z$ real. Then
\begin{equation}\label{eq:twisted-commutator}
-\imath(Q^*\Box-\Box Q)=d^* C d,
\end{equation}
where $C\in\CI(X;\End(\zT^*X))$ and in the basis
$\{\frac{dx}{x},\frac{dy_1}{x},\ldots,\frac{dy_{n-1}}{x}\}$,
$$
C_{ij}=\sum_\ell g_{i\ell}\sum_k\Big(-J^{-1}\pa_k(J a_k \hat G^{\ell j})
+\hat G^{\ell k}(\pa_k a_j)+\hat G^{jk}(\pa_k a_\ell)\Big).
$$
\end{prop}

\begin{proof}
We write
$$
-\imath(Q^*\Box-\Box Q)=-\imath(Q^*-Q)\Box-\imath [Q,\Box]\in\Diffz^2(X),
$$
and
compute the principal symbol, which we check agrees with that of
$d^*C d$. One way of achieving this is to do the computation over $X^\circ$;
by continuity if the symbols agree here, they agree on $\zT^*X$. But
over the interior this is the standard computation leading
to \eqref{eq:B_ij-formula}; in coordinates
$z_j$, with dual coordinates $\zeta_j$, writing $Z=\sum Z^j\pa_{z_j}$,
$G=\sum G^{ij}\pa_{z_i}\pa_{z_j}$,
both sides have principal symbol
$$
\sum_{ij}B_{ij}\zeta_i\zeta_j,\ B_{ij}=\sum_k\Big(-J^{-1}\pa_k(J Z^k G^{ij})
+G^{ik}(\pa_k Z^j)+G^{jk}(\pa_k Z^i)\Big).
$$

Now both sides of \eqref{eq:twisted-commutator} are elements
of $\Diffz^2(X)$, are formally self-adjoint, real,
and have the same principal symbol. Thus, their difference is
a first order, self-adjoint and real operator; it follows that its
principal symbol vanishes, so in fact this difference is zeroth order.
Since it annihilates constants (as both sides do), it actually vanishes.
\end{proof}

We particularly care about the terms in which the coefficients $a_j$ are
differentiated, with the idea being that we write $Z=\chi W$, and choose the
derivative of $\chi$ large enough relative to $\chi$ to
dominate the other terms.
Thus, as in \eqref{eq:stress-energy},
\begin{equation}\begin{split}\label{eq:B_ij-exp}
B_{ij}&=
\sum_k(\pa_k \chi) (G^{ik}W^j+G^{jk}W^i-G^{ij}W^k)\\
&\qquad\qquad\qquad+\chi
\sum_k(G^{ik}(\pa_k Z^j)+G^{jk}(\pa_k Z^i)-J^{-1}\pa_k(JZ^kG^{ij}))
\end{split}\end{equation}
and multiplying the first term on the right hand side by
$\pa_i u\,\overline{\pa_j u}$ (and summing over $i,j$)
gives
\begin{equation*}
\sum_{i,j,k}
(\pa_k \chi) (G^{ik}W^j+G^{jk}W^i-G^{ij}W^k)\pa_i u\,\overline{\pa_j u},
\end{equation*}
which is twice the sesquilinear stress-energy tensor
$\frac{1}{2}E_{W,d\chi}(du)$
associated to the wave $u$. As we mentioned before, this is
positive definite when
$W$ and $d\chi$ are both forward time-like for smooth Lorentz metrics.
In the present setting, the metric is degenerate at the boundary,
but the analogous result still holds since
\begin{equation}\begin{split}\label{eq:stress-energy-deg}
E_{W,d\chi}(du)&=
\sum_{i,j,k}(\pa_k \chi) (\hat G^{ik}W^j+\hat G^{jk}W^i
-\hat G^{ij}W^k)(x\pa_i u)\,\overline{x\pa_j u}\\
&=( x\,du,d\chi)_{\hat G} \,\overline{x\,du(W)}
+x\,du(W)\,( d\chi,x\,du)_{\hat G}-d\chi(W) (x\,du,x\,du)_{\hat G},
\end{split}\end{equation}
so the Lorentzian non-degenerate nature of $\hat G$ proves the (uniform)
positive
definiteness in $x\,du$, considered as an element of $T^*_q X$,
hence in $du$, regarded as an element of $\Tz^*_q X$.
Indeed, we recall the quick proof here since we need to improve on
this statement to get an optimal result below.

Thus, we wish to show that for $\alpha\in T^*_qX$, $W\in T_q X$,
$\alpha$ and $W$ forward time-like,
$$
\hat E_{W,\alpha}(\beta)=
(\beta,\alpha)_{\hat G} \,\overline{\beta(W)}
+\beta(W)\,( \alpha,\beta)_{\hat G}-\alpha(W) (\beta,\beta)_{\hat G}
$$
is positive definite as a quadratic form in $\beta$. Since replacing
$W$ by a positive multiple does not change the positive definiteness,
we may assume, as we do below, that $(W,W)_{\hat G}=1$. Then we may
choose local coordinates $(z_1,\ldots,z_n)$ such that $W=\pa_{z_n}$ and
$\hat g|_q=dz_n^2-(dz_1^2+\ldots+dz_{n-1}^2)$, thus
$\hat G|_q=\pa_{z_n}^2-(\pa_{z_1}^2+\ldots+\pa_{z_{n-1}}^2)$. Then
$\alpha=\sum\alpha_j\,dz_j$ being forward time-like means that $\alpha_n>0$
and $\alpha_n^2>\alpha_1^2+\ldots+\alpha_{n-1}^2$. Thus,
\begin{equation}\begin{split}\label{eq:stress-energy-pos-def}
&\hat E_{W,\alpha}(\beta)=(\beta_n\alpha_n
-\sum_{j=1}^{n-1}\beta_j\alpha_j)\overline{\beta_n}
+\beta_n(\alpha_n\overline{\beta_n}-\sum_{j=1}^{n-1}\alpha_j\overline{\beta_j})
-\alpha_n(|\beta_n|^2-\sum_{j=1}^{n-1}|\beta_j|^2)\\
&\qquad=\alpha_n\sum_{j=1}^n|\beta_j|^2
-\beta_n\sum_{j=1}^{n-1}\alpha_j\overline{\beta_j}
-\sum_{j=1}^{n-1}\beta_j\alpha_j\overline{\beta_n}\\
&\qquad
\geq \alpha_n\sum_{j=1}^n|\beta_j|^2
-2|\beta_n|(\sum_{j=1}^{n-1}\alpha_j^2)^{1/2}(\sum_{j=1}^{n-1}|\beta_j|^2)^{1/2}\\
&\qquad\geq \alpha_n\sum_{j=1}^n|\beta_j|^2
-2|\beta_n|\alpha_n(\sum_{j=1}^{n-1}|\beta_j|^2)^{1/2}
=\alpha_n\Big(|\beta_n|-(\sum_{j=1}^{n-1}|\beta_j|^2)^{1/2}\Big)^2\geq 0,
\end{split}\end{equation}
with the last inequality strict if
$|\beta_n|\neq (\sum_{j=1}^{n-1}|\beta_j|^2)^{1/2}$, and the
preceding one (by the strict forward time-like character of
$\alpha$) strict if $\beta_n\neq 0$ and $\sum_{j=1}^{n-1}|\beta_j|^2\neq 0$.
It is then immediate that at least one of these inequalities is
strict unless $\beta=0$, which is the claimed positive definiteness.

We claim that we can make a stronger statement
if $U\in T_q X$ and $\alpha(U)=0$
and $(U,W)_{\hat g}=0$ (thus $U$ is necessarily space-like,
i.e.\ $(U,U)_{\hat g}<0$):
$$
\hat E_{W,\alpha}(\beta)+c\,\frac{\alpha(W)}{(U,U)_{\hat g}} |\beta(U)|^2,
\ c<1,
$$
is positive definite in $\beta$. Indeed, in this case (again assuming
$(W,W)_{\hat g}=1$) we can choose coordinates
as above such that $W=\pa_{z_n}$, $U$ is a multiple of $\pa_{z_1}$,
namely $U=(-(U,U)_{\hat g})^{1/2}\pa_{z_1}$, $\hat g|_q=dz_n^2
-(dz_1^2+\ldots+dz_{n-1}^2)$. To achieve this, we complete
$e_n=W$ and $e_1=(-(U,U)_{\hat g})^{-1/2} U$ (which are orthogonal
by assumption)
to a $\hat g$ normalized orthogonal basis $(e_1,e_2,\ldots,e_n)$
of $T_q X$, and then
choose coordinates such that the coordinate vector fields are given by the
$e_j$ at $q$. Then $\alpha$ forward time-like means
that $\alpha_n>0$
and $\alpha_n^2>\alpha_1^2+\ldots+\alpha_{n-1}^2$, and
$\alpha(U)=0$ means that $\alpha_1=0$. Thus, with $c<1$,
\begin{equation*}\begin{split}
&\hat E_{W,\alpha}(\beta)+c\,\frac{\alpha(W)}{(U,U)_{\hat g}} |\beta(U)|^2\\
&=(\beta_n\alpha_n
-\sum_{j=2}^{n-1}\beta_j\alpha_j)\overline{\beta_n}
+\beta_n(\alpha_n\overline{\beta_n}
-\sum_{j=2}^{n-1}\alpha_j\overline{\beta_j})\\
&\qquad-\alpha_n(|\beta_n|^2-\sum_{j=1}^{n-1}|\beta_j|^2)
-c\alpha_n|\beta_1|^2\\
&\geq (1-c)\alpha_n|\beta_1|^2\\
&\qquad+\Big((\beta_n\alpha_n
-\sum_{j=2}^{n-1}\beta_j\alpha_j)\overline{\beta_n}
+\beta_n(\alpha_n\overline{\beta_n}
-\sum_{j=2}^{n-1}\alpha_j\overline{\beta_j})\\
&\qquad\qquad-\alpha_n(|\beta_n|^2-\sum_{j=2}^{n-1}|\beta_j|^2)\Big).
\end{split}\end{equation*}
On the right hand side the term in the large parentheses is the same
kind of expression as in \eqref{eq:stress-energy-pos-def}, with the
terms with $j=1$ dropped, thus is positive definite in
$(\beta_2,\ldots,\beta_n)$, and for $c<1$, the first term is positive
definite in $\beta_1$, so the left hand side is indeed positive definite
as claimed. Rewriting this in terms of $G$ in our setting,
we obtain that for $c<1$
$$
E_{W,d\chi}(du)-c (W\chi)|xU u|^2
$$
is positive definite in $du$, considered an element of $\Tz^*_qX$,
when $q\in\pa X$, and hence is positive definite sufficiently
close to $\pa X$.

Stating the result as a lemma:

\begin{lemma}\label{lemma:energy-normal-deriv}
Suppose $q\in \pa X$, $U,W\in T_q X$, $\alpha\in T_q^*X$ and $\alpha(U)=0$
and $(U,W)_{\hat g}=0$. Then
$$
E_{W,\alpha}(\beta)+c\,\frac{\alpha(W)}{(U,U)_{\hat g}} |\beta(xU)|^2,
\ c<1,
$$
is positive definite in $\beta\in\zT^*_q X$.
\end{lemma}

At this point we modify the choice of our time function $t$ so that
we can construct $U$ and $W$ satisfying the requirements of the lemma.

\begin{lemma}
Assume (TF) and (PT).
Given $\delta_0>0$ and a compact interval $I$
there exists a function $\tau\in\CI(X)$ such that
$|t-\tau|<\delta_0$ for $t\in I$, $d\tau$ is time-like in the same component
of the time-like cone as $dt$, and $\hat G(d\tau,dx)=0$ at $x=0$.
\end{lemma}

\begin{proof}
Let $\chi\in\CI_{\compl}([0,\infty))$, identically $1$ near $0$,
$0\leq \chi\leq 1$, $\chi'\leq 0$, supported
in $[0,1]$, and for $\ep,\delta>0$
to be specified let
$$
\tau=t- x\chi\left(\frac{x^\delta}{\ep}\right)
\frac{\hat G(dt,dx)}{\hat G(dx,dx)}.
$$
Note that on the support of $\chi\left(\frac{x^\delta}{\ep}\right)$,
$x\leq \ep^{1/\delta}$, so if $\ep^{1/\delta}$ is sufficiently small,
$\hat G(dx,dx)<0$, and bounded away from $0$, there in view of (PT)
and as $\hat G(dx,dx)<0$ at $Y$.

At $x=0$
$$
d\tau=dt-\frac{\hat G(dt,dx)}{\hat G(dx,dx)}\,dx,
$$
so $\hat G(d\tau,dx)=0$.
As already noted, on the support of $\chi\left(\frac{x^\delta}{\ep}\right)$,
$x\leq \ep^{1/\delta}$, so for $t\in I$, $I$ compact, in view of (PT),
\begin{equation}\label{eq:t-tau-diff}
|\tau-t|\leq C \ep^{1/\delta},
\end{equation}
with $C$ independent of $\ep,\delta$.
Next,
$$
d\tau=dt-\alpha \gamma\,dx-\tilde\alpha\gamma\,dx-\beta\mu,
$$
where
\begin{equation*}\begin{split}
&\alpha=\chi\left(\frac{x^\delta}{\ep}\right),
\ \gamma=\frac{\hat G(dt,dx)}{\hat G(dx,dx)},
\ \tilde\alpha=\delta \frac{x^\delta}{\ep}\chi'\left(\frac{x^\delta}{\ep}\right),\\
&\beta=x\chi\left(\frac{x^\delta}{\ep}\right),
\ \mu=d\left(\frac{\hat G(dt,dx)}{\hat G(dx,dx)}\right).
\end{split}\end{equation*}
Now,
\begin{equation*}\begin{split}
\hat G(dt-\alpha\gamma dx,dt-\alpha\gamma dx)
&=\hat G(dt,dt)-2\alpha\gamma \hat G(dt,dx)+\alpha^2\gamma^2\hat G(dx,dx)\\
&=\hat G(dt,dt)-(2\alpha-\alpha^2)\frac{\hat G(dt,dx)^2}{\hat G(dx,dx)},
\end{split}\end{equation*}
which is $\geq \hat G(dt,dt)$ if $2\alpha-\alpha^2\geq 0$,
i.e.\ $\alpha\in[0,2]$. But $0\leq\alpha\leq 1$, so
$$
\hat G(dt-\alpha\gamma dx,dt-\alpha\gamma dx)\geq \hat G(dt,dt)>0
$$
indeed,
i.e.\ $dt-\alpha\gamma dx$ is timelike. Since $dt-\rho\alpha\gamma\, dx$
is still time-like for $0\leq \rho\leq 1$, $dt-\alpha\gamma dx$ is
in the same component of time-like covectors as $dt$, i.e.\ is
forward oriented.
Next, observe that with $C'=\sup s|\chi'(s)|$,
$$
|\tilde\alpha|\leq C'\delta,\ |\beta|\leq \ep^{1/\delta},
$$
so over compact sets
$\tilde\alpha\gamma\,dx+\beta\mu$ can be made arbitrarily small
by first choosing $\delta>0$ sufficiently small and then $\ep>0$
sufficiently small. Thus, $\hat G(d\tau,d\tau)$ is forward
time-like as well. Reducing $\ep>0$ further if
needed, \eqref{eq:t-tau-diff} completes the proof.
\end{proof}

This lemma can easily be made global.

\begin{lemma}\label{lemma:t-global-construction}
Assume (TF) and (PT).
Given $\delta_0>0$
there exists a function $\tau\in\CI(X)$ such that
$|t-\tau|<\delta_0$ for $t\in \RR$, $d\tau$ is time-like in the same component
of the time-like cone as $dt$, and $\hat G(d\tau,dx)=0$ at $x=0$.

In particular, $\tau$ also satisfies (TF) and (PT).
\end{lemma}

\begin{proof}
We proceed as above, but let
$$
\tau=t- x\chi\left(\frac{x^{\delta(t)}}{\ep(t)}\right)
\frac{\hat G(dt,dx)}{\hat G(dx,dx)}.
$$
We then have two additional terms,
$$
-x^{1-\delta(t)}\delta'(t)\log x\frac{x^{\delta(t)}}{\ep(t)}
\chi'\left(\frac{x^\delta(t)}{\ep(t)}\right)
\frac{\hat G(dt,dx)}{\hat G(dx,dx)}\,dt,
$$
and
$$
x\frac{\ep'(t)}{\ep(t)}
\frac{x^{\delta(t)}}{\ep(t)}\chi'\left(\frac{x^{\delta(t)}}{\ep(t)}\right)
\frac{\hat G(dt,dx)}{\hat G(dx,dx)}\,dt,
$$
in $d\tau$. Note that on the support of both terms
$x\leq \ep(t)^{1/\delta(t)}$, while
$\frac{x^{\delta(t)}}{\ep(t)}\chi'\left(\frac{x^{\delta(t)}}{\ep(t)}\right)$ is
uniformly bounded. Thus, if $\delta(t)<1/3$, $|\delta'(t)|\leq 1$,
$|\ep'(t)|\leq 1$,
the factors in front of $dt$ in both terms is bounded in absolute value
by $C\ep(t)\frac{\hat G(dt,dx)}{\hat G(dx,dx)}$.
Now for any $k$ there are $\delta_k,\ep_k>0$, which we may assume are in
$(0,1/3)$ and are decreasing with $k$,
such that on $I=[-k,k]$,
$\tau$ so defined, satisfies all the requirements if $0<\ep(t)<\ep_k$,
$0<\delta(t)<\delta_k$
on $I$ and $|\ep'(t)|\leq 1$, $|\delta'(t)|\leq 1$. But now in view
of the bounds on $\ep_k$ and $\delta_k$ it is straightforward
to write down $\ep(t)$ and $\delta(t)$ with the desired properties, e.g.\ by
approximating the piecewise linear function which takes the
value $\ep_k$ at $\pm(k-1)$, $k\geq 2$,
to get $\ep(t)$, and similarly with $\delta$, finishing
the proof.
\end{proof}

{\em From this point on, within this section,
we assume that (TF) and (PT) hold.
From now on we simply replace $t$ by $\tau$.} We let
$W=\hat G(dt,.)$, $U_0=\hat G(dx,.)$. Thus, {\em at} $x=0$,
$$
dt(U_0)=\hat G(dx,dt)=0,\ (U_0,W)_{\hat g}=\hat G(dx,dt)=0.
$$
We extend $U_0|_Y$ to a vector field $U$ such that $Ut=0$, i.e.\ $U$ is
tangent to the level surfaces of $t$.
Then we have on all of $X$,
\begin{equation}\label{eq:W-time-like}
W(dt)=\hat G(dt,dt)>0,
\end{equation}
and
\begin{equation}\label{eq:U-space-like}
U(dx)=\hat G(dx,dx)<0
\end{equation}
on a neighborhood of $Y$,
with uniform upper and lower bounds (bounding away from $0$)
for both \eqref{eq:W-time-like} and \eqref{eq:U-space-like} on
compact subsets of $X$.

Using Lemma~\ref{lemma:energy-normal-deriv} and the equations just above,
we thus deduce that for $\chi=\tilde\chi\circ t$, $c<1$,
$\rho\in\CI(X)$, identically $1$ near $Y$, supported sufficiently
close to $Y$, $Q=-\imath Z$, $Z=\chi W$,
\begin{equation}\begin{split}\label{eq:energy-estimate}
&\langle -\imath(Q^*P-P^* Q)u,u\rangle\\
&\qquad=\int E_{W,d\chi}(du)\,dg
-\re\lambda \langle (W\chi) u,u\rangle\\
&\qquad\qquad\qquad+
\im\lambda(\langle \chi Wu,u\rangle+\langle u,\chi Wu\rangle)
+\langle
\chi R du,du\rangle+\langle \chi R'u,u\rangle\\
&\qquad
=\langle (\chi' A+\chi R)du,du\rangle+\langle c\rho(W\chi)xU u,xU u\rangle
-\re\lambda\langle (W\chi) u,u\rangle\\
&\qquad\qquad\qquad
+\im\lambda(\langle \chi Wu,u\rangle+\langle u,\chi Wu\rangle)
+\langle \chi R'u,u\rangle
\end{split}\end{equation}
with $A,R\in\CI(X;\End(\zT^*X))$, $R'\in\CI(X)$
and $A$ is positive definite, all independent of $\chi$.
Here $\rho$ is used since $E_{W,d\chi}(du)-c (W\chi)|xU u|^2$
is only positive definite near $Y$.

Fix $t_0<t_0+\ep<t_1$. Let
$\chi_0(s)=e^{-1/s}$ for $s>0$, $\chi_0(s)=0$ for $s<0$, $\chi_1\in\CI(\RR)$
identically 1 on $[1,\infty)$, vanishing on $(-\infty,0]$,
Thus,
$s^2\chi_0'(s)=\chi_0(s)$ for $s\in\RR$.
Now
consider
$$
\tilde\chi(s)=\chi_0(-\digamma^{-1}(s-t_1))\chi_1((s-t_0)/\ep),
$$
so
$$
\supp\tilde\chi\subset [t_0,t_1]
$$
and
$$
s\in [t_0+\ep,t_1]\Rightarrow \tilde\chi'=-\digamma^{-1}
\chi_0'(-\digamma^{-1}(s-t_1)),
$$
so
$$
s\in [t_0+\ep,t_1]\Rightarrow \tilde\chi=-\digamma^{-1} (s-t_1)^2\tilde\chi',
$$
so for $\digamma>0$ sufficiently large, this is bounded by a small
multiple of $\tilde\chi'$, namely
\begin{equation}\label{eq:chip-gamma-est}
s\in [t_0+\ep,t_1]\Rightarrow \tilde\chi=-\gamma\tilde\chi',
\ \gamma=(t_1-t_0)^2\digamma^{-1}.
\end{equation}
In particular, for sufficiently large
$\digamma$,
$$
-(\chi' A+\chi R)\geq -\chi' A/2
$$
on $[t_0+\ep,t_1]$.
In addition, by \eqref{eq:sharp-const-weighted-Poincare}
and \eqref{eq:chip-gamma-est}, for
$\re\lambda<(n-1)^2/4$, and $c'>0$ sufficiently close to $1$
\begin{equation*}\begin{split}
-\langle \re\lambda (W\chi) u,u\rangle \leq c'
\langle \rho (-W\chi)xU u,xU u\rangle+C'\digamma^{-1}\| \chi^{1/2}du\|^2
\end{split}\end{equation*}
while
$$
|\langle\chi R'u,u\rangle|\leq C'\|\chi^{1/2}u\|^2
$$
and
\begin{equation}\begin{split}\label{eq:chi-u-est}
&\|\chi^{1/2}u\|^2\leq C'\digamma^{-1}\langle (-W\chi)u,u\rangle\\
&\qquad\leq
C''\digamma^{-1}
\langle (-W\chi)xU u,xU u\rangle+C''\digamma^{-2}
\| \chi^{1/2}du\|^2.
\end{split}\end{equation}
However,
$\im\lambda(\langle \chi Wu,u\rangle+\langle u,\chi Wu\rangle)$
is too large to be controlled by the stress energy tensor since
$W$ is a b-vector field, but not a 0-vector field. Thus, in order to
control the $\im\lambda$ term for $t\in[t_0+\ep,t_1]$,
we need to assume that $\im\lambda=0$.
Then, writing $Qu=Q^*u +(Q-Q^*)u$, and choosing $\digamma>0$ sufficiently
large to absorb the first term on the right hand side of
\eqref{eq:chi-u-est},
\begin{equation}\begin{split}
&\langle-\chi' A du,du\rangle/2\leq
-\langle -\imath P u,Qu\rangle+\langle \imath P u, Qu\rangle
+\gamma\langle(-\chi') d u,d u\rangle\\
&\leq 2C\|\chi^{1/2}W P u\|_{H^{-1}_0(X)} \,\|\chi^{1/2} u\|_{H^1_0(X)}+
2C\|(-\chi')^{1/2} P u\|_{L^2_0(X)} \,\|(-\chi')^{1/2} u\|_{L^2_0(X)}\\
&\qquad\qquad+
C\gamma \|(-\chi')^{1/2}du\|^2\\
&\leq 2C\delta^{-1}(\|W P u\|^2_{H^{-1}_0(X)}
+\|P u\|^2_{L^2_0(X)})+2C\delta
(\|\chi^{1/2} u\|^2_{H^1_0(X)}+\|(-\chi')^{1/2}u\|_{L^2(X)}^2)\\
&\qquad\qquad+C\digamma^{-1} \|(-\chi')^{1/2}du\|^2.
\end{split}\end{equation}
For sufficiently small $\delta>0$ and sufficiently large $\digamma>0$
we absorb all but the first parenthesized term on the right
hand side into the left hand
side by the positive definiteness of $A$ and the Poincar\'e inequality,
Proposition~\ref{prop:Poincare}, to conclude that
for $u$ supported in $[t_0+\ep,t_1]$,
\begin{equation}\label{eq:Pu-est}
\|(-\chi')^{1/2}du\|_{L^2_0(X;\Tz^*X)}\leq C\|P u\|_{H^{-1,1}_{0,\bl}(X)}.
\end{equation}
In view of the Poincar\'e inequality we conclude:

\begin{lemma}\label{lemma:energy-est}
Suppose $\lambda<(n-1)^2/4$,
$t_0<t_0+\ep<t_1$, $\chi$ as above.
For $u\in\dCI(X)$ supported in $[t_0+\ep,t_1]$
one has
\begin{equation}\label{eq:Pu-est-2}
\|(-\chi')^{1/2}u\|_{H^1_0(X)}\leq C\|P u\|_{H^{-1,1}_{0,\bl}(X)}.
\end{equation}
\end{lemma}

\begin{rem}\label{rem:uniform-length}
Note that if $I$ is compact then there is $T>0$
such that for $t_0\in I$ we can take any $t_1\in (t_0,t_0+T]$, i.e.\ the
time interval over which we can make the estimate is uniform over such
compact intervals $I$.
\end{rem}

This lemma gives local in time uniqueness immediately,
hence iterative
application of the lemma, together with Remark~\ref{rem:uniform-length},
yields:

\begin{cor}\label{cor:unique}
Suppose $\lambda<(n-1)^2/4$.
For $f\in H^{-1,1}_{0,\bl,\loc}(X)$ supported in $t>t_0$, there is at most one
$u\in H^{1}_{0,\loc}(X)$ such that
$\supp u\subset \{p: t(p)\geq t_0\}$ and $Pu=f$.
\end{cor}

Via the standard functional analytic argument, we deduce from
\eqref{eq:Pu-est}:

\begin{lemma}\label{lemma:weak-local-exist}
Suppose $\lambda<(n-1)^2/4$, $I$ a compact interval.
There is $\sigma>0$ such that for $t_0\in I$, and
for $f\in H^{-1}_{0,\loc}(X)$ supported in $t>t_0$, there exists
$u\in H^{1,-1}_{0,\bl,\loc}(X)$,
$\supp u\subset \{p: t(p)\geq t_0\}$ and $Pu=f$ in $t<t_0+\sigma$.
\end{lemma}

\begin{proof}
For any subspace
$\hilbert$ of $\dist(X)$ let
$\hilbert|_{[\tau_0,\tau_1]}$ consist of elements of $\hilbert$
restricted
to $t\in[\tau_0,\tau_1]$,
$\hilbert^\bullet_{[\tau_0,\tau_1]}$ consist of elements of $\hilbert$
supported
in $t\in[\tau_0,\tau_1]$. In particular, an element of
$\dCI_{\compl}(X)^\bullet_{[\tau_0,\tau_1]}$ vanishes to infinite
order at $t=\tau_0,\tau_1$. Thus, the dot over $\CI$ denotes the infinite
order vanishing at $\pa X$, while the $\bullet$ denotes the infinite
order vanishing at the time boundaries we artificially imposed.

We assume that $f$ is supported in $t>t_0+\delta_0$.
We use Lemma~\ref{lemma:energy-est}, with
the role of $t_0$ and $t_1$ reversed (backward in time propagation),
and our requirement on $\sigma$ is that it is sufficiently small
so that the backward version of the lemma is valid with $t_1=t_0+2\sigma$.
(This can be done uniformly over $I$ by Remark~\ref{rem:uniform-length}.)
Let
$T_1=t_1-\ep$ and $t_1$ be such that $t_0+\sigma=T_1'<T_1<t_1<t_0+2\sigma$.
Applying the estimate \eqref{eq:Pu-est}, using $P=P^*$, with
$u$ replaced by $\phi\in\dCI_{\compl}(X)^\bullet_{[t_0,T_1]}$
with $t_1$ in the role of $t_0$ there (backward estimate), $\tau_0\in[t_0,T_1)$
in the role of $t_0$, we obtain:
\begin{equation}\label{eq:Pphi-est}
\|(\chi')^{1/2}\phi\|_{H^1_0(X)|_{[\tau_0,T_1]}}
\leq C\|P^*\phi\|_{H^{-1,1}_{0,\bl}(X)|_{[\tau_0,T_1]}},
\ \phi\in\dCI_{\compl}(X)^\bullet_{[\tau_0,T_1]}.
\end{equation}
It is also useful to rephrase this as
\begin{equation}\label{eq:Pphi-est-2}
\|\phi\|_{H^1_0(X)|_{[\tau_0',T_1]}}
\leq C\|P^*\phi\|_{H^{-1,1}_{0,\bl}(X)|_{[\tau_0,T_1]}},
\ \phi\in\dCI_{\compl}(X)^\bullet_{[\tau_0,T_1]},
\end{equation}
when $\tau_0'>\tau_0$.
By \eqref{eq:Pphi-est},
$P^*:\dCI_\compl(X)^\bullet_{[t_0,T_1]}\to\dCI_\compl(X)^\bullet_{[t_0,T_1]}$
is injective.
Define
$$
(P^*)^{-1}:\Ran_{\dCI_\compl(X)^\bullet_{[t_0,T_1]}}P^*\to
\dCI_{\compl}(X)^\bullet_{[t_0,T_1]}
$$
by $(P^*)^{-1}\psi$ being the unique $\phi\in \dCI_{\compl}(X)^\bullet
_{[t_0,T_1]}$
such
that $P^*\phi=\psi$.
Now consider the conjugate
linear functional on $\Ran_{\dCI_\compl(X)^\bullet_{[t_0,T_1]}} P^*$
given by
\begin{equation}\label{eq:dual-map-def}
\ell:\psi\mapsto \langle f,(P^*)^{-1}\psi\rangle.
\end{equation}
In view of
\eqref{eq:Pphi-est}, and the support condition on $f$ (namely the
support is in $t>t_0+\delta_0$) and $\psi$ (the support is
in $t\leq T_1$)\footnote{We use below that we can thus regard
$f$ as an element of $H^{-1}_0(X)^\bullet_{[t_0+\delta_0,\infty)}$,
while $(P^*)^{-1}\psi$ as an element of $H^1_0(X)^\bullet_{(-\infty,T_1]}$,
so these can be naturally paired, with the pairing bounded
in the appropriate norms. We then write these norms
as $H^{-1}_0(X)|_{[t_0+\delta_0,T_1]}$ and $H^1_0(X)|_{[t_0+\delta_0,T_1]}$.},
\begin{equation*}\begin{split}
|\langle f,(P^*)^{-1}\psi\rangle|&\leq
\|f\|_{H^{-1}_0(X)|_{[t_0+\delta_0,T_1]}}
\,\|(P^*)^{-1}\psi\|_{H^1_0(X)|_{[t_0+\delta_0,T_1]}}\\
&\leq C\|f\|_{H^{-1}_0(X)|_{[t_0+\delta_0,T_1]}}
\|\psi\|_{H^{-1,1}_{0,\bl}(X)|_{[t_0,T_1]}},
\end{split}\end{equation*}
so $\ell$ is a continuous conjugate
linear functional if we equip
$\Ran_{\dCI_{\compl}(X)^\bullet_{[t_0,T_1]}} P^*$ with the
$H^{-1,1}_{0,\bl}(X)|_{[t_0,T_1]}$ norm.

If we did not care about the solution vanishing in $t<t_0+\delta_0$, we
could simply use Hahn-Banach to extend this to a continuous
conjugate linear functional $u$ on $H^{-1,1}_{0,\bl}(X)^\bullet_{[t_0,T_1]}$,
which can thus be identified
with an element of $H^{1,-1}_{0,\bl}(X)|_{[t_0,T_1]}$. This would give
$$
Pu(\phi)=\langle Pu,\phi\rangle=\langle u,P^*\phi\rangle=\ell(P^*\phi)
=\langle f,(P^*)^{-1}P^*\phi\rangle=\langle f,\phi\rangle,
$$
$\phi\in\dCI_{\compl}(X)^\bullet_{[t_0,T_1]}$, so $Pu=f$.

We do want the vanishing of $u$ in $(t_0,t_0+\delta_0)$, i.e.\ when applied
to $\phi$ supported in this region.
As a first step in this direction, let $\delta_0'\in(0,\delta_0)$, and
note that if
$$
\phi\in\dCI_\compl(X)^\bullet_{[t_0,t_0+\delta'_0)}
\cap \Ran_{\dCI_\compl(X)^\bullet_{[t_0,T_1]}} P^*
$$
then $\ell(\phi)=0$ directly by \eqref{eq:dual-map-def}, namely the
right hand side vanishes by the support condition on $f$. Correspondingly,
the conjugate linear map $L$ is well-defined on the algebraic sum
\begin{equation}\label{eq:smooth-sum-space}
\dCI_\compl(X)^\bullet_{[t_0,t_0+\delta'_0)}+\Ran_{\dCI_\compl(X)^\bullet_{[t_0,T_1]}} P^*
\end{equation}
by
$$
L(\phi+\psi)=\ell(\psi),\ \phi\in\dCI_\compl(X)^\bullet_{[t_0,t_0+\delta'_0)},
\ \psi\in\Ran_{\dCI_\compl(X)^\bullet_{[t_0,T_1]}} P^*.
$$
We claim that the functional $L$ is actually continuous when
\eqref{eq:smooth-sum-space} is equipped with the
$H^{-1,1}_{0,\bl}(X)|_{[t_0,T_1]}$ norm.
But this follows from
$$
|\langle f,(P^*)^{-1}\psi\rangle|
\leq C\|f\|_{H^{-1}_0(X)|_{[t_0+\delta_0,T_1]}}
\|\psi\|_{H^{-1,1}_{0,\bl}(X)|_{[t_0+\delta'_0,T_1]}}
$$
together with
$$
\|\psi\|_{H^{-1,1}_{0,\bl}(X)|_{[t_0+\delta'_0,T_1]}}\leq
\|\phi+\psi\|_{H^{-1,1}_{0,\bl}(X)|_{[t_0,T_1]}}
$$
since $\phi$ vanishes on $[t_0+\delta'_0,T_1]$. Correspondingly, by the
Hahn-Banach theorem, we
can extend $L$ to a continuous conjugate linear map
$$
u: H^{-1,1}_{0,\bl}(X)^\bullet_{[t_0,T_1]}\to\Cx,
$$
which can thus by identified
with an element of $H^{1,-1}_{0,\bl}(X)|_{[t_0,T_1]}$.
This gives
$$
Pu(\phi)=\langle Pu,\phi\rangle=\langle u,P^*\phi\rangle=\ell(P^*\phi)
=\langle f,(P^*)^{-1}P^*\phi\rangle=\langle f,\phi\rangle,
$$
$\phi\in\dCI_{\compl}(X)^\bullet_{[t_0,T_1]}$ supported in $(t_0,T_1)$, so $Pu=f$, and
in addition
$$
u(\phi)=0,\ \phi\in\dCI_{\compl}(X)^\bullet_{[t_0,t_0+\delta'_0]},
$$
so
\begin{equation}\label{eq:supp-u}
t\geq t_0+\delta'_0\ \text{on}\ \supp u.
\end{equation}
In particular, extending $u$ to vanish on $(-\infty,t_0+\delta'_0)$, which
is compatible with the existing definition in view of \eqref{eq:supp-u}, we
have a distribution solving the PDE, defined on $t<T_1$, with the
desired support condition. In particular, using
a cutoff function $\chi$ which is identically $1$ for $t\in(-\infty,T'_1]$,
is supported for $t\in(-\infty,T_1]$, $\chi u\in H^{1,-1}_{0,\bl}(X)$,
$\chi u$ vanishes for $t<t_0+\delta'_0$ as well as $t\geq T_1$,
and $Pu=f$ on $(-\infty,T'_1)$, thus completing the proof.
\end{proof}

\begin{prop}\label{prop:weak-exist}
Suppose $\lambda<(n-1)^2/4$.
For $f\in H^{-1}_{0,\loc}(X)$ supported in $t>t_0$, there exists
$u\in H^{1,-1}_{0,\bl,\loc}(X)$,
$\supp u\subset \{p: t(p)\geq t_0\}$ and $Pu=f$.
\end{prop}

\begin{proof}
We subdivide the time line into intervals
$[t_j,t_{j+1}]$,
each of which is sufficiently short so that energy estimates hold even on
$[t_{j-2},t_{j+3}]$; this can be done in view of the uniform estimates
on the length of such intervals over compact subsets.
Using a partition of unity, we may assume
that $f$ is supported in $[t_{k-1},t_{k+2}]$, and need to construct
a global solution of $Pu=f$ with $u$ supported in $[t_{k-1},\infty)$.
First we obtain $u_k$ as above solving the PDE on $(-\infty,t_{k+2}]$
(i.e.\ $Pu_k-f$ is supported in $(t_{k+2},\infty)$) and
supported in $[t_{k-1},t_{k+3}]$. Let $f_{k+1}=Pu_k-f$, this is
thus supported in $[t_{k+2},t_{k+3}]$. We next solve $Pu_{k+1}=-f_{k+1}$
on $(-\infty,t_{k+3}]$ with a result supported in $[t_{k+1},t_{k+4}]$.
Then $P(u_k+u_{k+1})-f$ is supported in $[t_{k+3},t_{k+4}]$, etc.
Proceeding inductively, and noting that the resulting sum is locally
finite, we obtain the solution on all of $X$.
\end{proof}

Well-posedness of the solution will follow once we show that
for solutions $u\in H^{1,s'}_{0,\bl,\loc}(X)$
of $Pu=f$, $f\in H^{-1,s}_{0,\bl,\loc}(X)$ supported
in $t>t_0$, we in fact have
$u\in H^{1,s-1}_{0,\bl,\loc}(X)$; indeed, this is a consequence of the
propagation
of singularities. We state this as a theorem now, recalling the
standing assumptions as well:

\begin{thm}\label{thm:well-posed}
Assume that (TF) and (PT) hold.
Suppose $\lambda<(n-1)^2/4$.
For $f\in H^{-1,1}_{0,\bl,\loc}(X)$ supported in $t>t_0$, there exists a unique
$u\in H^{1}_{0,\loc}(X)$ such that
$\supp u\subset \{p: t(p)\geq t_0\}$ and $Pu=f$. Moreover,
for $K\subset X$ compact there is $K'\subset X$ compact, depending
on $K$ and $t_0$ only, such that
\begin{equation}\label{eq:stability}
\|u|_K\|_{H^1_0(X)}\leq \|f|_{K'}\|_{H^{-1,1}_{0,\bl}(X)}.
\end{equation}
\end{thm}

\begin{rem}
While we used $\tau$ of Lemma~\ref{lemma:t-global-construction} instead
of $t$ throughout, the conclusion of this theorem is invariant
under this change (since $\delta_0>0$ is arbitrary in
Lemma~\ref{lemma:t-global-construction}),
and thus is actually valid for the original $t$ as well.
\end{rem}

\begin{proof}
Uniqueness and \eqref{eq:stability}
follow from Corollary~\ref{cor:unique} and the estimate
preceding it. By Proposition~\ref{prop:weak-exist}, this problem
has a solution $u\in H^{1,-1}_{0,\bl,\loc}(X)$ with the desired support
property. By the propagation
of singularities, Theorem~\ref{thm:prop-sing},
$u\in H^1_{0,\loc}(X)$ since $u$ vanishes for $t<t_0$.
\end{proof}

\section{Zero-differential operators and b-pseudodifferential operators}
\label{sec:Diff0-Psib}

In order to microlocalize, we need to replace $\Diffb(X)$ by $\Psib(X)$
and $\Psibc(X)$. We refer to \cite{Melrose:Atiyah} for a thorough
discussion and \cite[Section~2]{Vasy:Propagation-Wave}
for a concise introduction
to these operator algebras including all the facts that are required
here. In particular, the distinction between $\Psib(X)$
and $\Psibc(X)$ is the same as between $\Psi_{\cl}(\RR^n)$ and $\Psi(\RR^n)$
of {\em classical}, or {\em one step polyhomogeneneous},
resp.\ standard, pseudodifferential operators,
i.e.\ elements of the former ($\Psib(X)$, resp.\ $\Psi_{\cl}(\RR^n)$) are (locally)
quantizations of symbols with
a full one-step polyhomogeneous
asymptotic expansion (also called {\em classical} symbols),
while those of the latter ($\Psibc(X)$, resp.\ $\Psi(\RR^n)$)
are (locally) quantizations of symbols which merely satisfy symbolic estimates.
While the former are convenient since they have homogeneous principal
symbols, the latter are more
useful when one must use approximations (e.g.\ by smoothing operators), as
is often the case below.
Before proceeding, we recall that points in the b-cotangent bundle, $\Tb^*X$, of $X$
are of the form
$$
\xib\,\frac{dx}{x}+\sum_{j=1}^{n-1}\zetab_j\,dy_j.
$$
Thus, $(x,y,\xib,\zetab)$ give coordinates on $\Tb^*X$. If $(x,y,\xi,\zeta)$ are
the standard coordinates on $T^*X$ induced by local coordinates
on $X$, i.e.\ one-forms are written as
$\xi\,dx+\zeta\,dy$, then the map $\pi:T^*X\to\Tb^*X$ is given by
$\pi(x,y,\xi,\zeta)=(x,y,x\xi,\zeta)$.

To be a bit more concrete (but again we refer to \cite{Melrose:Atiyah}
and \cite[Section~2]{Vasy:Propagation-Wave} for more detail), we can define a
large subspace (which in fact is sufficient for our purposes here)
of $\Psibc^m(X)$ and $\Psib^m(X)$ locally by explicit {\em quantization maps};
these can be combined to a global quantization map by a partition of unity
as usual.
Thus, we have $q=q_m:S^m(\Tb^*X)\to
\Psibc^m(X)$, which restrict to $q:S^m_{\cl}(\Tb^*X)\to
\Psib^m(X)$, $\cl$ denoting classical symbols. Namely, over a local
coordinate chart $U$ with coordinates $(x,y)$, $y=(y_1,\ldots,y_{n-1})$,
and with $a$ supported in $\Tb^*_K X$,
$K\subset U$ compact, we may take
\begin{equation*}\begin{split}
&q(a)u(x,y)\\
&\qquad=(2\pi)^{-n}\int e^{\imath((x-x')\xi+(y-y')\cdot\zeta)}
\phi\big(\frac{x-x'}{x}\big) 
a(x,y,x\xi,\zeta) u(x',y')\,dx'\,dy'\,d\xi\,d\zeta,
\end{split}\end{equation*} 
understood as an oscillatory integral,
where $\phi\in\Cinf_\compl((-1/2,1/2))$ is identically $1$ near $0$,
and the integral in $x'$ is over $[0,\infty)$.
Note that $\phi$ is irrelevant as far as the behavior of Schwartz kernels near the
diagonal is concerned (it is identically $1$ there); it simply localizes to a neighborhood
of the diagonal. Somewhat inaccurately, one may write $q(a)$ as $a(x,y,xD_x,D_y)$,
so $a$ is symbolic in b-vector fields; a more accurate way of reflecting this
is to change variables, writing $\xib=x\xi$, $\zetab=\zeta$, so
\begin{equation}\begin{split}\label{eq:quant-b}
&q(a)u(x,y)\\
&\qquad=(2\pi)^{-n}\int e^{\imath(\tfrac{x-x'}{x}\xib+(y-y')\cdot\zetab)}
\phi\big(\frac{x-x'}{x}\big)
a(x,y,\xib,\zetab) u(x',y')\,\frac{dx'}{x}\,dy'\,d\xib\,d\zetab.
\end{split}\end{equation}

With this explicit quantization, the {\em principal symbol} $\sigma_{\bl,m}(A)$
of $A=q(a)$ is the class $[a]$ of $a$ in $S^m(\Tb^*X)/S^{m-1}(\Tb^*X)$. If
$a$ is classical, this class can be further identified with a homogeneous, of degree $m$,
symbol, i.e.\ an element of $S^m_{\hom}(\Tb^*X\setminus o)$. On the other
hand, the {\em operator wave front set} $\WFb'(A)$ of $A=q(a)$ can be defined by
saying that $p\in\Tb^*X\setminus o$ is {\em not} in $\WFb'(A)$ if $p$ has
a conic neighborhood $\Gamma$
in $\Tb^*X\setminus o$ such that $a=a(x,y,\xib,\zetab)$ is rapidly decreasing
(i.e.\ is an order $-\infty$ symbol) in $\Gamma$. Thus, on the complement of
$\WFb'(A)$, $A$ is microlocally order $-\infty$.

A somewhat better definition of $\Psibc(X)$ and $\Psib(X)$ is
directly in terms of the Schwartz kernels. The Schwartz kernels are well-behaved
on the b-double space $X^2_{\bl}=[X^2;(\pa X)^2]$ created by blowing
up the corner $(\pa X)^2$ in the product space $X^2=X\times X$; in particular they
are smooth away from the diagonal and vanish to infinite order
off the front face. In these terms $\phi$ above localizes to a neighborhood
of the diagonal which only intersects the boundary of $X^2_{\bl}$ in the front
face of the blow-up. The equivalence of the two descriptions can be
read off directly from \eqref{eq:quant-b}, which shows that the Schwartz kernel
is a right b-density valued (this is the factor $\frac{dx'}{x}\,dy'$ in
\eqref{eq:quant-b}) conormal distribution to $\tfrac{x-x'}{x}=0$, $y-y'=0$,
i.e.\ the lift of the diagonal to $X^2_{\bl}$.

The space $\Psibc(X)$ forms a filtered algebra, so for $A\in\Psibc^m(X)$,
$B\in\Psibc^{m'}(X)$, one has
$AB\in\Psibc^{m+m'}(X)$. In addition, the commutator satisfies
$[A,B]\in\Psibc^{m+m'-1}(X)$, i.e.\ it is one order lower than the product,
but there is no gain of decay at $\pa X$. However,
we also recall the following crucial lemma from
\cite[Section~2]{Vasy:Propagation-Wave}:

\begin{lemma}\label{lemma:xD_x-improvement}
For $A\in\Psibc^m(X)$, resp.\ $A\in\Psib^m(X)$, one has
$[xD_x, A]\in x\Psibc^m(X)$, resp.\ $[xD_x, A]\in x\Psib^m(X)$.
\end{lemma}

\begin{proof}
The lemma is an immediate consequence of $xD_x$ having a commutative
normal operator; see \cite{Melrose:Atiyah} for a detailed discussion
and \cite[Section~2]{Vasy:Propagation-Wave} for a brief explanation.
\end{proof}

{\em For simplicity of notation we state the results below through
Lemma~\ref{lemma:b-to-0-conversion-ps} for $\Psib(X)$;
they work equally well if one replaces $\Psib(X)$ by
$\Psibc(X)$ throughout.}

The analogue of Lemma~\ref{lemma:Diffb-commutant} with $\Diffb(X)$
replaced by $\Psib(X)$ still holds, without the awkward restriction
on positivity of b-orders (which is simply due to the lack of non-trivial
negative order differential operators).

\begin{Def}
Let $\Diffz^k\Psib^m(X)$ be the (complex) vector space
of operators on $\dCI(X)$ of the
form
$$
\sum P_j Q_j,\ P_j\in\Diffz^k(X),\ Q_j\in\Psib^m(X),
$$
where the sum is locally finite, and let
$$
\Diffz\Psib(X)=\cup_{k=0}^\infty\cup_{m\in\RR}^\infty\Diffz^k\Psib^m(X).
$$

We define $\Diffz^k\Psibc^m(X)$ similarly, by replacing $\Psib(X)$ by $\Psibc(X)$
throughout the definition.
\end{Def}

The ring structure (even with a weight $x^r$) of $\Diffz\Psib(X)$
was proved in \cite[Corollary~4.4 and
Lemma~4.5]{Vasy:De-Sitter}, which we recall here. We add to the statements
of \cite[Corollary~4.4 and
Lemma~4.5]{Vasy:De-Sitter} that $\Diffz\Psib(X)$ is also closed under adjoints
with respect to any weighted non-degenerate b-density, in particular
with respect to
a non-degenerate
0-density such as $|dg|$, for both $\Diffz(X)$ and $\Psib(X)$ are
closed under these adjoints and $(AB)^*=B^*A^*$.

\begin{lemma}\label{lemma:Psib-filtered-ring}
$\Diffz\Psib(X)$ is a filtered *-ring under composition (and
adjoints) with
$$
A\in\Diffz^k\Psib^m(X),\ B\in\Diffz^{k'}\Psib^{m'}(X)
\Rightarrow AB\in\Diffz^{k+k'}\Psib^{m+m'}(X),
$$
and
$$
A\in\Diffz^k\Psib^m(X)\Rightarrow A^*\in\Diffz^k\Psib^m(X),
$$
where the adjoint is taken with respect to a (i.e.\ {\em any fixed})
non-degenerate 0-density.
Moreover, composition is commutative to leading order in $\Diffz$,
i.e.\ for $A,B$ as above, $k+k'\geq 1$,
$$
[A,B]\in\Diffz^{k+k'-1}\Psib^{m+m'}(X).
$$
\end{lemma}

Just like for differential operators,
we again have a lemma that improves the b-order (rather than
merely the 0-order) of the commutator
provided one of the commutants is in $\Psib(X)$. Again, it
is crucial here that there are no weights on $\Psib(X)$.

\begin{lemma}\label{lemma:b-comm-improve}
For $A\in\Psib^s(X)$, $B\in\Diffz^k\Psib^m(X)$,
$[A,B]\in\Diffz^k\Psib^{s+m-1}(X)$.
\end{lemma}

\begin{proof}
Expanding elements of $\Diffz^k(X)$ as finite sums of
products of vector fields and functions, and using that $\Psib(X)$ is
commutative to leading order, we need to consider commutators
$[f,A]$, $f\in\CI(X)$, $A\in\Psib^s(X)$ and show that this is
in $\Psib^{s-1}(X)$, which is automatic as $\CI(X)\subset\Psib^0(X)$,
as well as $[V,A]$, $V\in\Vz(X)$, $A\in\Psib^s(X)$, and show that
this is in $\Diffz^1\Psib^{s-1}(X)$, i.e.
$$
[V,A]=\sum W_jB_j+C_j,\ B_j,C_j\in\Psib^{s-1}(X),\ W_j\in\Vz(X).
$$
But $V=xV'$, $V'\in\Vf(X)$,
and
$$
[V',A]=\sum_j W'_j B'_j+C'_j,\ W'_j\in\Vf(X),\ B'_j,C'_j\in\Psib^{s-1}(X),
$$
see \cite[Lemma~2.2]{Vasy:Propagation-Wave}, while
$B''=[x,A]x^{-1}\in \Psib^{s-1}(X)$, so
$$
[V,A]=[x,A]V'+x[V',A]=B''(xV')+\sum_j (xW'_j )B'_j+xC'_j,
$$
which is of the desired form once the first term is rearranged using
Lemma~\ref{lemma:Psib-filtered-ring},
i.e.\ explicitly $B''(xV')=(xV')B''+[B'',xV']$, with the
last term being an element of $\Psib^{s-1}(X)$.
\end{proof}

We also have an analogue of Lemma~\ref{lemma:b-to-0-conversion}.

\begin{lemma}\label{lemma:b-to-0-conversion-ps}
For any $l\geq 0$ integer,
$$
x^l\Diffz^k\Psib^m(X)\subset\Diffz^{k+l}\Psib^{m-l}(X).
$$
\end{lemma}

\begin{proof}
It suffices to show that $x\Psib^m(X)\subset \Diffz^1\Psib^{m-1}(X)$,
for the rest follows by induction. Also, we may localize and assume that
$A$ is supported in a coordinate patch; note that
$$
\Psib^{-\infty}(X)\subset
\Diffz^1\Psib^{-\infty}(X)
$$
since $\CI(X)\subset\Diffz^1(X)$.
Thus, suppose $A\in\Psib^m(X)$.
Then there exist $A_j\in\Psib^{m-1}(X)$, $j=0,\ldots,n-1$, and
$R\in\Psib^{-\infty}(X)$ such that
$$
A=(xD_x)A_0+\sum_j D_{y_j}A_j+R;
$$
indeed, one simply needs to use the ellipticity of $L=(xD_x)^2+\sum D_{y_j}^2$
to achieve this by constructing a parametrix $G\in\Psib^{-2}(X)$
to it, and writing
$A=LGA+EA$, $E\in\Psib^{-\infty}(X)$. As $x(xD_x),xD_{y_j}\in\Vz(X)$, the
conclusion follows.
\end{proof}

As a consequence of our results thus far,
we deduce that $\Psib^0(X)$ is bounded on
$H^m_0(X)$, already stated in \cite[Lemma~4.7]{Vasy:De-Sitter}.

\begin{prop}\label{prop:H10-bd}
Suppose $m\in\ZZ$. Any $A\in\Psibc^0(X)$ with compact support
defines a bounded operator on $H^m_0(X)$, with
operator norm bounded by a seminorm of $A$ in $\Psibc^0(X)$.
\end{prop}

\begin{proof}
For $m\geq 0$ this is a special case of \cite[Lemma~4.7]{Vasy:De-Sitter},
though the fact that the operator norm is
bounded by a seminorm of $A$ in $\Psibc^0(X)$ was not explicitly stated
there though follows from the proof; $m<0$ follows by duality.

For the convenience of the reader we recall the proof in the
case we actually use in this paper, namely $m=1$ (then $m=-1$
follows by duality).
Any $A$ as in the statement of the proposition
is bounded on $L^2(X)$ with the stated properties.
Thus, we need to show that if $V\in\Vz(X)$, then $VA:H^1_0(X)\to L^2(X)$.
But $VA=AV+[V,A]$, $[V,A]\in\Diffz^1\Psib^{-1}(X)\subset\Psib^0(X)$, hence
$AV:H^1_0(X)\to L^2(X)$ and $[V,A]:L^2(X)\to L^2(X)$, with the
claimed norm behavior.
\end{proof}

If $q$ is a homogeneous function on $\Tb^* X\setminus o$, then we again
consider the
Hamilton vector field $\sH_q$ associated to it on $T^*X^\circ\setminus o$. A
change of coordinates calculation shows that in the
b-canonical coordinates given above
\begin{equation*}
\sH_q=(\pa_\xib q)x\pa_x-(x\pa_x q)\pa_\xib
+(\pa_\zetab q)\pa_y-(\pa_y q)\pa_\zetab,
\end{equation*}
so $\sH_q$ extends to a $\CI$ vector field on $\Tb^* X\setminus o$
which is tangent to $\Tb^*_{\pa X}X$.
If $Q\in\Psib^{m'}(X)$, $P\in\Psib^{m}(X)$, then $[Q,P]\in\Psib^{m+m'-1}(X)$
has principal symbol
$$
\sigma_{\bl,m+m'-1}([Q,P])=\frac{1}{\imath}\sH_q p.
$$

Using Proposition~\ref{prop:H10-bd} we can define a meaningful
$\WFb$ relative to $H^1_0(X)$. First we recall the definition of
the corresponding global function
space from \cite[Section~4]{Vasy:De-Sitter}:

For $k\geq 0$ the b-Sobolev spaces relative to $H^r_0(X)$ are given
by\footnote{We do not need weighted spaces, unlike
in \cite{Vasy:De-Sitter}, so we only state the definition in the
special case when the weight is identically $1$. On the other hand,
we are working on a non-compact space, so we must consider local
spaces and spaces of compactly supported functions
as in \cite[Section~3]{Vasy:Propagation-Wave}. Note also that we reversed
the index convention (which index comes first) relative to
\cite{Vasy:De-Sitter}, to match the notation for the wave front sets.}
\begin{equation*}
H^{r,k}_{0,\bl,\compl}(X)=\{u\in H^r_{0,\compl}(X):\ \forall A\in\Psib^k(X),
\ Au\in H^r_{0,\compl}(X)\}.
\end{equation*}
These can be normed by taking any properly supported
elliptic $A\in\Psib^k(X)$ and
letting
\begin{equation*}
\|u\|_{H^{r,k}_{0,\bl,\compl}(X)}^2=\|u\|_{H^r_0(X)}^2
+\|Au\|^2_{H^r_0(X)}.
\end{equation*}
Although the norm depends on the choice of $A$, for $u$ supported in a
fixed compact set, different choices
give equivalent norms, see \cite[Section~4]{Vasy:De-Sitter} for
details in the 0-setting (where supports are not an issue), and
\cite[Section~3]{Vasy:Propagation-Wave} for an analysis involving supports.
We also
let $H^{r,k}_{0,\bl,\loc}(X)$ be the subspace of $H^r_{0,\loc}(X)$
consisting of $u\in H^r_{0,\loc}(X)$ such that for any
$\phi\in\Cinf_\compl(X)$,
$\phi u\in H^{r,k}_{0,\bl,\compl}(X)$.

Here it is also useful to have Sobolev spaces with a negative amount
of b-regularity, in a manner completely analogous to
\cite[Definition~3.15]{Vasy:Propagation-Wave}:

\begin{Def}\label{Def:H1m-neg}
Let $r$ be an integer, $k<0$, and $A\in\Psib^{-k}(X)$ be elliptic on $\Sb^*X$
with proper support.
We let $H^{r,k}_{0,\bl,\compl}(X)$ be the space of all $u\in\dist(X)$ of the
form $u=u_1+Au_2$ with $u_1,u_2\in H^r_{0,\compl}(X)$.
We let
\begin{equation*}
\|u\|_{H^{r,k}_{0,\bl,\compl}(X)}
=\inf\{\|u_1\|_{H^r_0(X)}+\|u_2\|_{H^r_0(X)}:\ u=u_1+Au_2\}.
\end{equation*}
We also let $H^{r,k}_{0,\bl,\loc}(X)$ be the space of all $u\in\dist(X)$
such that $\phi u\in H^{r,k}_{0,\bl,\compl}(X)$ for all
$\phi\in\Cinf_\compl(X)$.
\end{Def}

As discussed for analogous spaces in \cite{Vasy:Propagation-Wave}
following Definition~3.15 there, this definition is independent of the
particular $A$ chosen, and different $A$ give equivalent norms for
distributions $u$ supported in a fixed compact set $K$. Moreover,
we have

\begin{lemma}
Suppose $r\in\ZZ$, $k\in\RR$. Any $B\in\Psibc^0(X)$ with compact support
defines a bounded operator on $H^{r,k}_{0,\bl}(X)$, with
operator norm bounded by a seminorm of $B$ in $\Psibc^0(X)$.
\end{lemma}

\begin{proof}
Suppose $k\geq 0$ first. Then for an $A\in\Psib^k(X)$ as in the definition
above,
\begin{equation*}
\|Bu\|_{H^{r,k}_{0,\bl,\compl}(X)}^2=\|Bu\|_{H^r_0(X)}^2
+\|ABu\|^2_{H^r_0(X)}.
\end{equation*}
The first term on the right hand side is bounded in the desired manner
due to Proposition~\ref{prop:H10-bd}.
Letting $G\in\Psib^{-k}(X)$ be a properly supported parametrix for $A$
so $GA=\Id+E$, $E\in\Psib^{-\infty}(X)$, we have
$ABu=AB(GA-E)u=(ABG)Au-(ABE)u$, with $ABG\in\Psibc^0(X)$,
$ABE\in\Psibc^{-\infty}(X)\subset\Psibc^0(X)$, so
$$
\|ABu\|_{H^r_0(X)}\leq C\|Au\|_{H^r_0(X)}+C\|u\|_{H^r_0(X)}
$$
by Proposition~\ref{prop:H10-bd}, with $C$ bounded by a seminorm of $B$.
This completes the proof if $k\geq 0$.

For $k<0$, let $A\in\Psib^{-k}(X)$ be as in the definition. If $u=u_1+Au_2$,
and $G\in\Psib^k(X)$ is a parametrix for $A$ so $AG=\Id+F$,
$F\in\Psib^{-\infty}(X)$, hence
$$
Bu=Bu_1+BAu_2=Bu_1+(AG-F)BAu_2=Bu_1+A(GBA)u_2-(FBA)u_2.
$$
Now, $B,FBA,GBA\in\Psib^0(X)$ so $Bu\in H^{r,k}_{0,\bl,\compl}(X)$ indeed,
and choosing $u_1$, $u_2$ so that $\|u_1\|_{H^r_0(X)}+\|u_2\|_{H^r_0(X)}
\leq 2\|u\|_{H^{r,k}_{0,\bl,\compl}(X)}$ shows the desired continuity,
as well as that the operator norm of $B$ is bounded by a
$\Psibc^0(X)$-seminorm.
\end{proof}

Now we define the wave front set relative to $H^r_{0,\loc}(X)$.
We also allow negative a priori b-regularity relative to this space.

\begin{Def}\label{def:WFb}
Suppose $u\in H^{r,k}_{0,\loc}(X)$, $r\in\ZZ$, $k\in\RR$.
Then $q\in\Tb^*X\setminus o$
is {\em not} in $\WFb^{r,\infty}(u)$ if there is an
$A\in\Psib^0(X)$ such that
$\sigma_{\bl,0}(A)(q)$ is invertible and
$QAu\in H^r_{0,\loc}(X)$ for all
$Q\in\Diffb(X)$, i.e.\ if $Au\in H^{r,\infty}_{0,\bl,\loc}(X)$.

Moreover, $q\in\Tb^*X\setminus o$
is {\em not} in $\WFb^{r,m}(u)$ if there is an $A\in\Psib^m(X)$
such that
$\sigma_{\bl,0}(A)(q)$ is invertible and $Au\in H^r_{0,\loc}(X)$.
\end{Def}

Proposition~\ref{prop:H10-bd} implies that $\Psibc(X)$ acts microlocally,
i.e.\ preserves $\WFb$; see \cite[Section~3]{Vasy:Propagation-Wave}
for a similar argument. In particular, the proofs for both the qualitative and
quantitative version of microlocality go through without any significant
changes; one simply replaces the use of \cite[Lemma~3.2]{Vasy:Propagation-Wave}
by Proposition~\ref{prop:H10-bd}.

\begin{lemma}(cf.\ \cite[Lemma~3.9]{Vasy:Propagation-Wave}\label{lemma:WFb-mic}
Suppose that $u\in H^{r,k'}_{0,\bl,\loc}(X)$, $B\in\Psibc^k(X)$.
Then $\WFb^{r,m-k}(Bu)
\subset\WFb^{r,m}(u)\cap\WFb'(B)$.
\end{lemma}

As in \cite[Section~3]{Vasy:Propagation-Wave}, the wave front set
microlocalizes the `b-singular support relative to $H^r_{0,\loc}(X)$',
meaning:

\begin{lemma}(cf.\ \cite[Lemma~3.10]{Vasy:Propagation-Wave})
\label{lemma:microloc-to-loc}
Suppose $u\in H^{r,k}_{0,\bl,\loc}(X)$, $p\in X$. If $\Sb^*_p X\cap
\WFb^{1,m}(u)=\emptyset$, then in a neighborhood of $p$,
$u$ lies in $H^{1,m}_{0,\bl}(X)$, i.e.\ there is $\phi\in\Cinf_\compl(X)$ with
$\phi\equiv 1$ near $p$ such that $\phi u\in H^{1,m}_{0,\bl}(X)$.
\end{lemma}

\begin{cor}(cf.\ \cite[Corollary~3.11]{Vasy:Propagation-Wave})
\label{cor:WF-to-H1}
If $u\in H^{r,k}_{0,\bl,\loc}(X)$
and $\WFb^{r,m}(u)=\emptyset$, then $u\in H^{r,m}_{0,\bl,\loc}(X)$.

In particular, if $u\in H^{r,k}_{0,\bl,\loc}(X)$
and $\WFb^{r,m}(u)=\emptyset$ for all $m$,
then $u\in H^{r,\infty}_{0,\bl,\loc}(X)$, i.e.\ $u$ is conormal in the sense
that $Au\in H^r_{0,\loc}(X)$ for all
$A\in\Diffb(X)$ (or indeed $A\in\Psib(X)$).
\end{cor}

Finally, we have the following quantitative bound for which we recall
the definition of the wave front set of bounded subsets of $\Psibc^k(X)$:

\begin{Def}(cf.\ \cite[Definition~3.12]{Vasy:Propagation-Wave})
Suppose that $\cB$ is a bounded subset of $\Psibc^k(X)$, and $q\in\Sb^*X$.
We say that $q\notin\WFb'(\cB)$ if there is some $A\in\Psib(X)$
which is elliptic at $q$ such that $\{AB:\ B\in\cB\}$ is a bounded
subset of $\Psib^{-\infty}(X)$.
\end{Def}

\begin{lemma}(cf.\ \cite[Lemma~3.13, Lemma~3.18]{Vasy:Propagation-Wave})
\label{lemma:WFb-mic-q}
Suppose that $K\subset\Sb^*X$ is compact, and $U$ a neighborhood of $K$
in $\Sb^*X$. Let $\tilde K\subset X$ compact,
and $\tilde U$ be a neighborhood of $\tilde K$ in $X$ with compact closure.
Let $Q\in\Psib^k(X)$ be elliptic on $K$ with $\WFb'(Q)
\subset U$, with Schwartz kernel supported in $\tilde K\times\tilde K$.
Let $\cB$ be a bounded subset of $\Psibc^k(X)$
with $\WFb'(\cB)\subset K$ and
Schwartz kernel supported in $\tilde K\times\tilde K$.
Then for any $s\leq 0$
there is a constant $C>0$ such that for $B\in\cB$, $u\in H^{r,s}_{0,\bl,\loc}(X)$
with $\WFb^{r,k}(u)\cap U=\emptyset$,
\begin{equation*}
\|Bu\|_{H^r_0(X)}\leq C(\|u\|_{H^{r,s}_{0,\bl}(\tilde U)}+\|Qu\|_{H^r_0(X)}).
\end{equation*}
\end{lemma}

We can use this lemma to obtain uniform bounds for pairings. We call
a subset $\cB$ of $\Diff^m_0\Psibc^{2k}(X)$ bounded if its elements are
locally finite linear combinations of a fixed, locally finite, collection
of elements of $\Diff^m_0(X)$ with
coefficients that lie in a bounded subset of $\Psibc^{2k}(X)$.

\begin{cor}
\label{cor:WFb-mic-q}
Suppose that $K\subset\Sb^*X$ is compact, and $U$ a neighborhood of $K$
in $\Sb^*X$. Let $\tilde K\subset X$ compact,
and $\tilde U$ be a neighborhood of $\tilde K$ in $X$ with compact closure.
Let $Q\in\Psib^k(X)$ be elliptic on $K$ with $\WFb'(Q)
\subset U$, with Schwartz kernel supported in $\tilde K\times\tilde K$.
Let $\cB$ be a bounded subset of $\Diffz^2\Psibc^{2k}(X)$
with $\WFb'(\cB)\subset K$ and
Schwartz kernel supported in $\tilde K\times\tilde K$.
Then there is a constant $C>0$ such that for $B\in\cB$,
$u\in H^{1,s}_{0,\bl,\loc}(X)$
with $\WFb^{1,k}(u)\cap U=\emptyset$,
\begin{equation*}
|\langle Bu,u\rangle|\leq C(\|u\|_{H^1_0(\tilde U)}+\|Qu\|_{H^1_0(X)})^2.
\end{equation*}
\end{cor}

\begin{proof}
Using Lemma~\ref{lemma:Psib-filtered-ring} we can write $B$ as
$\sum B'_{ij} P_i^* R_j \Lambda$, where $P_i,R_j\in\Diffz^1(X)$,
$\Lambda\in\Psib^k(X)$ (which we take to be elliptic on $K$, but such
that $Q$ is elliptic on $\WFb'(\Lambda)$),
$B'_{ij}$ lies
in a bounded subset $\cB'$ of $\Psib^{k}(X)$ and the sum is finite. Then
\begin{equation*}\begin{split}
|\langle Bu,u\rangle|&\leq \sum_{ij}|\langle R_j\Lambda u,
P_i (B'_{ij})^* u\rangle|\\
&\leq \sum_{ij}\|R_j\Lambda u\|_{L^2(X)} \,\|P_i (B'_{ij})^* u\|_{L^2(X)}\\
&\leq \sum_{ij}\|\Lambda u\|_{H^1_0(X)} \,\|P_i (B'_{ij})^* u\|_{H^1_0(X)}\\
&\leq \sum C(\|u\|_{H^{1,s}_{0,\bl}(\tilde U)}+\|Qu\|_{H^1_0(X)})^2,
\end{split}\end{equation*}
where in the last step we used Lemma~\ref{lemma:WFb-mic-q}.
\end{proof}

It is useful to note that infinite
order b-regularity relative to $L^2_0(X)$ and $H^1_0(X)$ are the same.

\begin{lemma}\label{lemma:WFb-infty-independence}
For $u\in H^1_{0,\loc}(X)$,
$$
\WFb^{1,\infty}(u)=\WFb^{0,\infty}(u).
$$
\end{lemma}

\begin{proof}
The complements of the two sides are the set of points $q\in\Sb^*X$ for which
there exist $A\in\Psib^0(X)$ (with compactly supported Schwartz
kernel, as one may assume) such that $\sigma_{\bl,0}(A)(q)$ is
invertible and $LAu\in H^1_0(X)$, resp.\ $LAu\in L^2_0(X)$.
Since $H^1_0(X)\subset L^2_0(X)$,
$\WFb^{0,\infty}(u)\subset\WFb^{1,\infty}(u)$ follows immediately. For the
converse, if $LAu\in L^2_0(X)$ for all $L\in\Diffb(X)$, then in
particular $\Diff_0(X)\subset\Diffb(X)$ shows that
$QLAu\in L^2_0(X)$ for $Q\in \Diff_0^1(X)$ and $L\in\Diffb(X)$,
so $LAu\in H^1_0(X)$, i.e.\ $\WFb^{1,\infty}(u)\subset\WFb^{0,\infty}(u)$,
completing the proof.
\end{proof}

We finally recall that $u\in\cA^k(X)$, i.e.\ that $u$ is conormal relative
to $x^k L^2_\bl(X)$,
means that
$Lu\in x^k L^2_\bl(X)$ for all $L\in\Diffb(X)$, so in particular
$u\in x^k L^2_\bl(X)$. Thus,
$$
\WFb^{0,\infty}(u)=\emptyset\ \text{if and only if}
\ u\in\cA^{(n-1)/2}(X),
$$
in view of $L^2_0(X)=x^{(n-1)/2}L^2_\bl(X)$.

\section{Generalized broken bicharacteristics}\label{sec:GBB}
We recall here the structure of the compressed characteristic
set and \GBBsp from \cite[Section~1-Section~2]{Vasy:Maxwell}. In that paper
$X$ is a manifold with corners, and $k$ is the codimension of the highest
codimension corner in the local coordinate chart. Thus, for application
to the present paper, the reader should take $k=1$ when referring to
\cite[Section~1-Section~2]{Vasy:Maxwell}.
It is often convenient to work on the cosphere bundle, here
$\Sb^*X$, which is
equivalent to working on conic subsets of $\Tb^*X\setminus o$.
In a region where, say,
\begin{equation}\label{eq:zeta-large}
|\xib|<C|\zetab_{n-1}|,\ |\zetab_j|<C|\zetab_{n-1}|,
j=1,\ldots,n-2,
\end{equation}
$C>0$ fixed, we can take
\begin{equation*}\begin{split}
&x,y_1,\ldots,y_{n-1},\xibh,\zetabh_1,
\ldots,\zetabh_{n-2},|\zetab_{n-1}|,\\
&\text{where}\ \xibh=\frac{\xib}{|\zetab_{n-1}|},
\ \zetabh_j=\frac{\zetab_j}{|\zetab_{n-1}|},
\end{split}\end{equation*}
as (projective) local coordinates on $\Tb^*X\setminus o$, hence
$$
x,y_1,\ldots,y_{n-1},\xibh,\zetabh_1,
\ldots,\zetabh_{n-2}
$$
as local coordinates on the image of this region under the quotient
map in $\Sb^*X$; cf.\ \cite[Equation~(1.4)]{Vasy:Maxwell}.

First, we choose local coordinates more carefully. In arbitrary local
coordinates
$$
(x,y_1,\ldots,y_{n-1})
$$
on a neighborhood $U_0$
of a point on $Y=\pa X$, so that $Y$ is given by $x=0$ inside $x\geq 0$,
any symmetric bilinear form on $T^*X$ can be written
as
\begin{equation}\label{eq:metric-form2}
\hat G(x,y)=A(x,y)\,\pa_x\,\pa_x
+\sum_{j}2 C_{j}(x,y)\,\pa_{x}\,\pa_{y_j}+\sum_{i,j}
B_{ij}(x,y)\,\pa_{y_i}\,\pa_{y_j}
\end{equation}
with $A,B,C$ smooth.
In view of \eqref{eq:g-form}, using $x$ given there and coordinates $y_j$
on $Y$ pulled by to a collar neighborhood of $Y$ by the product structure,
we have in addition
$$
A(0,y)=-1,\ C_j(0,y)=0,
$$
for all $y$, and $B(0,y)=(B_{ij}(0,y))$ is Lorentzian for all $y$.
Below we write covectors as
\begin{equation}\label{eq:T-star-X-coords}
\alpha=\xi\,dx+\sum_{i=1}^{n-1}\zeta_i\,dy_i.
\end{equation}
Thus,
\begin{equation}\label{eq:H-at-corner2}
\hat G|_{x=0}=-\pa_x^2
+\sum_{i,j=1}^{n-1}B_{ij}(0,y)\,\pa_{y_i}\,\pa_{y_j},
\end{equation}
and hence the metric function,
$$
p(q)=\hat G(q,q),\ q\in T^*X,
$$
is
\begin{equation}\label{eq:p-at-bdy2}
p|_{x=0}=-\xi^2+\zeta \cdot B(y)\zeta.
\end{equation}
Since $A(0,y)=-1<0$, $Y$
is indeed time-like in the sense
that the restriction of the dual metric
$\hat G$ to $N^*Y$ is negative definite,
for locally the conormal bundle $N^*Y$ is given by
$$
\{(x,y,\xi,\zeta):\ x=0,\ \zeta=0\}.
$$
We write
$$
h=\zeta\cdot B(y)\zeta
$$
for the metric function on the boundary.
Also, from \eqref{eq:p-at-bdy2},
\begin{equation}\label{eq:H_p-at-bdy}
\sH_p=-2\xi\cdot\pa_x+\sH_h+\beta \pa_\xi+xV,
\end{equation} 
where $V$ is a $\CI$ vector field in $\cU_0=T^*U_0$,
and $\beta$ is a $\CI$ function on
$\cU_0$.

It is sometimes convenient to improve the form of $B$ near a particular
point $p_0$, around which the coordinate system is centered.
Namely, as $B$ is Lorentzian, we can further arrange,
by adjusting the $y_j$ coordinates,
\begin{equation}\label{eq:B-orth-0}
\sum B_{ij}(0,0)\pa_{y_i}\pa_{y_j}=\pa_{y_{n-1}}^2-\sum_{i<n-1}\pa_{y_i}^2.
\end{equation}

We now recall from the introduction that $\pi:T^*X\to\Tb^*X$ is the
natural map corresponding to the identification of a section of $T^*X$ as
a section of $\Tb^*X$, and in local coordinates $\pi$ is given by
\begin{equation*}
\pi(x,y,\xi,\zeta)=(x,y,x\xi,\zeta).
\end{equation*}
Moreover, the image of the characteristic set
$\Sigma\subset T^*X\setminus o$, given by
$$
\Sigma=\{q\in T^*X:\ p(q)=0\},
$$
under $\pi$ is the compressed characteristic set,
$$
\dot\Sigma=\pi(\Sigma).
$$
Note that \eqref{eq:p-at-bdy2} gives that
\begin{equation}\begin{split}\label{eq:dot-Sigma-coords2}
\dot\Sigma\cap\cU_0\cap\Tb^*_Y X=\{(0,y,0,\zetab):\ 0\leq
\zetab \cdot B(y)\zetab,\ \zetab\neq 0\}.
\end{split}\end{equation}
In particular, in view of \eqref{eq:B-orth-0},
$\dot\Sigma\cap\cU_0$ lies in the region \eqref{eq:zeta-large}, at least
after we possibly shrink $U_0$ (recall that $\cU_0=T^*U_0$), as we assume from now.
We also remark that, using \eqref{eq:H_p-at-bdy},
\begin{equation}\label{eq:pi-H_p-at-bdy}
\pi_*|_{(x,y,\xi,\zeta)}\sH_p
=-2\xi\cdot(\pa_x+\xi\pa_{\xib})+\sH_h+x\beta \pa_{\xib}+x\pi_*V,
\end{equation}
and correspondingly
\begin{equation}\label{eq:pi-H_p-xib-at-bdy}
\sH_p\pi^*\xib\big|_{x=0}=-2\xi^2=2(p-\zeta\cdot B(y)\zeta)=-2\zeta\cdot B(y)\zeta,
\ (0,y,\xi,\zeta)\in\Sigma.
\end{equation} 
As we already noted, $\zetab_{n-1}$ cannot vanish on $\dot\Sigma\cap \cU_0$, so
\begin{equation}\begin{split}\label{eq:pi-H_p-xibh-at-bdy}
\sH_p\pi^*(\xib/|\zetab_{n-1}|)|_{x=0}
&=-2|\zeta_{n-1}|^{-1}\xi^2-x\xi|\zeta_{n-1}|^{-2}(\sH_h|\zeta_{n-1}|)\,\big\vert_{x=0}\\
&=-2|\zeta_{n-1}|^{-1}\zeta\cdot B(y)\zeta,
\qquad (0,y,\xi,\zeta)\in\Sigma.
\end{split}\end{equation}

In order to better understand the generalized broken bicharacteristics
for $\Box$,
we divide $\dot\Sigma$ into two subsets.
We thus define the {\em glancing set} $\cG$
as the set of points in $\dot\Sigma$ whose preimage under
$\hat\pi=\pi|_{\Sigma}$
consists of a single point, and define the {\em hyperbolic set} $\cH$
as its complement
in $\dot\Sigma$. Thus, $\Tb^*_{X^\circ} X\cap\dot\Sigma\subset\cG$
as $\pi$ is a diffeomorphism on $T^*_{X^\circ}X$, while
$q\in\dot\Sigma\cap\Tb^*_YX$ lies in $\cG$ if and only
if on $\hat\pi^{-1}(\{q\})$, $\xi=0$.
More explicitly, with the notation of
\eqref{eq:dot-Sigma-coords2},
\begin{equation}\begin{split}\label{eq:H-G-exp2}
&\cG\cap\cU_0\cap\Tb^*_Y X=\{(0,y,0,\zetab):\ \zetab \cdot B(y)\zetab=0,\ \zetab\neq 0\},\\
&\cH\cap\cU_0\cap\Tb^*_Y X=\{(0,y,0,\zetab):\ \zetab \cdot B(y)\zetab>0,\ \zetab\neq 0\}.
\end{split}\end{equation}
Thus, $\cG$ corresponds to generalized broken bicharacteristics which
are tangent to $Y$ in view of the vanishing of $\xi$ at $\hat\pi^{-1}(\cG)$
(recall that the $\pa_x$ component of $\sH_p$ is $-2\xi$), while
$\cH$ corresponds to generalized broken bicharacteristics which
are normal to $Y$. Note that if $Y$ is one-dimensional (hence $X$ is
2-dimensional), then
$\zetab\cdot B(y)\zetab=0$ necessarily implies $\zetab=0$, so in fact
$\cG\cap\Tb^*_YX=\emptyset$,
hence there are no glancing rays.

We next make the role of $\cG$ and $\cH$ more explicit, which
explains the relevant phenomena better. A
characterization of \GBBsp, which is equivalent to
Definition~\ref{def:gen-br-bich}, is

\begin{lemma}(See the discussion in \cite[Section~1]{Vasy:Propagation-EDP}
after the statement of Definition~1.1.)\label{lemma:gen-br-bichar}
A continuous map
$\gamma:I\to\dot\Sigma$,
where $I\subset\Real$ is an interval, is a \GBBsp (in the analytic sense
that we use here) if and only if it satisfies
the following requirements:

\begin{enumerate}
\item
If $q_0=\gamma(s_0)\in\cG$
then for all
$f\in\Cinf(\Tb^*X)$,
\begin{equation}\label{eq:cG-bich}
\frac{d}{ds}(f\circ \gamma)(s_0)=\sH_p (\pi^*f)(\tilde q_0),\ \tilde q_0=\hat\pi^{-1}(q_0).
\end{equation}

\item
If $q_0=\gamma(s_0)\in\cH$ then
there exists $\ep>0$ such that
\begin{equation}
s\in I,\ 0<|s-s_0|<\ep\Rightarrow\gamma(t)\notin \Tb^*_{Y} X.
\end{equation}

\end{enumerate}
\end{lemma}

The idea of the proof of this lemma is that at $\cG$, the requirement in (i) is
equivalent to Definition~\ref{def:gen-br-bich} since $\hat\pi^{-1}(q_0)$ contains
a single point. On the other hand, at $\cH$, the requirement in (ii) follows from
Definition~\ref{def:gen-br-bich} applied to the functions $f=\pm\xib$, using
\eqref{eq:pi-H_p-xib-at-bdy}, to conclude that
$\xib$ is strictly decreasing at $\cH$ along \GBB.
Since on $\dot\Sigma\cap\{x=0\}$, one has $\xib=0$, 
for a \GBBsp
$\gamma$ through $\gamma(s_0)=q_0\in\cH$, on a punctured neighborhood
of $s_0$, $\xib(\gamma(s))\neq 0$, so $\gamma(s)\notin\Tb^*_Y X$ (as
$\gamma(s)\in\dot\Sigma$). For the converse direction at $\cH$ we refer
to \cite{Lebeau:Propagation}; see \cite[Section~1]{Vasy:Propagation-EDP} for details.

\section{Microlocal elliptic regularity}
\label{sec:elliptic}

We first note the form of $\Box$ with commutator calculations in mind.
Note that rather than thinking of the tangential terms, $xD_y$, as
`too degenerate', we think of $xD_x$ as `too singular' in that it
causes the failure of $\Box$ to lie in $x^2\Diffb^2(X)$. This makes
the calculations rather analogous to the conformal case, and also
it facilitates the use of the symbolic machinery for b-ps.d.o's.

\begin{prop}\label{prop:Box-form}
On a collar neighborhood of $Y$, $\Box$ has the form
\begin{equation}\label{eq:Box-form}
-(xD_x)^* \alpha (xD_x)+(xD_x)^* M'+M''(xD_x)+ \tilde P,
\end{equation}
with
\begin{equation*}\begin{split}
&\alpha-1\in x\CI(X),\ M',M''\in x^2\Diffb^1(X)\subset x\Diffz^1(X),\\
& \tilde P\in x^2\Diffb^2(X),
\ \tilde P- x^2\Box_h\in x^3\Diffb^2(X)\subset x\Diffz^2(X),
\end{split}\end{equation*}
where $\Box_h$ is the d'Alembertian
of the conformal metric on the boundary (extended to a neighborhood
of $Y$ using the collar structure).
\end{prop}

\begin{proof}
Writing the coordinates as $(z_1,\ldots,z_n)$,
the operator $\Box_g$ is given by
$$
\Box_g=\sum_{ij} D_{z_i}^* G_{ij} D_{z_j},
$$
with adjoints taken with respect to $dg=|\det g|^{1/2}\,|dz_1\ldots dz_n|$.
With $z_n=x$, $z_j=y_j$ for $j=1,\ldots,n-1$, this can be rewritten as
\begin{equation*}\begin{split}
\Box_g&=\sum_{ij} (xD_{z_i})^* \hat G_{ij} (xD_{z_j})\\
&=(xD_x)^*\hat G_{nn} (xD_x)+\sum_{j=1}^{n-1} (xD_x)^* \hat G_{nj}(xD_{y_j})
+\sum_{j=1}^{n-1} (xD_{y_j})^* \hat G_{jn}(xD_{y_j})\\
&\qquad\qquad\qquad+\sum_{i,j=1}^{n-1} (xD_{y_i})^*\hat G_{ij}(xD_{y_j}).
\end{split}\end{equation*}
As $\hat G_{nn}+1\in x\CI(X)$, we may take $\alpha=-\hat G_{nn}$
and conclude that $\alpha-1\in x\CI(X)$. As
$\hat G_{jn},\hat G_{nj}\in x\CI(X)$, taking
$M'=\sum_{j=1}^{n-1}\hat G_{nj}(xD_{y_j})$ and
$M''=\sum_{j=1}^{n-1}(xD_{y_j})^*\hat G_{jn}$, $M',M''\in x^2\Diffb^1(X)$
follow. Finally,
$$
\tilde P=\sum_{ij=1}^{n-1} (xD_{y_i})^*\hat G_{ij}(xD_{y_j})
\in x^2\Diffb^2(X),
$$
and modulo $x^3\Diffb^2(X)$, we can pull out
the factors of $x$ and restrict $\hat G_{ij}$ to $Y$, so
$\tilde P$ differs from $x^2\Box_h=x^2\sum D_{y_i}^*h_{ij} D_{y_j}$
by an element of $x^3\Diffb^2(X)$, completing the proof.
\end{proof}

We next state the lemma regarding Dirichlet form which is of fundamental
use in both the elliptic and hyperbolic/glancing estimates.
Below the main assumption is that $P=\Box_g+\lambda$, with $\Box_g$
as in \eqref{eq:Box-form}. We first recall the notation for
local norms:

\begin{rem}\label{rem:localize}
Since $X$ is non-compact and our results are microlocal, we may always
fix a compact set $\tilde K\subset X$ and assume that all ps.d.o's
have Schwartz kernel supported in $\tilde K\times\tilde K$. We
also let $\tilde U$ be a neighborhood of $\tilde K$ in $X$ such that
$\tilde U$ has compact closure, and use the $H^1_0(\tilde U)$ norm
in place of the $H^1_0(X)$ norm to accommodate $u\in H^1_{0,\loc}(X)$.
(We may instead take $\phi\in\Cinf_\compl(\tilde U)$ identically $1$ in a
neighborhood of $\tilde K$, and use $\|\phi u\|_{H^1_0(X)}$.)
Below we use the notation $\|.\|_{H^1_{0,\loc}(X)}$ for
$\|.\|_{H^1_0(\tilde U)}$ to avoid having to specify $\tilde U$.
We also use $\|v\|_{H^{-1}_{0,\loc}(X)}$ for
$\|\phi v\|_{H^{-1}_0(X)}$.
\end{rem}

\begin{lemma}(cf.\ \cite[Lemma~4.2]{Vasy:Propagation-Wave})
\label{lemma:Dirichlet-form}
Suppose that $K\subset\Sb^*X$ is compact, $U\subset\Sb^*X$ is open,
$K\subset U$.
Suppose that $\cA=\{A_r:\ r\in(0,1]\}$ is a bounded
family of ps.d.o's in $\Psibc^s(X)$ with $\WFb'(\cA)\subset K$, and
with $A_r\in\Psib^{s-1}(X)$ for $r\in (0,1]$.
Then there are $G\in\Psib^{s-1/2}(X)$, $\tilde G\in\Psib^{s+1/2}(X)$
with $\WFb'(G),\WFb'(\tilde G)\subset U$
and $C_0>0$ such that for $r\in(0,1]$, $u\in H^{1,k}_{0,\bl,\loc}(X)$
(here $k\leq 0$)
with $\WFb^{1,s-1/2}(u)
\cap U=\emptyset$, $\WFb^{-1,s+1/2}(Pu)\cap U=\emptyset$, we have
\begin{equation*}\begin{split}
&|\langle d A_r u, dA_r u\rangle_G+\lambda\|A_r u\|^2|\\
&\qquad\leq
C_0(\|u\|^2_{H^{1,k}_{0,\bl,\loc}(X)}+\|Gu\|^2_{H^1_0(X)}
+\|Pu\|^2_{H^{-1,k}_{0,\bl,\loc}(X)}
+\|\tilde G Pu\|^2_{H^{-1}_0(X)}).
\end{split}\end{equation*}
\end{lemma}

\begin{rem}
The point of this lemma is $G$ is $1/2$ order lower ($s-1/2$ vs. $s$)
than the {\em family} $\cA$. We will later take a limit, $r\to 0$,
which gives control of the Dirichlet form evaluated on $A_0u$,
$A_0\in\Psibc^s(X)$, in terms of lower order information.

The role of $A_r$, $r>0$, is to regularize such an argument, i.e.\ to make
sure various terms in a formal computation, in which one uses $A_0$
directly, actually make sense.

The main difference with \cite[Lemma~4.2]{Vasy:Propagation-Wave} is
that $\lambda$ is {\em not negligible}.
\end{rem}

\begin{proof}
Then for $r\in(0,1]$, $A_r u\in H^1_0(X)$, so
\begin{equation*}
\langle d A_r u,dA_r u\rangle+\lambda \|A_r u\|^2
=\langle PA_r u,A_r u\rangle.
\end{equation*}
Here the right hand side is the pairing of $H^{-1}_0(X)$ with $H^1_0(X)$.
Writing $PA_r=A_rP+[P,A_r]$,
the right hand side can be estimated by
\begin{equation}\label{eq:Dirichlet-form-8}
|\langle A_rPu,A_r u\rangle|+|\langle[P,A_r] u,A_r u\rangle|.
\end{equation}
The lemma is thus proved if we show that the first term of
\eqref{eq:Dirichlet-form-8} is bounded by
\begin{equation}\label{eq:Dirichlet-bd}
C_0'(\|u\|^2_{H^{1,k}_{0,\bl,\loc}(X)}+\|Gu\|^2_{H^1_0(X)}
+\|Pu\|^2_{H^{-1,k}_{0,\bl,\loc}(X)}
+\|\tilde G Pu\|^2_{H^{-1}_0(X)}),
\end{equation}
the second term
is bounded by $C_0''(\|u\|^2_{H^{1,k}_{0,\bl,\loc}(X)}+\|Gu\|^2_{H^1_0(X)})$.
{\em (Recall that the `local' norms were
defined in Remark~\ref{rem:localize}.)}

The first term is straightforward to estimate.
Let $\Lambda\in\Psib^{-1/2}(X)$ be elliptic with
$\Lambda^-\in\Psib^{1/2}(X)$ a parametrix, so
\begin{equation*}
E=\Lambda\Lambda^--\Id,E'=\Lambda^-\Lambda-\Id\in\Psib^{-\infty}(X).
\end{equation*}
Then
\begin{equation*}\begin{split}
\langle A_rPu,A_r u\rangle&=\langle(\Lambda\Lambda^--E)A_r Pu,
A_r u\rangle\\
&=\langle\Lambda^-A_r Pu,\Lambda^* A_r u\rangle
-\langle A_r Pu,E^*A_r u\rangle.
\end{split}\end{equation*}
Since $\Lambda^- A_r$ is uniformly bounded in $\Psibc^{s+1/2}(X)$,
and $\Lambda^* A_r$ is uniformly bounded in $\Psibc^{s-1/2}(X)$,
$\langle\Lambda^-A_r Pu,\Lambda^* A_r \rangle$
is uniformly bounded, with a bound like \eqref{eq:Dirichlet-bd}
using Cauchy-Schwartz and Lemma~\ref{lemma:WFb-mic-q}.
Indeed, by Lemma~\ref{lemma:WFb-mic-q},
choosing any $G\in\Psib^{s-1/2}(X)$ which is elliptic on
$K$, there is a constant $C_1>0$ such that
\begin{equation*}
\|\Lambda^*A_r u\|^2_{H^1_0(X)}\leq
C_1(\|u\|^2_{H^{1,k}_{0,\bl,\loc}(X)}+\|Gu\|^2_{H^1_0(X)}).
\end{equation*}
Similarly, by Lemma~\ref{lemma:WFb-mic-q} and its analogue for
$\WFb^{-1,s}$,
choosing any $\tilde G\in\Psib^{s+1/2}(X)$ which is elliptic on
$K$, there is a constant $C_1'>0$ such that
$\|\Lambda^- A_rPu\|^2_{H^{-1}_0(X)}\leq
C_1'(\|Pu\|^2_{H^{-1,k}_{0,\bl,\loc}(X)}+\|\tilde GPu\|^2_{H^{-1}_0(X)})$.
Combining these gives, with $C_0'=C_1+C_1'$,
\begin{equation*}\begin{split}
|\langle &\Lambda^-A_r Pu,\Lambda^* A_r u\rangle|\leq\|\Lambda^-A_r Pu\|
\,\|\Lambda^*A_r u\|\leq \|\Lambda^-A_r Pu\|^2+\|\Lambda^*A_r u\|^2\\
&\leq C_0'(\|u\|^2_{H^{1,k}_{0,\bl,\loc}(X)}+\|Gu\|^2_{H^1_0(X)}
+\|Pu\|^2_{H^{-1,k}_{0,\bl,\loc}(X)}+\|\tilde GPu\|^2_{H^{-1}_0(X)}),
\end{split}\end{equation*}
as desired.

A similar argument, using that
$A_r$ is uniformly bounded in $\Psibc^{s+1/2}(X)$ (in fact in
$\Psibc^{s}(X)$),
and $E^* A_r$ is uniformly bounded in $\Psibc^{s-1/2}(X)$
(in fact in $\Psibc^{-\infty}(X)$), shows that
$\langle A_r Pu,E^* A_r u\rangle$
is uniformly bounded.

Now we turn to the second term in \eqref{eq:Dirichlet-form-8}, whose
uniform boundedness is a direct consequence of
Lemma~\ref{lemma:b-comm-improve} and Corollary~\ref{cor:WFb-mic-q}.
Indeed, by Lemma~\ref{lemma:b-comm-improve},
$[P,A_r]$ is a bounded family in $\Diffz^2\Psibc^{s-1}(X)$, hence
$A_r^*[P,A_r]$ is a bounded family in $\Diffz^2\Psibc^{2s-1}(X)$.
Then one can apply Corollary~\ref{cor:WFb-mic-q} to conclude that
$$
\langle A_r^*[P,A_r]u,u\rangle\leq C'(\|u\|^2_{H^{1,k}_{0,\bl,\loc}(X)}+\|Gu\|^2_{H^1(X)}),$$
proving the lemma.
\end{proof}

A more precise version, in terms of requirements on $Pu$, is the following.
Here, as in Section~\ref{sec:Poincare}, we fix
a positive definite inner product on the fibers of $\zT^*X$
(i.e.\ a Riemannian 0-metric) to compute
$\|d v\|^2_{L^2(X;\zT^*X)}$; as $v$ has support in a compact set below, the
choice of the inner product is irrelevant.

\begin{lemma}\label{lemma:Dirichlet-form-2}
(cf.\ \cite[Lemma~4.4]{Vasy:Propagation-Wave})
Suppose that $K\subset\Sb^*X$ is compact, $U\subset\Sb^*X$ is open,
$K\subset U$.
Suppose that $\cA=\{A_r:\ r\in(0,1]\}$ is a bounded
family of ps.d.o's in $\Psibc^s(X)$ with $\WFb'(\cA)\subset K$, and
with $A_r\in\Psib^{s-1}(X)$ for $r\in (0,1]$.
Then there are $G\in\Psib^{s-1/2}(X)$, $\tilde G\in\Psib^{s}(X)$
with $\WFb'(G),\WFb'(\tilde G)\subset U$
and $C_0>0$ such that for $\ep>0$, $r\in(0,1]$, $u\in H^{1,k}_{0,\bl,\loc}(X)$
($k\leq 0$)
with $\WFb^{1,s-1/2}(u)
\cap U=\emptyset$, $\WFb^{-1,s}(Pu)\cap U=\emptyset$, we have
\begin{equation*}\begin{split}
|\langle dA_r u,&dA_r u\rangle_G+\lambda\|A_r u\|^2|\\
\leq
\ep&\|d A_r u\|^2_{L^2(X;\zT^*X)}
+C_0(\|u\|^2_{H^{1,k}_{0,\bl,\loc}(X)}+\|Gu\|^2_{H^1_0(X)}\\
&+\ep^{-1}\|Pu\|^2_{H^{-1,k}_{0,\bl,\loc}(X)}
+\ep^{-1}\|\tilde G Pu\|^2_{H^{-1}_0(X)}).
\end{split}\end{equation*}
\end{lemma}

\begin{rem}
The point of this lemma is that on the one hand
the new term $\ep\|d A_r u\|^2$ can be absorbed
in the left hand side in the elliptic region, hence is negligible,
on the other hand, there is a gain in the
order of $\tilde G$ ($s$, versus $s+1/2$ in the previous lemma).
\end{rem}

\begin{proof}
We only need to modify the previous proof slightly. Thus, we need to
estimate the term $|\langle A_rPu,A_r u\rangle|$ in
\eqref{eq:Dirichlet-form-8} differently, namely
\begin{equation*}
|\langle A_rPu,A_r u\rangle|\leq \|A_r Pu\|_{H^{-1}_0(X)}
\|A_r u\|_{H^1_0(X)}\leq \tilde\ep\|A_r u\|^2_{H^1_0(X)}+\tilde\ep^{-1}
\|A_r Pu\|^2_{H^{-1}_0(X)}.
\end{equation*}
Now the lemma follows by using
Lemma~\ref{lemma:WFb-mic-q} and the remark following it,
namely
choosing any $\tilde G\in\Psib^{s}(X)$ which is elliptic on
$K$, there is a constant $C_1'>0$ such that
$\|A_rPu\|^2_{H^{-1}_0(X)}\leq
C_1'(\|Pu\|^2_{H^{-1,k}_{0,\bl,\loc}(X)}+\|\tilde G Pu\|^2_{H^{-1}_0(X)})$,
then using the Poincar\'e inequality to estimate
$\|A_r u\|_{H^1_0(X)}$ by $C_2\|dA_r u\|_{L^2(X)}$,
and finishing the proof exactly as for Lemma~\ref{lemma:Dirichlet-form}.
\end{proof}

We next state microlocal elliptic regularity. Note that for this result
the restrictions on $\lambda\in\Cx$ are weak (only a half-line is
disallowed), but
on the other hand, a solution $u$ satisfying our
hypotheses may not exist for values of $\lambda$ when $\lambda\notin
(-\infty,(n-1)^2/4)$.

\begin{prop}(Microlocal elliptic regularity.)\label{prop:elliptic}
Suppose that $P=\Box+\lambda$, $\lambda\in\Cx
\setminus[(n-1)^2/4,\infty)$ and $m\in\RR$ or $m=\infty$.
Suppose $u\in
H^{1,k}_{0,\bl,\loc}(X)$ for some $k\leq 0$. Then
\begin{equation*}
\WFb^{1,m}(u)\setminus\dot\Sigma\subset \WFb^{-1,m}(Pu).
\end{equation*}
\end{prop}

\begin{proof}
We first prove a slightly weaker result in which
$\WFb^{-1,m}(Pu)$ is replaced by $\WFb^{-1,m+1/2}(Pu)$ -- we rely on
Lemma~\ref{lemma:Dirichlet-form}. We then prove the
original statement using Lemma~\ref{lemma:Dirichlet-form-2}.

Suppose that $q\in\Tb^*_YX\setminus\dot\Sigma$.
We may assume iteratively that $q\notin\WFb^{1,s-1/2}(u)$;
we need to prove then that $q\notin\WFb^{1,s}(u)$ provided $s\leq m+1/2$
(note that the inductive
hypothesis holds for $s=k+1/2$ since $u\in H^{1,k}_{0,\bl,\loc}(X)$).
We use local coordinates $(x,y)$ as in Section~\ref{sec:GBB}, centered
so that $q\in\Tb^*_{(0,0)}X$, arranging that
\eqref{eq:B-orth-0} holds, and further group the variables
as $y=(y',y_{n-1})$, and hence the b-dual variables $(\zetab',\zetab_{n-1})$.
We denote the Euclidean norm by $|\zetab'|$.

Let $A\in\Psib^{s}(X)$ be such that
$$
\WFb'(A)\cap \WFb^{1,s-1/2}(u)=\emptyset,\ \WFb'(A)\cap\WFb^{1,s+1/2}(Pu)
=\emptyset,
$$
and have $\WFb'(A)$ in a small
conic neighborhood $U$ of $q$ so that for a suitable $C>0$ or $\ep>0$, in $U$
\begin{enumerate}
\item
$\zetab_{n-1}^2<C \xib^2$ if $\xib(q)\neq 0$,
\item
$|\xib|<\ep|\zetab|$ for all $j$, and
$\frac{|\zetab'|}{|\zetab_{n-1}|}>1+\ep$, if $\xib(q)=0$ and $\zetab(q)\cdot
B(y(q))\zetab(q)<0$.
\end{enumerate}
Let $\Lambda_r\in\Psib^{-2}(X)$ for $r>0$, such that $\cL=\{\Lambda_r:
\ r\in(0,1]\}$ is a bounded family in $\Psib^0(X)$, and $\Lambda_r\to\Id$
as $r\to 0$ in $\Psib^{\tilde\ep}(X)$, $\tilde\ep>0$,
e.g.\ the symbol of $\Lambda_r$ could be taken as
$(1+r(|\zetab|^2+|\xib|^2))^{-1}$. Let $A_r=\Lambda_r A$.
Let $a$ be the symbol of $A$, and
let $A_r$ have symbol $(1+r(|\zetab|^2+|\xib|^2))^{-1}a$, $r>0$,
so $A_r\in
\Psib^{s-2}(X)$ for $r>0$, and $A_r$ is uniformly bounded in
$\Psibc^{s}(X)$, $A_r\to A$ in $\Psibc^{s+\tilde\ep}(X)$.

By Lemma~\ref{lemma:Dirichlet-form},
\begin{equation*}
\langle dA_r u,dA_r u\rangle_G+\lambda\|A_r u\|^2
\end{equation*}
is uniformly bounded for $r\in(0,1]$, so
\begin{equation*}
\langle dA_r u,dA_r u\rangle_G+\re\lambda\|A_r u\|^2\Mand
\im\lambda\|A_r u\|^2
\end{equation*}
are uniformly bounded.
If $\im\lambda\neq 0$, then taking the imaginary part at once
shows that $\|A_r u\|$ is in fact uniformly bounded.
On the other hand, whether $\im\lambda=0$ or not,
\begin{equation*}\begin{split}
\langle dA_r u,dA_r u\rangle_G
=&\int_X A(x,y)xD_xA_r u \,\overline{xD_x A_r u}\,dg\\
&+\int_X \sum B_{ij}(x,y) xD_{y_i}A_r u \,\overline{xD_{y_j} A_r u}\,dg\\
&+\int_X \sum C_j(x,y) xD_x A_r u \,\overline{x D_{y_j} A_r u}\,dg\\
&+\int_X \sum C_j(x,y) xD_{y_j}A_r u \,\overline{x D_x A_r u}\,dg.
\end{split}\end{equation*}
Using that $A(x,y)=-1+x A'(x,y)+\sum (y_j-y_j(q))A_j(x,y)$, we see that
if $A_r$ is supported in $x<\delta$, $|y_j-y_j(q)|<\delta$ for all $j$,
then for some $C>0$ (independent of $A_r$),
\begin{equation}\begin{split}\label{eq:elliptic-28}
&\left|\int_X A(x,y)\,xD_xA_r u \,\overline{xD_x A_r u}\,dg-
\int_X A(0,y(q))\,xD_x A_r u \,\overline{xD_x A_r u}\,dg\right|\\
&\qquad\leq C\delta\|xD_xA_r u\|^2,
\end{split}\end{equation}
with analogous estimates\footnote{Recall that $C_j(0,y)=0$ and
$B_{ij}(0,y(q))=0$ if $i\neq j$, $B_{ij}(0,y(q))=1$ if $i=j=n-1$,
$B_{ij}(0,y(q))=-1$ if $i=j\neq n-1$.}
for $B_{ij}(x,y)-B_{ij}(0,y(q))$ and
for $C_{j}(x,y)$.
Thus,
there exists $\tilde C>0$ and $\delta_0>0$ such that if $\delta<\delta_0$
and $A$ is supported
in $|x|<\delta$ and $|y-y(q)|<\delta$ then
\begin{equation}\begin{split}\label{eq:ell-32}
&\int_X \Big((1-\tilde C\delta)|xD_x A_r u|^2-\re\lambda|A_r u|^2\Big)\,dg\\
&\qquad\qquad+\sum_{j=1}^{n-2}\int_X\Big((1-\tilde C\delta)\sum_j xD_{y_j} A_r u\,\overline{x D_{y_j} A_r u}\Big)\,dg\\
&\qquad\qquad
-\int_X\Big((1+\tilde C\delta)\sum_j xD_{y_{n-1}} A_r u\,\overline{x D_{y_{n-1}} A_r u}\Big)\,dg\\
\\
&\qquad\leq |\langle dA_r u,dA_r u\rangle_G+\re\lambda \|A_r u\|^2|.
\end{split}\end{equation}

Now we distinguish the cases $\xib(q)=0$ and $\xib(q)\neq 0$.
If $\xib(q)=0$,
we choose $\delta\in(0,\frac{1}{2\tilde C})$, $\delta<\delta_0$, so that
$$
(1-\tilde C\delta)\frac{|\zetab'|^2}{\zetab_{n-1}^2}>1+2\tilde C\delta
$$
on
a neighborhood of $\WFb'(A)$, which is possible in view of (ii) at the
beginning of the proof.
Then the second integral on the left hand side of \eqref{eq:ell-32}
can be written
as $\| BxA_r u\|^2$, with the symbol of $B$ given by
$$
\left((1-\tilde C\delta)|\zetab'|^2
-(1+\tilde C\delta)\zetab_{n-1}^2\right)^{1/2}
$$
(which is $\geq\delta|\zetab_{n-1}|$), modulo a term
\begin{equation*}
\int_X F xA_r u\,\overline{xA_r u}\,dg,\ F\in\Psib^{1}(X).
\end{equation*}
But $A_r^* xFxA_r$ is uniformly bounded in
$x^2\Psibc^{2s+1}(X)\subset\Diffz^2\Psibc^{2s-1}(X)$, so this expression
is uniformly bounded as $r\to 0$ by
Corollary~\ref{cor:WFb-mic-q}. We thus deduce that
\begin{equation*}
\int_X \Big((1-\tilde C\delta)|xD_x A_r u|^2-\re\lambda|A_r u|^2\Big)\,dg+\| BxA_r u\|^2
\end{equation*}
is uniformly bounded as $r\to 0$.

If $\xib(q)\neq 0$, and $A$ is supported in $|x|<\delta$,
\begin{equation*}
\tilde C\delta\int_X \delta^{-2}|x^2D_x A_r u|^2\,dg\leq \tilde C\delta
\int_X |xD_x A_r u|^2\,dg.
\end{equation*}
On the other hand, near $\{q':\ \xib(q')=0\}$,
for $\delta>0$ sufficiently small,
\begin{equation*}
\int_X\left(\frac{\tilde C\delta}{\delta^{2}}
|x^2 D_{x}A_r u|^2-|xD_{y_{n-1}} A_r u|^2\right)\,dg=\|BxA_r u\|^2
+\int_X F xA_r u\,\overline{xA_r u}\,dg,
\end{equation*}
with the symbol of $B$ given by
$(\frac{\tilde C}{\delta}\xib^2-\zetab_{n-1}^2)^{1/2}$
(which does not vanish on $U$ for
$\delta>0$ small), while
$F\in\Psib^{1}(X)$, so the second term on the right hand side
is uniformly bounded as $r\to 0$ just as above.
We thus deduce in this case that
\begin{equation*}
\int_X ((1-2\tilde C\delta)|xD_{x} A_r u|^2\,dg-\re\lambda|A_r u|^2)+\| BxA_r u\|^2
\end{equation*}
is uniformly bounded as $r\to 0$.

If $\im\lambda\neq 0$ then we already saw that $\|A_r u\|_{L^2}$ is
uniformly bounded, so we deduce that
\begin{equation}\label{eq:ell-unif-bd}
A_r u, xD_x A_r u, BxA_r u\ \text{are uniformly bounded in}\ L^2(X).
\end{equation}
If $\im\lambda=0$, but $\lambda<(n-1)^2/4$, then
the Poincar\'e inequality allows us to
reach the same conclusion, since on the one hand in case (ii)
$$
(1-\tilde C\delta)\|xD_x A_r u\|^2-\re\lambda\|A_r u\|^2,
$$
resp.\ in case (i)
$$
(1-2\tilde C\delta)\|xD_x A_r u\|^2-\re\lambda\|A_r u\|^2,
$$
is uniformly bounded, on the other hand by
Proposition~\ref{prop:Poincare-no-weight}, for $\delta>0$ sufficiently small
there exists $c>0$ such that
$$
(1-2\tilde C\delta)\|xD_x A_r u\|^2-\re\lambda\|A_r u\|^2\geq
c(\|xD_x A_r u\|^2+\|A_r u\|^2).
$$
Correspondingly
there are sequences $A_{r_k}u$, $xD_xA_{r_k}u$, $BxA_{r_k}u$,
weakly convergent in $L^2(X)$, and such that $r_k\to 0$, as $k\to\infty$.
Since they converge to $Au$, $xD_xAu$, $BxAu$, respectively, in $\dist(X)$,
we deduce that the weak limits are $Au$, $xD_xAu$, $BxAu$, which therefore
lie in $L^2(X)$. Consequently,
that $q\notin\WFb^{1,s}(u)$, hence proving the
proposition with $\WFb^{-1,m}(Pu)$ replaced by $\WFb^{-1,m+1/2}(Pu)$.

To obtain the optimal result, we note that due to
Lemma~\ref{lemma:Dirichlet-form-2} we still have, for any $\ep>0$,
that
\begin{equation*}\begin{split}
&\langle dA_r u,dA_ru\rangle_G-\ep\|d A_r u\|^2
\end{split}\end{equation*}
is uniformly bounded above for $r\in(0,1]$.
By arguing just as above, with $B$ as above,
for sufficiently small $\ep>0$, the right hand side gives an upper bound for
\begin{equation*}
\int_X \Big((1-2\tilde C\delta-\ep)|xD_x A_r u|^2-\re\lambda|A_r u|^2\Big)\,dg+\| BxA_r u\|^2,
\end{equation*}
which is thus uniformly bounded as $r\to 0$. The proof is then finished
exactly as above.
\end{proof}

We remark that the analogous argument works for the conformally
compact elliptic problem, i.e.\ on asymptotically hyperbolic spaces,
to give that for $\lambda\in\Cx\setminus[(n-1)^2/4,\infty)$,
local solutions of $(\Delta_g-\lambda)u$ are actually conormal to $Y$
provided they lie in $H^1_0(X)$ locally, or indeed in
$H^{1,-\infty}_{0,\bl}(X)$.

\section{Propagation of singularities}\label{sec:prop-sing}
In this section we prove propagation of singularities for $P$ by positive commutator
estimates. We do so by first performing a general commutator calculation
in Proposition~\ref{prop:commutator}, then
using it to prove rough propagation estimates first at hyperbolic, then at
glancing points, in Propositions~\ref{prop:normal-prop}, resp.\ \ref{prop:tgt-prop}.
An argument originally due to Melrose and Sj\"ostrand \cite{Melrose-Sjostrand:I}
then proves the main theorems, Theorems~\ref{thm:prop-sing} and
\ref{thm:prop-sing-im}. Finally we discuss consequences of these results.

We first describe the form of commutators of $P$ with $\Psib(X)$. We state this
as an analogue of \cite[Proposition~3.10]{Vasy:Maxwell}, and later in the
section we follow the structure of \cite{Vasy:Maxwell} as well.
Given Proposition~\ref{prop:commutator} below,
the proof of propagation of singularities proceeds
with the same commutant construction as in \cite{Vasy:Propagation-Wave},
see also \cite{Vasy:Propagation-Wave-Correction}. Although it is in
a setting that is more complicated in some ways, since it deals with the
equation on differentials forms, we follow the structure of
\cite{Vasy:Maxwell} as it was written in a
more systematic way than \cite{Vasy:Propagation-Wave}.
Recall from the introduction
that $\xib$ is the b-dual variable of $x$,
$\hat\xib=\xib/|\zetab_{n-1}|$.

\begin{prop}\label{prop:commutator}
Suppose $\cA=\{A_r:\ r\in(0,1]\}$ is a family of operators
$A_r\in\Psib^0(X)$ uniformly bounded in $\Psibc^{s+1/2}(X)$, of the
form $A_r=A\Lambda_r$, $A\in\Psib^0(X)$, $a=\sigma_{\bl,0}(A)$,
$w_r=\sigma_{\bl,s+1/2}(\Lambda_r)$. Then
\begin{equation}\label{eq:Box-commutator-form}
\imath [A^*_rA_r,\Box]=(xD_x)^* C_r^\sharp(xD_x)+(xD_x)^* xC_r'+xC_r''(xD_x)
+x^2C_r^\flat,
\end{equation}
with
$$
C_r^\sharp\in L^\infty((0,1];\Psibc^{2s}(X)),
\ C'_r, C''_r\in L^\infty((0,1];\Psibc^{2s+1}(X)),
\ C_r^\flat\in \Psibc^{2s+2}(X),
$$
and
\begin{equation*}\begin{split}
&\sigma_{\bl,2s}(C^\sharp_r)=2w_r^2 a(V^\sharp a+a\tilde c^\sharp_r),\\
&\sigma_{\bl,2s+1}(C'_r)=
\sigma_{\bl,2s+1}(C''_r)=2w_r^2 a(V' a+a\tilde c'_r),\\
&\sigma_{\bl,2s+2}(C_r^\flat)=2w_r^2 a(V^\flat a+a\tilde c_r^\flat),
\end{split}\end{equation*}
with $\tilde c^\sharp_r,\tilde c'_r,\tilde c_r^\flat$ uniformly bounded
in $S^{-1},S^0,S^1$ respectively,
$V^\sharp,V',V^\flat$ smooth and homogeneous of degree $-1,0,1$ respectively
on $\Tb^*X\setminus o$, $V^\sharp|_Y$ and $V'|_Y$ annihilate $\xib$
and
\begin{equation}\label{eq:V-form}
V^\flat|_Y=2h\pa_{\xib}-\sH_h.
\end{equation}
\end{prop}

\begin{proof}
We start by observing that, in Proposition~\ref{prop:Box-form},
$\Box$ is decomposed into a sum of products of weighted b-operators,
so analogously expanding the commutator,
all calculations can be done in $x^l\Psib(X)$ for various values of $l$.
In particular, keeping in mind Lemma~\ref{lemma:xD_x-improvement} (which
gives the additional order of decay),
$$
\imath[A_r^*A_r,xD_x],\imath[A_r^*A_r,(xD_x)^*]
\in L^\infty((0,1]_r,x\Psib^{2s+1}(X)),
$$
with principal symbol $-2 w_r^2 a x\pa_x a-2a^2 w_r (x\pa_x w_r)$.
By this observation, all commutators
with factors of $xD_x$ or $(xD_x)^*$
in \eqref{eq:Box-form} can be absorbed into the
`next term' of \eqref{eq:Box-commutator-form},
so $[A_r^*A_r,(xD_x)^*] \alpha (xD_x)$ is absorbed into 
$xC_r''(xD_x)$, $(xD_x)\alpha[A_r^*A_r,xD_x]$ is absorbed into
$(xD_x)^* xC_r'$, $[A_r^*A_r,(xD_x)^*] M'$ and $M''[A_r^*A_r,(xD_x)]$
are absorbed into $x^2C_r^\flat$. The principal symbols of these terms are of 
the desired form, i.e.\ after factoring out $2w_r^2 a$,
they are the result of a vector field applied
to $a$ plus a multiple of $a$, and this vector field is $-\alpha \pa_x$
in the case
of the first two terms (thus annihilates $\xib$), $-mx^{-1}\pa_x$ in the case
of the last two terms, which in view of $m=\sigma_{\bl,1}(M')=
\sigma_{\bl,1}(M'')\in x^2 S^1$, shows that it actually does not affect
$V^\flat|_Y$.

Next, $\imath(xD_x)^*[A_r^*A_r,\alpha](xD_x)$ can be absorbed
into (in fact taken equal to)
$(xD_x)^* C_r^\sharp(xD_x)$ with principal symbol of $C_r^\sharp$
given by
$$
-(\pa_y \alpha) \pa_\zetab (a^2 w_r^2)-(x\pa_x\alpha) \pa_\xib(a^2 w_r^2)
$$
in local coordinates,
thus again is of the desired form since the $\pa_\xib$ term has a vanishing
factor of $x$ preceding it.

Since $[A_r^*A_r,M'],[A_r^*A_r,M'']$ are uniformly bounded
in $x^2\Psib^{2s+1}(X)$, the corresponding commutators can be absorbed
into $(xD_x)^* xC_r'$, resp.\ $xC_r''(xD_x)$, without affecting the
principal symbols of $C_r'$ and $C_r''$ at $Y$, and possessing the
desired form.

Next, $\tilde P=x^2\Box_h+R$, $R\in x^3\Diffb^2(X)$, so
$[A_r^*A_r,R]$ is uniformly bounded in $x^3\Psib^{2s+2}(X)$, and thus
can be absorbed into $C_r^\flat$ without affecting its principal symbol at $Y$
and possessing the desired form. Finally, $\imath[A_r^*A_r,x^2\Box_h]
\in x^2\Psib^{2s+2}(X)$
has principal symbol $\pa_{\xib}(a^2 w_r^2) 2x^2h-x^2\sH_h(a^2 w_r^2)$,
thus can be absorbed into $C_r^\flat$
yielding the stated principal symbol at $Y$.
\end{proof}

We start our propagation results with the
propagation estimate at hyperbolic points.

\begin{prop}\label{prop:normal-prop}(Normal, or hyperbolic, propagation.)
Suppose that $P=\Box_g+\lambda$, $\lambda\in\Cx\setminus [(n-1)^2/4,\infty)$.
Let $q_0=(0,y_0,0,\zetab_0)\in\cH\cap \Tb^*_{Y} X$,
and let
$$
\eta=-\xibh
$$
be the
function defined
in the local coordinates discussed above, and suppose that $u\in
H^{1,k}_{0,\bl,\loc}(X)$ for some $k\leq 0$, $q_0\notin\WFb^{-1,\infty}(f)$,
$f=P u$.
If $\im\lambda\leq 0$ and
there exists a conic neighborhood $U$ of $q_0$ in $\Tb^*X\setminus o$
such that
\begin{equation}\begin{split}\label{eq:prop-9a}
q\in U\Mand \eta(q)<0\Rightarrow q\notin\WFb^{1,\infty}(u)
\end{split}\end{equation}
then $q_0\notin\WFb^{1,\infty}(u)$.

In fact, if the wave front set assumptions are relaxed to
$q_0\notin\WFb^{-1,s+1}(f)$ ($f=P u$)
and the existence of a conic neighborhood $U$ of
$q_0$ in $\Tb^*X\setminus o$ such that
\begin{equation}\begin{split}\label{eq:prop-9a-s}
q\in U\Mand \eta(q)<0\Rightarrow q\notin\WFb^{1,s}(u),
\end{split}\end{equation}
then we can still conclude that $q_0\notin\WFb^{1,s}(u)$.
\end{prop}

\begin{rem}\label{rem:normal-remark}
As follows immediately from the proof given below,
in \eqref{eq:prop-9a} and \eqref{eq:prop-9a-s},
one can replace $\eta(q)<0$ by $\eta(q)>0$, i.e.\ one has the conclusion
for either direction (backward or forward) of propagation, {\em provided
one also switches the sign of $\im\lambda$, when it is non-zero,
i.e.\ the assumption should be $\im\lambda\geq 0$}.
In particular, if $\im\lambda=0$, one obtains propagation estimates
both along increasing and along decreasing $\eta$.

Note that $\eta$
is {\em increasing} along the \GBBsp of $\Box_{\hat g}$ by
\eqref{eq:pi-H_p-xibh-at-bdy}. Thus, the hypothesis
region, $\{q\in U:\ \eta(q)<0\}$, on the left hand side of \eqref{eq:prop-9a},
is {\em backwards} from $q_0$, so this proposition, roughly speaking, propagates
regularity {\em forwards}.

Moreover, every neighborhood $U$ of
$q_0=(y_0,\zetab_0)\in\cH\cap \Tb^*_YX$ in $\dot\Sigma$
contains an open set of the form
\begin{equation}\label{eq:prop-rem-9b}
\{q:\ |x(q)|^2+|y(q)-y_0|^2+|\zetabh(q)
-\zetabh_0|^2<\delta\},
\end{equation}
see \cite[Equation~(5.1)]{Vasy:Propagation-Wave}.
Note also that
\eqref{eq:prop-9a} implies the same statement with $U$ replaced by
any smaller neighborhood of $q_0$; in particular, for
the set \eqref{eq:prop-rem-9b}, provided that $\delta$ is sufficiently small.
We can also assume by the same observation that $\WFb^{-1,s+1}(Pu)\cap U
=\emptyset$. Furthermore, we
can also arrange that $h(x,y,\xib,\zetab)>|(\xib,\zetab)|^2
|\zetab_0|^{-2}
h(q_0)/2$ on $U$ since
$\zetab_0\cdot B(y_0)\zetab_0=h(0,y_0,0,\zetab_0)>0$. We write
$$
\hat h=|\zetab_{n-1}|^{-2}h=|\zetab_{n-1}|^{-2}\zetab\cdot B(y)\zetab
$$
for the rehomogenized
version of $h$, which is thus homogeneous of degree zero and bounded below
by a positive constant on $U$.
\end{rem}

\begin{proof}
This proposition is the analogue of
Proposition~6.2 in \cite{Vasy:Propagation-Wave}, and as
the argument is similar, we mainly emphasize the differences. These enter
by virtue of $\lambda$ not being negligible and the use of the
Poincar\'e inequality.
In \cite{Vasy:Propagation-Wave}, one uses a commutant
$A\in\Psib^0(X)$ and weights $\Lambda_r\in\Psib^0(X)$, $r\in(0,1)$,
uniformly bounded in $\Psibc^{s+1/2}(X)$, $A_r=A\Lambda_r$,
in order to obtain the
propagation of $\WFb^{1,s}(u)$ with the notation of that paper, whose
analogue is $\WFb^{1,s}(u)$ here (the difference is the space
relative to which one obtains b-regularity: $H^1(X)$ in the previous
paper, the zero-Sobolev space $H^1_0(X)$ here).
One can use {\em exactly the same} commutant as in
\cite{Vasy:Propagation-Wave}.
Then Proposition~\ref{prop:commutator} lets one calculate
$\imath[A^*_rA_r,P]$ to obtain a completely analogous expression
to Equation~(6.18) of \cite{Vasy:Propagation-Wave} in the hyperbolic
case.
We also refer the reader to \cite{Vasy:Maxwell} because,
although it studies a more delicate problem, namely natural boundary
conditions (which are not scalar), the main ingredient of the proof,
the commutator calculation, is written up exactly as above in
Proposition~\ref{prop:commutator}, see \cite[Proposition~3.10]{Vasy:Maxwell}
and the way it is used subsequently in Propositions 5.1 there.

As in \cite[Proof of Proposition~5.1]{Vasy:Maxwell},
we first construct a commutant by defining its scalar principal symbol, $a$.
This completely follows the scalar case, see \cite[Proof of
Proposition~6.2]{Vasy:Propagation-Wave}. Next we show how to obtain
the desired estimate.

So, as in \cite[Proof of
Proposition~6.2]{Vasy:Propagation-Wave}, let
\begin{equation}\label{eq:prop-omega-def}
\omega(q)=|x(q)|^2+|y(q)-y_0|^2+|\zetabh(q)-\zetabh_0|^2,
\end{equation}
with $|.|$ denoting the Euclidean norm.
For $\ep>0$, $\delta>0$, with other restrictions to be imposed later on,
let
\begin{equation}\label{eq:normal-phi-def}
\phi=\eta+\frac{1}{\ep^2\delta}\omega,
\end{equation}
Let $\chi_0\in\Cinf(\Real)$ be equal to $0$ on $(-\infty,0]$ and
$\chi_0(t)=\exp(-1/t)$ for $t>0$. Thus,
$t^2\chi_0'(t)=\chi_0(t)$ for $t\in\RR$.
Let $\chi_1\in\Cinf(\Real)$ be $0$
on $(-\infty,0]$, $1$ on $[1,\infty)$, with $\chi_1'\geq 0$ satisfying
$\chi_1'\in\Cinf_\compl((0,1))$. Finally, let $\chi_2\in\Cinf_\compl(\Real)$
be supported in $[-2c_1,2c_1]$, identically $1$ on $[-c_1,c_1]$,
where $c_1$ is such that
$|\xibh|^2<c_1/2$ in $\dot\Sigma\cap U$. Thus,
$\chi_2(|\xibh|^2)$ is a cutoff in $|\xibh|$, with its
support properties ensuring that $d\chi_2(|\xibh|^2)$ is supported in
$|\xibh|^2\in [c_1,2c_1]$ hence outside $\dot\Sigma$ -- it should
be thought of as a factor that microlocalizes near the characteristic
set but effectively commutes with $P$ (since we already have the
microlocal elliptic result).
Then, for $\cte>0$ large, to be determined, let
\begin{equation}\label{eq:prop-22}
a=\chi_0(\cte^{-1}(2-\phi/\delta))\chi_1(\eta/
\delta+2)\chi_2(|\xibh|^2);
\end{equation}
so $a$ is a homogeneous degree zero $\Cinf$ function on a conic neighborhood
of $q_0$ in $\Tb^*X\setminus o$.
Indeed, as we see momentarily, for any $\ep>0$,
$a$ has compact
support inside this neighborhood (regarded as a subset of $\Sb^*X$, i.e.
quotienting out by the $\Real^+$-action) for $\delta$ sufficiently small,
so in fact it is globally well-defined.
In fact, on $\supp a$ we have $\phi\leq 2\delta$ and $\eta\geq-2\delta$.
Since $\omega\geq 0$, the first of these inequalities implies that
$\eta\leq 2\delta$, so on $\supp a$
\begin{equation}
|\eta|\leq 2\delta.
\end{equation}
Hence,
\begin{equation}\label{eq:omega-delta-est}
\omega\leq \ep^2\delta(2\delta-\eta)\leq4\delta^2\ep^2.
\end{equation}
In view of \eqref{eq:prop-omega-def} and \eqref{eq:prop-rem-9b},
this shows that given any $\ep_0>0$ there exists $\delta_0>0$ such that
for any $\ep\in (0,\ep_0)$ and $\delta\in(0,\delta_0)$, $a$ is supported
in $U$.
The role that $\cte$ large
plays (in the definition of $a$)
is that it increases the size of the first derivatives of $a$ relative
to the size of $a$, hence it allows us to give a bound for
$a$ in terms of a small multiple of
its derivative along the Hamilton vector field, much like the stress
energy tensor was used to bound other terms, by making $\chi'$ large
relative to $\chi$, in the (non-microlocal) energy estimate.

Now let $A_0\in\Psib^0(X)$ with $\sigma_{\bl,0}(A_0)=a$, supported
in the coordinate chart.
Also let
$\Lambda_r$ be scalar, have symbol
\begin{equation}
|\zetab_{n-1}|^{s+1/2}(1+r|\zetab_{n-1}|^2)^{-s}\,\Id,\quad r\in[0,1),
\end{equation}
so $A_r=A\Lambda_r\in\Psib^{0}(X)$ for $r>0$ and it is
uniformly bounded in $\Psibc^{s+1/2}(X)$.
Then, for $r>0$,
\begin{equation}\begin{split}\label{eq:pairing-identity}
\langle \imath A^*_r A_r Pu,u\rangle
-\langle \imath A^*_r A_r u, Pu\rangle
&=\langle \imath [A^*_r A_r,P] u,u\rangle+\langle \imath(P-P^*)A^*_rA_r u,u\rangle\\
&=\langle \imath [A^*_r A_r,P] u,u\rangle-2\im\lambda\|A_r u\|^2.
\end{split}\end{equation}
We can compute this using Proposition~\ref{prop:commutator}. We arrange
the terms of the proposition so that the terms in which a vector
field differentiates $\chi_1$ are included in $E_r$, the terms in which
a vector fields differentiates $\chi_2$ are included in $E'_r$. Thus, we
have
\begin{equation}\begin{split}\label{eq:gend-commutator-normal}
&\imath A^*_r A_r P-\imath P A^*_r A_r\\
&\qquad =
(xD_x)^* C^\sharp_r(xD_x)+(xD_x)^* xC'_r+xC''_r(xD_x)+x^2C^\flat_r
+E_r+E'_r+F_r,
\end{split}\end{equation}
with
\begin{equation}\begin{split}\label{eq:gend-commutator-normal-symbol}
&\sigma_{\bl,2s}(C_r^\sharp)=w_r^2\Big(\digamma^{-1}\delta^{-1}a
|\zetab_{n-1}|^{-1}(\hat f^\sharp+\ep^{-2}\delta^{-1}
f^\sharp)
\chi_0'\chi_1\chi_2+a^2 \tilde c^\sharp_r\Big),\\
&\sigma_{\bl,2s+1}(C'_r)=w_r^2\Big(\digamma^{-1}\delta^{-1}a
(\hat f'+\delta^{-1}\ep^{-2}f')
\chi_0'\chi_1\chi_2+a^2 \tilde c'_r\Big),\\
&\sigma_{\bl,2s+1}(C''_r)=w_r^2\Big(\digamma^{-1}\delta^{-1}a
(\hat f''+\delta^{-1}\ep^{-2}f'')
\chi_0'\chi_1\chi_2+a^2 \tilde c''_r\Big),\\
&\sigma_{\bl,2s+2}(C_r)
=w_r^2\Big(\digamma^{-1}\delta^{-1}|\zetab_{n-1}|a
(4\hat h+\hat f^\flat+\delta^{-1}\ep^{-2}f^\flat)
\chi_0'\chi_1\chi_2+a^2 \tilde c^\flat_r\Big),
\end{split}\end{equation}
where $f^\sharp$, $f'$, $f''$ and $f^\flat$ as well as $\hat f^\sharp, \hat f',
\hat f''$ and $\hat f^\flat$ are all smooth
functions on $\Tb^* X\setminus o$,
homogeneous of degree 0 (independent of
$\ep$ and $\delta$), and
$\hat h=|\zetab_{n-1}|^{-2}h$
is the rehomogenized
version of $h$. Moreover, $f^\sharp, f', f'', f^\flat$ arise from when
$\omega$ is differentiated in $\chi(\cte^{-1}(2-\phi/\delta))$, and thus
vanish when $\omega=0$, while $\hat f^\sharp, \hat f',
\hat f''$ and $\hat f^\flat$ arise when $\eta$ is differentiated in
$\chi(\cte^{-1}(2-\phi/\delta))$, and comprise all such terms with
the exception of those arising from the $\pa_{\xib}$
component of $V^\flat|_Y$ (which gives $4\hat h=4|\zeta_{n-1}|^{-2}h$
on the last line above)
hence are the sums of functions vanishing at $x=0$ (corresponding to
us only specifying the restrictions of the vector fields in
\eqref{eq:V-form} at $Y$) and functions vanishing at $\hat\xib=0$
(when $|\zetab_{n-1}|^{-1}$ in $\eta=-\xib|\zetab_{n-1}|^{-1}$
is differentiated)\footnote{Terms of the latter kind
did not occur in \cite{Vasy:Propagation-Wave}
as time-translation invariance was assumed, but it does occur in
\cite{Vasy:Diffraction-edges} and \cite{Vasy:Maxwell}, where the Lorentzian scalar setting
is considered.}.

In this formula we think of
\begin{equation}\label{eq:normal-main-term}
4\digamma^{-1}\delta^{-1}w_r^2 a
|\zetab_{n-1}|\hat h\chi_0'\chi_1\chi_2
\end{equation}
as the main term; note that $\hat h$ is positive near $q_0$.
Compared to this,
the terms with $a^2$ are negligible, for they can all
be bounded by
$$
c\digamma^{-1}(\digamma^{-1}\delta^{-1}w_r^2 a
|\zetab_{n-1}|^{-1}\chi_0'\chi_1\chi_2)
$$
(cf.\ \eqref{eq:normal-main-term}), i.e.\ by a small multiple of
$\digamma^{-1}\delta^{-1}w_r^2 a
|\zetab_{n-1}|^{-1}\chi_0'\chi_1\chi_2$
when $\digamma$ is taken large, using that
$2-\phi/\delta\leq 4$ on $\supp a$ and
\begin{equation}\label{eq:chi_0-chi_0-prime}
\chi_0(\digamma^{-1}t)=(\digamma^{-1}t)^2
\chi_0'(\digamma^{-1}t)\leq 16\digamma^{-2}\chi'_0(\digamma^{-1}t),
\ t\leq 4;
\end{equation}
see the discussion in \cite[Section~6]{Vasy:Diffraction-edges}
and \cite{Vasy:Propagation-Wave} following
Equation~(6.19).

The vanishing condition on the $f^\sharp,f',f'',f^\flat$ ensures that,
on $\supp a$,
\begin{equation}\label{eq:normal-f-estimates}
|f^\sharp|,|f'|,|f''|,|f^\flat|\leq C\omega^{1/2}\leq 2C\ep\delta,
\end{equation}
so the corresponding terms can thus
be estimated using $w_r^2\digamma^{-1}\delta^{-1}a
|\zetab_{n-1}|^{-1}\chi_0'\chi_1\chi_2$ provided $\ep^{-1}$ is not too
large, i.e.\ there exists $\tilde\ep_0>0$ such that if $\ep>\tilde\ep_0$,
the terms with $f^\sharp,f',f'',f^\flat$ can be treated as error terms.

On the other hand, we have
\begin{equation}\label{eq:normal-hat-f-estimates}
|\hat f^\sharp|,|\hat f'|,|\hat f''|,|\hat f^\flat|
\leq C|x|+C|\hat\xib|\leq C\omega^{1/2}+C|\hat\xib|
\leq 2C\ep\delta+C|\hat\xib|.
\end{equation}
Now, on $\dot\Sigma$, $|\hat\xib|\leq 2|x|$ (for $|\xib|=x|\xi|\leq
2|x| |\zetab_{n-1}|$ with $U$ sufficiently small). Thus we can write
$\hat f^\sharp=\hat f^\sharp_\sharp+\hat f^\sharp_\flat$ with
$\hat f^\sharp_\flat$ supported away from $\dot\Sigma$ and
$\hat f^\sharp_\sharp$ satisfying
\begin{equation}\label{eq:normal-hat-f-sharp-estimates}
|\hat f^\sharp_\sharp|\leq C|x|+C|\hat\xib|\leq C'|x|\leq
C'\omega^{1/2}\leq 2C'\ep\delta;
\end{equation}
we can also obtain a similar decomposition for $\hat f',\hat f'',
\hat f^\flat$.

Indeed, using \eqref{eq:chi_0-chi_0-prime} it is useful to rewrite
\eqref{eq:gend-commutator-normal-symbol} as
\begin{equation}\begin{split}\label{eq:gend-commutator-normal-symbol-2}
&\sigma_{\bl,2s}(C^\sharp_r)=w_r^2\digamma^{-1}\delta^{-1}a
|\zetab_{n-1}|^{-1}(\hat f^\sharp+\ep^{-2}\delta^{-1}
f^\sharp+\digamma^{-1}\delta\hat c^\sharp_r)
\chi_0'\chi_1\chi_2,\\
&\sigma_{\bl,2s+1}(C'_r)=w_r^2\delta^{-1}\digamma^{-1}a(\hat f'+
\delta^{-1}\ep^{-2}f'+
\digamma^{-1}\delta\hat c'_r)
\chi_0'\chi_1\chi_2,\\
&\sigma_{\bl,2s+1}(C''_r)=w_r^2\delta^{-1}\digamma^{-1}a(\hat f''+
\delta^{-1}\ep^{-2}f''+
\digamma^{-1}\delta\hat c''_r)
\chi_0'\chi_1\chi_2,\\
&\sigma_{\bl,2s+2}(C^\flat_r)
=w_r^2\delta^{-1}\digamma^{-1}a|\zetab_{n-1}|
(4\hat h+\hat f^\flat+\delta^{-1}\ep^{-2}f^\flat+\digamma^{-1}\hat
c^\flat_r)
\chi_0'\chi_1\chi_2,
\end{split}\end{equation}
with
\begin{itemize}
\item
$f^\sharp$, $f'$, $f''$ and $f^\flat$ are all smooth
functions on $\Tb^* X\setminus o$, homogeneous of degree 0, satisfying
\eqref{eq:normal-f-estimates} (and are independent of $\digamma,\ep,\delta,r$),
\item
$\hat f^\sharp$, $\hat f'$, $\hat f''$ and $\hat f^\flat$ are all smooth
functions on $\Tb^* X\setminus o$, homogeneous of degree 0, with
$\hat f^\sharp=\hat f^\sharp_\sharp+\hat f^\sharp_\flat$,
$\hat f^\sharp_\sharp,\hat f'_\sharp,\hat f''_\sharp,
\hat f^\flat_\sharp$ satisfying
\eqref{eq:normal-hat-f-sharp-estimates}
(and are independent of $\digamma,\ep,\delta,r$),
while $\hat f^\sharp_\flat,\hat f'_\flat,\hat f''_\flat,
\hat f^\flat_\flat$ is supported away from $\dot\Sigma$,
\item
and $\hat c^\sharp_r$, $\hat c'_r$,
$\hat c''_r$ and $\hat c^\flat_r$ are all smooth
functions on $\Tb^* X\setminus o$, homogeneous of degree 0, uniformly bounded
in $\ep,\delta,r,\digamma$.
\end{itemize}

Let
$$
b_r=2w_r|\zetab_{n-1}|^{1/2}(\digamma\delta)^{-1/2}(\chi_0\chi'_0)^{1/2}\chi_1\chi_2,
$$
and let $\tilde B_r\in\Psib^{s+1}(X)$ with principal symbol
$b_r$.
Then let
\begin{equation*}
C\in\Psib^0(X),\ \sigma_{\bl,0}(C)=|\zetab_{n-1}|^{-1}
h^{1/2}\psi=\hat h^{1/2}\psi,
\end{equation*}
where
$\psi\in S^0_{\hom}(\Tb^*X\setminus o)$ is identically $1$ on $U$
considered as a subset
of $\Sb^*X$; recall from Remark~\ref{rem:normal-remark}
that $\hat h$ is bounded below by a positive
quantity here.

If $\tilde C_{r}\in\Psib^{2s}(X)$
with principal symbol
$$
\sigma_{\bl,2s}(\tilde C_r)=
-4w_r^2 \digamma^{-1}\delta^{-1}a |\zetab_{n-1}|^{-1}\chi_0'\chi_1\chi_2
=-|\zetab_{n-1}|^{-2}b_r^2,
$$
then we deduce from
\eqref{eq:gend-commutator-normal}-\eqref{eq:gend-commutator-normal-symbol-2}
that\footnote{The $f^\sharp_\sharp$ terms are included in $R^\sharp$,
while the $f^\sharp_\flat$ terms are included in $E'$, and similarly
for the other analogous terms in $f'$, $f''$, $f^\flat$. Moreover,
in view of Lemma~\ref{lemma:b-comm-improve},
we can freely rearrange factors, e.g.\ writing $C^* x^2 C$ as
$xC^*C x$, if we wish, with the exception of commuting powers of $x$
with $xD_x$ or $(xD_x)^*$ since we need to regard the latter as elements
of $\Diffz^1(X)$ rather than $\Diffb^1(X)$.
Indeed, the difference between rearrangements has lower b-order than
the product, in this case being in $x^2\Psib^{-1}(X)$, which
in view of Lemma~\ref{lemma:b-to-0-conversion-ps},
at the cost of dropping powers of $x$, can be translated into
a gain in 0-order, $x^2\Psib^{-1}(X)\subset \Diffz^2\Psib^{-3}(X)$, with the
result that these terms can be moved to the `error term', $R''\in
L^\infty((0,1);\Diffz^2\Psib^{2s-1}(X))$. }
\begin{equation}\begin{split}\label{eq:P-comm}
&\imath A^*_r A_r P-\imath P A^*_r A_r\\
&\quad=\tilde B^*_r\big(C^*x^2 C+x R^\flat x
+(xD_x)^* \tilde R' x+x\tilde R'' (xD_x)
+(xD_x)^*R^\sharp(xD_x)\big)\tilde B_r\\
&\qquad\qquad\qquad+R''_r+E_r+E'_r
\end{split}\end{equation}
with
\begin{equation*}\begin{split}
&R^\flat\in\Psib^0(X),\ \tilde R',\tilde R''\in\Psib^{-1}(X),
\ R^\sharp\in\Psib^{-2}(X),\\
&R''_r\in L^\infty((0,1);\Diffz^2\Psib^{2s-1}(X)),
\ E_r,E'_r\in L^\infty((0,1);\Diffz^2\Psib^{2s}(X)),
\end{split}\end{equation*}
with $\WFb'(E)\subset\eta^{-1}((-\infty,-\delta])
\cap U$, $\WFb'(E')\cap\dot\Sigma=\emptyset$,
and with
$r^\flat=\sigma_{\bl,0}(R^\flat)$, $\tilde r'=\sigma_{\bl,-1}(\tilde R')$,
$\tilde r''=\sigma_{\bl,-1}(\tilde R'')$,
$r^\sharp\in\sigma_{\bl,-2}(R^\sharp)$,
\begin{equation*}\begin{split}
&|r^\flat|\leq C_2(\delta\ep+\ep^{-1}+\delta\digamma^{-1}),
\ |\zetab_{n-1} \tilde r'|\leq C_2(\delta\ep+\ep^{-1}+\delta\digamma^{-1}),\\
&|\zetab_{n-1} \tilde r''|\leq C_2(\delta\ep+\ep^{-1}+\delta\digamma^{-1}),
\ |\zetab_{n-1}^2 r^\sharp|\leq C_2(\delta\ep+\ep^{-1}+\delta\digamma^{-1}).
\end{split}\end{equation*}
This is almost completely analogous to
\cite[Equation~(6.18)]{Vasy:Propagation-Wave} with the understanding that
each term of \cite[Equation~(6.18)]{Vasy:Propagation-Wave} inside
the parentheses attains an additional factor of $x^2$ (corresponding
to $\Box$ being in $\Diffz^2(X)$ rather than $\Diff^2(X)$) which
we partially include in $xD_x$ (vs.\ $D_x$).
The only difference is the presence of the
$\delta\digamma^{-1}$ term which however is treated like the $\ep\delta$ term
for $\digamma$ sufficiently large, hence the rest of the
proof proceeds very similarly to that paper.
We go through this argument to show the role that $\lambda$ and the
Poincar\'e inequality play, and in particular how the restrictions
on $\lambda$ arise.

Having calculated the commutator, we proceed to estimate the `error terms'
$R^\flat$, $\tilde R'$, $\tilde R''$
and $R^\sharp$ as operators. We start with $R^\flat$.
By the standard square root construction to prove the
boundedness of ps.d.o's on $L^2$, see e.g.\ the discussion after
\cite[Remark~2.1]{Vasy:Propagation-Wave}, there exists
$R^\flat_\flat\in\Psib^{-1}(X)$ such that
\begin{equation*}
\|R^\flat v\|\leq2\sup|r^\flat|\,\|v\|+\|R^\flat_\flat v\|
\end{equation*}
for all $v\in L^2(X)$. Here $\|\cdot\|$ is the $L^2(X)$-norm, as usual.
Thus, we can estimate, for any $\gamma>0$,
\begin{equation*}\begin{split}
|\langle R^\flat v,v\rangle|&\leq \|R^\flat v\|\,\|v\|
\leq 2\sup |r^\flat|\,\|v\|^2+\|R^\flat_\flat v\|\,\|v\|\\
&\leq 2C_2(\delta\ep+\ep^{-1}+\delta\digamma^{-1})\|v\|^2
+\gamma^{-1}\|R^\flat_\flat v\|^2+\gamma \|v\|^2.
\end{split}\end{equation*}

Now we turn to $\tilde R'$.
Let $T\in\Psib^{-1}(X)$ be elliptic (which we use to shift
the orders of ps.d.o's at our convenience), with symbol $|\zetab_{n-1}|^{-1}$
on $\supp a$,
$T^-\in\Psib^1(X)$ a parametrix, so
$T^-T=\Id+F$, $F\in\Psib^{-\infty}(X)$.
Then there exists
$\tilde R'_\flat\in\Psib^{-1}(X)$ such that
\begin{equation*}\begin{split}
\|(\tilde R')^* w\|&=\|(\tilde R')^* (T^- T -F)w\|
\leq\|((\tilde R')^* T^-)(Tw)\|+\|(\tilde R')^* Fw\|\\
&\leq 2C_2(\delta\ep+\ep^{-1}+\delta\digamma^{-1})\|Tw\|+
\|\tilde R'_\flat Tw\|+\|(\tilde R')^* Fw\|
\end{split}\end{equation*}
for all $w$ with $Tw\in L^2(X)$, and similarly,
there exists
$\tilde R''_\flat\in\Psib^{-1}(X)$ such that
\begin{equation*}\begin{split}
\|\tilde R'' w\|
&\leq 2C_2(\delta\ep+\ep^{-1}+\delta\digamma^{-1})\|Tw\|+
\|\tilde R''_\flat Tw\|+\|\tilde R'' Fw\|.
\end{split}\end{equation*}
Finally, there exists $R^\sharp_\flat\in\Psib^{-1}(X)$ such that
\begin{equation*}
\|(T^-)^*R^\sharp w\|\leq 2C_2(\delta\ep+\ep^{-1}+\delta\digamma^{-1})\|Tw\|+
\|R^\sharp_\flat Tw\|+\|(T^-)^*R^\sharp Fw\|
\end{equation*}
for all $w$ with $Tw\in L^2(X)$.
Thus,
\begin{equation*}\begin{split}
|\langle x v,(\tilde R')^*(xD_x)v\rangle|\leq
&2C_2(\delta\ep+\ep^{-1}+\delta\digamma^{-1})\|TxD_x v\|\,\|xv\|\\
&\qquad+2\gamma\|xv\|^2
+\gamma^{-1}\|\tilde R'_\flat T xD_x v\|^2+\gamma^{-1}\|F' xD_xv\|^2,\\
|\langle \tilde R'' xD_x v,xv\rangle|\leq
&2C_2(\delta\ep+\ep^{-1}+\delta\digamma^{-1})\|TxD_x v\|\,\|xv\|\\
&\qquad+2\gamma\|xv\|^2
+\gamma^{-1}\|\tilde R''_\flat T xD_x v\|^2+\gamma^{-1}\|F'' xD_xv\|^2,
\end{split}\end{equation*}
and, writing $xD_xv=T^-T(x D_x v)-F(xD_x v)$ in the right factor,
and taking the adjoint of $T^-$,
\begin{equation*}\begin{split}
|\langle R^\sharp xD_x v,xD_x v\rangle|\leq
&2C_2(\delta\ep+\ep^{-1}+\delta\digamma^{-1})\|T(xD_x)v\|\,\|T(xD_x)v\|\\
&\quad
+2\gamma\|T(xD_x)v\|^2+\gamma^{-1}\|R^\sharp_\flat T (xD_x)v\|^2
+\gamma^{-1}\|F (xD_x)v\|^2\\
&\quad+\|R^\sharp(x D_x)v\|\,\|F^\sharp (xD_xv)\|,
\end{split}\end{equation*}
with $F',F'',F^\sharp\in\Psib^{-\infty}(X)$.

Now, by \eqref{eq:P-comm},
\begin{equation}\begin{split}\label{eq:pos-comm}
\langle \imath[A_r^*A_r,P]u,u\rangle
&=\|Cx\tilde B_r u\|^2
+\langle R^\flat x\tilde B_r u,x\tilde B_r u\rangle\\
&\qquad+\langle \tilde R'' xD_x \tilde B_r u,x\tilde B_r u\rangle
+\langle x\tilde B_r u,(\tilde R')^*xD_x \tilde B_r u\rangle\\
&\qquad+\langle R^\sharp xD_x \tilde B_r u,
xD_x\tilde B_r u\rangle\\
&\qquad+\langle R''_ru, u\rangle
+\langle (E_r+E'_r)u,u\rangle
\end{split}\end{equation}
On the other hand, this commutator can be expressed as in
\eqref{eq:pairing-identity}, so
\begin{equation}\begin{split}\label{eq:pairing-rewrite}
\langle \imath A^*_r A_r Pu,u\rangle
-\langle \imath A^*_r A_r u, Pu\rangle
&=-2\im\lambda\|A_r u\|^2
+\|Cx\tilde B_r u\|^2
+\langle R^\flat x\tilde B_r u,x\tilde B_r u\rangle\\
&\qquad+\langle \tilde R'' xD_x \tilde B_r u,x\tilde B_r u\rangle
+\langle x\tilde B_r u,(\tilde R')^*xD_x \tilde B_r u\rangle\\
&\qquad+\langle R^\sharp xD_x \tilde B_r u,
xD_x\tilde B_r u\rangle\\
&\qquad+\langle R''_ru, u\rangle
+\langle (E_r+E'_r)u,u\rangle,
\end{split}\end{equation}
so the sign of the first two terms agree if $\im\lambda<0$, and
the $\im\lambda$ term vanishes if $\lambda$ is real.

Assume for the moment that $\WFb^{-1,s+3/2}(Pu)\cap U=\emptyset$ -- this
is certainly the case in our setup if $q_0\notin\WFb^{-1,\infty}(Pu)$, but this
assumption is a little stronger than $q_0\notin\WFb^{-1,s+1}(Pu)$,
which is what we need to assume for the second paragraph in the
statement of the proposition. We deal with the weakened
hypothesis $q_0\notin\WFb^{-1,s+1}(Pu)$ at the end of the proof.
Returning to \eqref{eq:pairing-rewrite}, the utility of the commutator
calculation is that we have good information about $Pu$ (this is
where we use that we have a microlocal solution of the PDE!). Namely,
we estimate the left hand side as
\begin{equation}\begin{split}\label{eq:Pu-s+32}
|\langle A_r Pu,A_r u\rangle|&\leq|\langle (T^-)^*A_rPu,TA_ru\rangle|
+|\langle A_rPu,FA_r u\rangle|\\
&\leq \|(T^-)^*A_r Pu\|_{H^{-1}_0(X)}
\|TA_r u\|_{H^1_0(X)}\\
&\qquad+\|A_r Pu\|_{H^{-1}_0(X)}\|FA_r u\|_{H^1_0(X)}.
\end{split}\end{equation}
Since $(T^-)^*A_r$ is uniformly bounded in $\Psibc^{s+3/2}(X)$,
$TA_r$ is uniformly bounded in $\Psibc^{s-1/2}(X)$, both with $\WFb'$
in $U$, with $\WFb^{-1,s+3/2}(Pu)$, resp.\ $\WFb^{1,s-1/2}(u)$ disjoint
from them, we deduce (using Lemma~\ref{lemma:WFb-mic-q} and its
$H^{-1}_0$ analogue) that $|\langle (T^-)^*A_rPu,TA_ru\rangle|$
is uniformly bounded. Similarly, taking into account that
$FA_r$ is uniformly bounded
in $\Psib^{-\infty}(X)$, we see that
$|\langle A_rPu,FA_r u\rangle|$ is also uniformly
bounded, so $|\langle A_r Pu,A_r u\rangle|$ is uniformly bounded
for $r\in(0,1]$.

Thus,
\begin{equation}\begin{split}\label{eq:prop-64}
&\|Cx\tilde B_r u\|^2-\im\lambda\|A_r u\|^2\\
&\leq
2|\langle A_rPu,A_ru\rangle|+|\langle (E_r+E'_r)u,u\rangle|\\
&\qquad
+\left(2C_2(\delta\ep+\ep^{-1}+\delta\digamma^{-1})+\gamma\right)
\|x\tilde B_r u\|^2
+\gamma^{-1}\|R^\flat_\flat x\tilde B_r u\|^2\\
&\qquad+4C_2(\delta\ep+\ep^{-1}+\delta\digamma^{-1})\|x\tilde B_ru\|
\|T (xD_x)\tilde B_ru\|\\
&\qquad
+\gamma^{-1}\|\tilde R'_\flat T (xD_x)
\tilde B_ru\|^2+\gamma^{-1}\|\tilde R''_\flat T (xD_x)
\tilde B_ru\|^2+4\gamma\|x\tilde B_ru\|^2\\
&\qquad+\left(2C_2(\delta\ep+\ep^{-1}+\delta\digamma^{-1})+2\gamma\right)\|T(xD_x)\tilde B_ru\|^2\\
&\qquad+\gamma^{-1}\|R^\sharp_\flat T(xD_x)\tilde B_r u\|^2
+\|R^\sharp (xD_x)\tilde B_r u\|\,\|F(xD_x)\tilde B_r u\|\\
&\qquad+\gamma^{-1}\|F (xD_x)\tilde B_r u\|^2\\
&\qquad+\gamma^{-1}\|F' (xD_x)\tilde B_r u\|^2
+\gamma^{-1}\|F'' (xD_x)\tilde B_r u\|^2.
\end{split}\end{equation}
All terms but the ones involving $C_2$ or $\gamma$ (not $\gamma^{-1}$)
remain bounded as $r\to 0$.
The $C_2$ and $\gamma$ terms can be estimated by
writing $T(xD_x)=(xD_x)T'+T''$ for some $T',T''\in\Psib^{-1}(X)$, and
using Lemma~\ref{lemma:Dirichlet-form} and the Poincar\'e lemma where
necessary. Namely, we use
either $\im\lambda\neq 0$ or $\lambda<(n-1)^2/4$ to
control $xD_x L\tilde B_r u$ and $L\tilde B_r u$ in $L^2(X)$ in terms of
$\|x \tilde B_r u\|_{L^2}$ where $L\in\Psib^{-1}(X)$; this is possible
by factoring $D_{y_{n-1}}$ (which is elliptic on $\WF'(\tilde B_r)$)
out of $\tilde B_r$ modulo an error $\tilde F_r$ bounded in $\Psibc^{s}(X)$,
which in turn can be incorporated into the `error' given
by the right hand side of Lemma~\ref{lemma:Dirichlet-form}.
Thus,
there exists $C_3>0$, $G\in\Psib^{s-1/2}(X)$, $\tilde G\in\Psib^{s+1/2}(X)$
as in Lemma~\ref{lemma:Dirichlet-form} such that
\begin{equation*}\begin{split}
&\|xD_x L\tilde B_r u\|^2+\|L\tilde B_r u\|\\
&\ \leq
C_3 (\|x\tilde B_r u\|^2+
\|u\|^2_{H^{1,k}_{0,\bl,\loc}(X)}+\|Gu\|^2_{H^1_0(X)}+\|Pu\|^2_{H^{-1,k}_{0,\bl,\loc}(X)}
+\|\tilde G Pu\|^2_{H^{-1}_0(X)}).
\end{split}\end{equation*}
We further estimate $\|x\tilde B_r u\|$ in terms
of $\|Cx\tilde B_r u\|$ and $\|u\|_{H^1_{0,\loc}(X)}$
using that $C$ is elliptic on $\WFb'(B)$ and Lemma~\ref{lemma:WFb-mic-q}.
We conclude, using $\im\lambda\leq 0$,
taking $\ep$ sufficiently large, then $\gamma,\delta_0$ sufficiently
small, and finally $\digamma$ sufficiently large,
that there exist $\gamma>0$, $\ep>0$,
$\delta_0>0$ and $C_4>0$, $C_5>0$ such that for
$\delta\in(0,\delta_0)$,
\begin{equation*}\begin{split}
C_4\|x\tilde B_r u\|^2\leq &2|\langle A_rPu,A_ru\rangle|
+|\langle (E_r+E'_r) u, u\rangle|\\
&\qquad+C_5(\|G u\|^2_{H^1_0(X)}
+\|\tilde GPu\|^2_{H^{-1}_0(X)})\\
&\qquad+C_5(\|u\|_{H^{1,k}_{0,\bl,\loc}(X)}+\|Pu\|_{H^{-1,k}_{0,\bl,\loc}(X)}).
\end{split}\end{equation*}
Letting $r\to 0$ now keeps the right hand side
bounded, proving that $\|x\tilde B_r u\|$ is uniformly bounded as
$r\to 0$, hence $x\tilde B_0 u\in L^2(X)$ (cf.\ the proof
of Proposition~\ref{prop:elliptic}).
In view of Lemma~\ref{lemma:Dirichlet-form} and the Poincar\'e
inequality (as in the proof of Proposition~\ref{prop:elliptic})
this proves that $q_0\notin\WFb^{1,s}(u)$, and hence proves the
first statement of the proposition.

In fact,
recalling that we needed $q_0\notin\WFb^{-1,s+3/2}(Pu)$ for
the uniform boundedness in \eqref{eq:Pu-s+32}, this proves a slightly
weaker version of the second statement of the proposition with
$\WFb^{-1,s+1}(Pu)$ replaced by $\WFb^{-1,s+3/2}(Pu)$.
For the more precise statement we modify \eqref{eq:Pu-s+32} -- this
is the only term in \eqref{eq:prop-64} that needs modification
to prove the optimal statement. Let
$\tilde T\in\Psib^{-1/2}(X)$ be elliptic, $\tilde T^-\in\Psib^{1/2}(X)$
a parametrix, $\tilde F=\tilde T^-\tilde T-\Id\in\Psib^{-\infty}(X)$.
Then, similarly to \eqref{eq:Pu-s+32}, we have for any $\gamma>0$,
\begin{equation}\begin{split}\label{eq:Pu-s+1}
|\langle A_r Pu,A_r u\rangle|&\leq|\langle (\tilde T^-)^*A_rPu,\tilde
TA_ru\rangle|
+|\langle A_rPu,\tilde FA_r u\rangle|\\
&\leq \gamma^{-1}\|(\tilde T^-)^*A_r Pu\|_{H^{-1}_0(X)}^2+
\gamma\|\tilde TA_r u\|_{H^1(X)}^2\\
&\qquad+\|A_r Pu\|_{H^{-1}(X)}\|\tilde FA_r u\|_{H^1_0(X)}.
\end{split}\end{equation}
The last term on the right hand side can be estimated as before.
As $(\tilde T^-)^*A_r$ is bounded in $\Psibc^{s+1}(X)$ with $\WFb'$
disjoint from $U$, we see that $\|(\tilde T^-)^*A_r Pu\|_{H^{-1}_0(X)}$
is uniformly bounded. Moreover, $\|\tilde TA\Lambda_r u\|^2_{H^1_0(X)}$ can
be estimated, using Lemma~\ref{lemma:Dirichlet-form} and the Poincar\'e
inequality,
by $\|xD_{y_{n-1}}\tilde TA\Lambda_r u\|^2_{L^2(X)}$
modulo terms that are uniformly bounded as $r\to 0$.
The principal symbol of $D_{y_{n-1}}\tilde TA$ is $\zetab_{n-1}\sigma_{\bl,-1/2}(\tilde T)a$,
with
$a=\chi_0\chi_1\chi_2$, where $\chi_0$ stands for
$\chi_0(A_0^{-1}(2-\frac{\phi}{\delta}))$, etc.,
so we can write:
\begin{equation*}\begin{split}
|\zetab_{n-1}|^{1/2}a&=|\zetab_{n-1}|^{1/2}\chi_0\chi_1\chi_2=A_0^{-1}(2-\phi/\delta)
|\zetab_{n-1}|^{1/2}(\chi_0\chi_0')^{1/2}\chi_1\chi_2\\
&=
\digamma^{-1/2}\delta^{1/2}(2-\phi/\delta)\tilde b, 
\end{split}\end{equation*}
where we used that
\begin{equation*}
\chi'_0(\digamma^{-1}(2-\phi/\delta))
=\digamma^2(2-\phi/\delta)^{-2}\chi_0(\digamma^{-1}(2-\phi/\delta))
\end{equation*}
when $2-\phi/\delta>0$, while $a$, $\tilde b$ vanish otherwise.
Correspondingly, using that
$|\zetab_{n-1}|^{1/2}\sigma_{\bl,-1/2}(\tilde T)$ is $\Cinf$,
homogeneous degree zero, near the support of $a$ in $\Tb^*X\setminus o$,
we can write $D_{y_{n-1}}\tilde TA=G\tilde B+F$, $G\in\Psib^0(X)$, $F\in\Psib^{-1/2}(X)$.
Thus, modulo terms that are bounded
as $r\to 0$, $\|xD_{y_{n-1}}\tilde TA\Lambda_r u\|^2$ (hence
$\|\tilde TA\Lambda_r u\|^2_{H^1_0(X)}$)
can be estimated from above
by $C_6\|x\tilde B_r u\|^2$. Therefore, modulo terms that
are bounded as $r\to 0$,
for $\gamma>0$ sufficiently small,
$\gamma\|\tilde TA_r u\|_{H^1_0(X)}^2$ can be absorbed
into $\|Cx\tilde B_r u\|^2$. As the treatment of the other terms
on the right hand side of \eqref{eq:prop-64} requires no change,
we deduce as above
that $x\tilde B_0 u\in L^2(X)$, which
(in view of Lemma~\ref{lemma:Dirichlet-form})
proves that $q_0\notin\WFb^{1,s}(u)$, completing the proof of the iterative
step.

We need to make one more remark to prove the proposition
for $\WFb^{1,\infty}(u)$, namely
we need to show that the neighborhoods of $q_0$ which are disjoint from
$\WFb^{1,s}(u)$ do not shrink uncontrollably to $\{q_0\}$ as $s\to\infty$.
This argument parallels to last paragraph of the proof of
\cite[Proposition~24.5.1]{Hor}.
In fact, note that above we have proved that the elliptic set of
$\tilde B=\tilde B_s$
is disjoint from $\WFb^{1,s}(u)$.
In the next step, when we are proving $q_0\notin\WFb^{1,s+1/2}(u)$,
we decrease $\delta>0$ slightly (by an arbitrary small amount),
thus decreasing the support
of $a=a_{s+1/2}$ in \eqref{eq:prop-22}, to make sure that $\supp a_{s+1/2}$
is a subset of the elliptic set of the union of $\tilde B_s$ with the
region $\eta<0$, and hence that $\WFb^{1,s}(u)
\cap\supp a_{s+1/2}=\emptyset$. Each iterative step thus shrinks the
elliptic set of $\tilde B_s$ by an arbitrarily small amount, which allows
us to conclude that $q_0$ has a neighborhood $U'$ such that $\WFb^{1,s}(u)
\cap U'=\emptyset$ for all $s$. This proves that $q_0\notin \WFb^{1,\infty}(u)$,
and indeed that $\WFb^{1,\infty}(u)\cap U'=\emptyset$, for if $A\in\Psib^m(X)$
with $\WFb'(A)\subset U'$ then $Au\in H^1_0(X)$ by Lemma~\ref{lemma:WFb-mic}
and Corollary~\ref{cor:WF-to-H1}.
\end{proof}

Before turning to tangential propagation we need
a technical lemma, roughly stating that when applied to solutions
of $Pu=0$, $u\in H^1_0(X)$, microlocally near $\cG$, $xD_x$ and $\Id$
are not merely bounded by $xD_{y_{n-1}}$, but it is small compared to it,
provided that $\lambda\in\Cx\setminus[(n-1)^2/4,\infty)$.
This result is the analogue of \cite[Lemma~7.1]{Vasy:Propagation-Wave},
and is proved as there, with the only difference being
that the term $\langle\lambda A_r u, A_r u\rangle$ cannot be dropped,
but it is treated just as in Proposition~\ref{prop:elliptic} above.
Below a $\delta$-neighborhood refers to a $\delta$-neighborhood with respect
to the metric associated to any Riemannian metric on the manifold $\Tb^*X$,
and we identify $\Sb^*X$ as the unit ball bundle
with respect to some fiber metric
on $\Tb^*X$.

\begin{lemma}(cf.\ \cite[Lemma~7.1]{Vasy:Propagation-Wave}.)\label{lemma:Dt-Dx}
Suppose that $P=\Box_g+\lambda$,
$$
\lambda\in\Cx\setminus [(n-1)^2/4,\infty).
$$
Suppose $u\in H^{1,k}_{0,\bl,\loc}(X)$, and suppose that we are given
$K\subset\Sb^*X$ compact satisfying
\begin{equation*}
K\subset\cG\cap T^*Y\setminus\WFb^{-1,s+1/2}(Pu).
\end{equation*}
Then
there exist $\delta_0>0$ and $C_0>0$ with the following property.
Let $\delta<\delta_0$,
$U\subset\Sb^*X$ open in a $\delta$-neighborhood of $K$,
and $\cA=\{A_r:\ r\in(0,1]\}$ be a bounded
family of ps.d.o's in $\Psibc^s(X)$ with $\WFb'(\cA)\subset U$, and
with $A_r\in\Psib^{s-1}(X)$ for $r\in (0,1]$.

Then
there exist $G\in\Psib^{s-1/2}(X)$, $\tilde G\in\Psib^{s+1/2}(X)$
with $\WFb'(G),\WFb'(\tilde G)\subset U$ and $\tilde C_0
=\tilde C_0(\delta)>0$ such that
for all $r>0$,
\begin{equation}\begin{split}\label{eq:Dt-Dx-est}
\|xD_x A_r u\|^2+\|A_ru\|^2
\leq C_0\delta\|xD_{y_{n-1}} A_r u\|^2
&+\tilde C_0\Big(\|u\|^2_{H^{1,k}_{0,\bl,\loc}(X)}
+\|Gu\|^2_{H^1_0(X)}\\
&\quad+\|Pu\|^2_{H^{-1,k}_{0,\bl,\loc}(X)}
+\|\tilde G Pu\|^2_{H^{-1}_0(X)}\Big).
\end{split}\end{equation}
The meaning of $\|u\|_{H^{1,k}_{0,\bl,\loc}(X)}$ and
$\|Pu\|^2_{H^{-1,k}_{0,\bl,\loc}(X)}$ is stated in
Remark~\ref{rem:localize}.
\end{lemma}

\begin{rem}
As $K$ is compact, this is essentially a local result. In particular, we
may assume that $K$ is a subset of $\Tb^*X$ over
a suitable local coordinate patch. Moreover, we may assume that $\delta_0>0$
is sufficiently small so that $D_{y_{n-1}}$ is elliptic on $U$.
\end{rem}

\begin{proof}
By Lemma~\ref{lemma:Dirichlet-form}
applied with $K$ replaced by $\WFb'(\cA)$
in the hypothesis (note that the latter is compact), we already know that
\begin{equation}\begin{split}\label{eq:mic-ell-gl-8}
&|\langle d A_r u, dA_r u\rangle_G+\lambda\|A_r u\|^2|\\
&\qquad\leq
C_0'(\|u\|^2_{H^{1,k}_{0,\bl,\loc}(X)}+\|Gu\|^2_{H^1_0(X)}+\|Pu\|^2_{H^{-1,k}_{0,\bl,\loc}(X)}
+\|\tilde G Pu\|^2_{H^{-1}_0(X)}).
\end{split}\end{equation}
for some $C'_0>0$ and for some $G$, $\tilde G$ as in the statement of the
lemma. Freezing the coefficients at $Y$, as in the proof
of Proposition~\ref{prop:elliptic}, see
\cite[Lemma~7.1]{Vasy:Propagation-Wave} for details, we deduce that
\begin{equation}\begin{split}\label{eq:Dt-Dx-15}
&\left|\|xD_x A_r u\|^2-\lambda\|A_r u\|^2\right|\\
&\qquad\leq \int_X
\left(B_{ij}(0,y) (xD_{y_i})A_r u\,\overline{(xD_{y_j})A_r u}\right)\,|dg|
+C_1\delta\|xD_{y_{n-1}} A_r u\|^2\\
&\qquad\qquad
+C''_0(\|u\|^2_{H^{1,k}_{0,\bl,\loc}(X)}+\|Gu\|^2_{H^1_0(X)}
+\|Pu\|^2_{H^{-1,k}_{0,\bl,\loc}(X)}
+\|\tilde G Pu\|^2_{H^{-1}(X)}).
\end{split}\end{equation}
Now, one can show that
\begin{equation}\begin{split}\label{eq:Dt-Dx-16}
&\left|\int_X
\left(\sum D_{y_i}^* B_{ij}(0,y)
D_{y_j}) x A_r u\,\overline{x A_r u}\right)\,|dg|\right|\\
&\qquad\leq C_2\delta \|D_{y_{n-1}} A_r u\|^2+\tilde C_2(\delta)
(\|u\|^2_{H^{1,k}_{0,\bl,\loc}(X)}+\|Gu\|^2_{H^1_0(X)})
\end{split}\end{equation}
{\em precisely} as in the proof of
\cite[Lemma~7.1]{Vasy:Propagation-Wave}.
Equations~\eqref{eq:Dt-Dx-15}-\eqref{eq:Dt-Dx-16} imply \eqref{eq:Dt-Dx-est}
with the left hand side replaced by
$\left|\|xD_x A_r u\|^2-\lambda\|A_r u\|^2\right|$.
If $\im\lambda\neq 0$, taking
the imaginary part of $\|xD_x A_r u\|^2-\lambda\|A_r u\|^2$ gives
the desired bound for $\|A_r u\|^2$, hence taking the real part
gives the desired bound for $\|xD_x A_r u\|^2$ as well. If
$\im\lambda=0$ but $\lambda<(n-1)^2/4$, we finish the proof
using the Poincar\'e inequality, cf.\ the proof of
Proposition~\ref{prop:elliptic}.
\end{proof}

We finally state the tangential, or glancing, propagation result.

\begin{prop}\label{prop:tgt-prop}(Tangential, or glancing, propagation.)
Suppose that $P=\Box_g+\lambda$, $\lambda\in\Cx\setminus [(n-1)^2/4,\infty)$.
Let $U_0$ be a coordinate chart in $X$, $U$ open with $\overline{U}
\subset U_0$.
Let $u\in
H^{1,k}_{0,\bl,\loc}(X)$ for some $k\leq 0$, and let $\tilde\pi:T^*X\to T^*Y$ be
the coordinate projection
$$
\tilde\pi:(x,y,\xi,\zeta)\mapsto(y,\zeta).
$$
Given $K\subset\Sb^*_{U}X$ compact
with
\begin{equation}
K\subset(\cG\cap \Tb^*_{Y} X)\setminus\WFb^{-1,\infty}(f),\ f=P u,
\end{equation}
there exist constants $C_0>0$,
$\delta_0>0$ such that the following holds. If $\im\lambda\leq 0$,
$q_0=(y_0,\zetab_0)\in K$, $\alpha_0=\hat\pi^{-1}(q_0)$,
$W_0=\tilde\pi_*|_{\alpha_0}\sH_p$
considered as a constant vector field in local
coordinates, and
for some $0<\delta<\delta_0$, $C_0\delta\leq\epsilon<1$ and
for all $\alpha=(x,y,\xi,\zeta)\in \Sigma$
\begin{equation}\begin{split}\label{eq:tgt-prop-est}
\alpha\in T^*X &\Mand
|\tilde\pi\big(\alpha-(\alpha_0-\delta W_0)\big)|
\leq\epsilon\delta\Mand
|x(\alpha)|\leq\epsilon\delta\\
&\Rightarrow
\pi(\alpha)\notin\WFb^{1,\infty}(u),
\end{split}\end{equation}
then $q_0\notin\WFb^{1,\infty}(u)$.

In addition, $\WFb^{-1,\infty}(f)$, resp.\ $\WFb^{1,\infty}(u)$, may be replaced
by $\WFb^{-1,s+1}(f)$, resp.\ $\WFb^{1,s}(u)$, $s\in\RR$.
\end{prop}

\begin{rem}\label{rem:tgt-rem}
Just like Proposition~\ref{prop:normal-prop},
this proposition gives regularity propagation in the {\em forward} direction along
$W_0$, i.e.\ in order to conclude regularity at $q_0$, one needs to know regularity
in the {\em backward} $W_0$-direction from $q_0$.

One can again change the direction of propagation, i.e.\ replace $\delta$
by $-\delta$ in $\alpha-(\alpha_0-\delta W_0)$, provided one also
changes the sign of $\im\lambda$ to $\im\lambda\geq 0$. In particular,
if $\im\lambda=0$, one obtains propagation estimates in both
the forward and backward directions.
\end{rem}

\begin{proof}
Again, the proof follows a proof in
\cite{Vasy:Propagation-Wave}
closely, in this case
Proposition 7.3,
as corrected at a point in \cite{Vasy:Propagation-Wave-Correction},
so we merely point out the main steps.
Again, one uses a commutant
$A\in\Psib^0(X)$ and weights $\Lambda_r\in\Psib^0(X)$, $r\in(0,1)$,
uniformly bounded in $\Psibc^{s+1/2}(X)$, $A_r=A\Lambda_r$,
in order to obtain the
propagation of $\WFb^{1,s}(u)$ with the notation of that paper, whose
analogue is $\WFb^{1,s}(u)$ here (the difference is the space
relative to which one obtains b-regularity: $H^1(X)$ in the previous
paper, the zero-Sobolev space $H^1_0(X)$ here).
One can use {\em exactly the same} commutants as in
\cite{Vasy:Propagation-Wave} (with a small correction
given in \cite{Vasy:Propagation-Wave-Correction}).
Then Proposition~\ref{prop:commutator} lets one calculate
$\imath[A^*_rA_r,P]$ to obtain a completely analogous expression
to the formulae below Equation~(7.16) of \cite{Vasy:Propagation-Wave},
as corrected in \cite{Vasy:Propagation-Wave-Correction}).
The rest of the argument is completely
analogous as well.
Again, we refer the reader to \cite{Vasy:Maxwell} because
the commutator calculation is written up exactly as above in
Proposition~\ref{prop:commutator}, see \cite[Proposition~3.10]{Vasy:Maxwell}
and it is used subsequently in 6.1 there the same way it needs to be
used here -- any modifications are analogous to those in
Proposition~\ref{prop:normal-prop} and arise due to the non-negligible
nature of $\lambda$.

Again, we first construct the symbol $a$ of our commutator following
\cite[Proof of Proposition~7.3]{Vasy:Propagation-Wave} as corrected
in \cite{Vasy:Propagation-Wave-Correction}.
Note that (with $\tilde p=x^{-2}\sigma_{\bl,2}(\tilde P)=h$)
$$
W_0(q_0)=\sH_{\tilde p}(q_0),
$$
and let
$$
W=|\zetab_{n-1}|^{-1}W_0,
$$
so $W$ is homogeneous of degree zero (with respect to the $\RR^+$-action on
the fibers of $T^*Y\setminus o$).
We use
$$
\tilde\eta=(\sgn(\zetab_{n-1})_0)(y_{n-1}-(y_{n-1})_0)
$$
now to measure propagation,
since $\zetab_{n-1}^{-1}\sH_{\tilde p}(y_{n-1})=2>0$ at $q_0$ by \eqref{eq:B-orth-0}, so
$\sH_{\tilde p}\tilde\eta$ is
$2|\zetab_{n-1}|>0$ at $q_0$. Note that $\tilde\eta$ is thus increasing along
\GBBsp of $\hat g$.

First, we require
$$
\rho_1=\tilde p(y,\zetabh)=|\zetab_{n-1}|^{-2}\tilde p(y,\zetab);
$$
note that $d\rho_1\neq 0$ at $q_0$ for $\zetab\neq 0$ there, but
$\sH_{\tilde p}\tilde p\equiv 0$, so
$$
W\rho_1(q_0)=0.
$$
Next, since
$\dim Y=n-1$, $\dim T^*Y=2n-2$, hence $\dim S^*Y=2n-3$. With a
slight abuse of notation, we also regard $q_0$ as a point in
$S^*Y$ -- recall that $S^*Y=(T^*Y\setminus o)/\RR^+$. We can also
regard $W$ as a vector field on $S^*Y$ in view of its homogeneity.
As $W$ does not vanish as a vector in $T_{q_0}S^*Y$
in view of
$W\tilde\eta(q_0)\neq 0$, $\tilde\eta$ being
homogeneous degree zero, hence a function on $S^*Y$,
the kernel of $W$ in $T^*_{q_0}S^*Y$ has dimension $2n-4$.
Thus there exist
$\rho_j$, $j=2,\ldots,2n-4$ be homogeneous degree zero functions
on $T^*Y$ (hence functions on $S^*Y$)
such that
\begin{equation}\begin{split}
&\rho_j(q_0)=0,\ j=2,\ldots,2n-4,\\
&W\rho_j(q_0)=0,\ j=2,\ldots,2n-4,\\
&d\rho_j(q_0),\ j=1,\ldots,2n-4\ \text{are linearly independent at}\ q_0.\\
\end{split}\end{equation}
By dimensional considerations,
$d\rho_j(q_0)$, $j=1,\ldots,2n-4$, together with $d\tilde\eta$
span the cotangent space of $S^*Y$ at $q_0$, i.e.\ of
the quotient of $T^*Y$ by the $\Real^+$-action, so the $\rho_j$, together
with $\tilde\eta$, can be used as local coordinates on a chart
$\tilde\cU_0\subset S^*Y$ near $q_0$. We also let $\tilde\cU$ be a neighborhood
of $q_0$ in $\Sb^*X$ such that $\rho_j$, together
with $\tilde\eta$, $x$ and $\xibh$ are local coordinates on $\tilde\cU$; this
holds if $\tilde\cU_0$ is identified with a subset of $\cG\cap\Sb^*_YX$ and $\tilde\cU$
is a product neighborhood of this in $\Sb^*X$ in terms of the coordinates
\eqref{eq:zeta-large}. Note that as $\xibh=0$ on $\dot\Sigma\cap\Sb^*_YX$,
for points $q$ in $\dot\Sigma$, one can ensure that $\xibh$ is small by ensuring
that $\tilde\pi(q)$ is close to $q_0$ and $x(q)$ is small; cf.\ the discussion
around \eqref{eq:prop-rem-9b} and after \eqref{eq:normal-phi-def}. By reducing
$\tilde\cU$ if needed (this keeps all previously discussed properties), we
may also assume that it is disjoint from $\WFb^{-1,\infty}(f)$.

Hence,
$$
|\zetab_{n-1}|^{-1}W_0 \rho_j=\sum_{i=1}^{2n-4}\tilde F_{ji}\rho_i+
\tilde F_{j,2n-3}\tilde\eta,\ j=2,\ldots,2n-4,
$$
with $\tilde F_{ji}$ smooth, $i=1,\ldots,2n-3$, $j=2,\ldots,2n-4$.
Then we
extend $\rho_j$ to a function on $\Tb^*X\setminus o$ (using the coordinates
$(x,y,\xib,\zetab)$), and conclude that
\begin{equation}\label{eq:sH-rho_2}
|\zetab_{n-1}|^{-1}\sH_{\tilde p} \rho_j=
\sum_{l=1}^{2n-4}\tilde F_{jl}\rho_l+
\tilde F_{j,2n-3}\tilde\eta+ \tilde F_{j0} x,\ j=2,\ldots,2n-4,
\end{equation}
with $\tilde F_{jl}$ smooth.
Similarly,
\begin{equation}\label{eq:sH-tilde-eta}
|\zetab_{n-1}|^{-1}\sH_{\tilde p} \tilde\eta=2+\sum_{l=1}^{2n-4}
\check F_l\rho_l
+\check F_{2n-3}\tilde\eta
+\check F_{0} x,
\end{equation}
with $\check F_l$ smooth.

Let
\begin{equation}
\omega=|x|^2+\sum_{j=1}^{2n-4}\rho_j^2.
\end{equation}
Finally, we let
\begin{equation}\label{eq:glancing-phi-def}
\phi=\tilde\eta+\frac{1}{\ep^2\delta}\omega,
\end{equation}
and define $a$ by
\begin{equation}\label{eq:prop-22t}
a=\chi_0(\digamma^{-1}(2-\phi/\delta))\chi_1((\tilde\eta\delta)/
\ep\delta+1)\chi_2(|\xib|^2/\zetab_{n-1}^2),
\end{equation}
with $\chi_0,\chi_1$ and $\chi_2$ as in the case of the normal
propagation estimate, stated after \eqref{eq:normal-phi-def}.
We always assume $\ep<1$, so on $\supp a$ we have
\begin{equation*}
\phi\leq 2\delta\Mand \tilde\eta\geq-\ep\delta-\delta\geq-2\delta.
\end{equation*}
Since $\omega\geq 0$,
the first of these inequalities implies that
$\tilde\eta\leq 2\delta$, so on $\supp a$
\begin{equation}
|\tilde\eta|\leq 2\delta.
\end{equation}
Hence,
\begin{equation}\label{eq:omega-delta-est-t}
\omega\leq \ep^2\delta(2\delta-\tilde\eta)\leq4\delta^2\ep^2.
\end{equation}
Thus, $\supp a$ lies in $\tilde\cU$ for $\delta>0$ sufficiently small.
Moreover, on $\supp d\chi_1$,
\begin{equation}\label{eq:dchi_1-supp}
\tilde\eta\in[-\delta-\ep\delta,-\delta],\ \omega^{1/2}
\leq 2\ep\delta,
\end{equation}
so this region lies in \eqref{eq:tgt-prop-est} after $\ep$ and $\delta$
are both replaced by appropriate constant multiples, namely the present
$\delta$ should be replaced by $\delta/(2|(\zetab_{n-1})_0|)$.

We proceed as in the case of hyperbolic points, letting
$A_0\in\Psib^0(X)$ with $\sigma_{\bl,0}(A_0)=a$, supported
in the coordinate chart.
Also let
$\Lambda_r$ be scalar, have symbol
\begin{equation}
|\zetab_{n-1}|^{s+1/2}(1+r|\zetab_{n-1}|^2)^{-s}\,\Id,\quad r\in[0,1),
\end{equation}
so $A_r=A\Lambda_r\in\Psib^{0}(X)$ for $r>0$ and it is
uniformly bounded in $\Psibc^{s+1/2}(X)$.
Then, for $r>0$,
\begin{equation}\begin{split}\label{eq:pairing-identity-tgt}
\langle \imath A^*_r A_r Pu,u\rangle
-\langle \imath A^*_r A_r u, Pu\rangle
&=\langle \imath [A^*_r A_r,P] u,u\rangle+\langle \imath(P-P^*)A^*_rA_r u,u\rangle\\
&=\langle \imath [A^*_r A_r,P] u,u\rangle-2\im\lambda\|A_r u\|^2.
\end{split}\end{equation}
and we compute the commutator here
using Proposition~\ref{prop:commutator}. We arrange
the terms of the proposition so that the terms in which a vector
field differentiates $\chi_1$ are included in $E_r$, the terms in which
a vector fields differentiates $\chi_2$ are included in $E'_r$. Thus, we
have
\begin{equation}\begin{split}\label{eq:gend-commutator-tgt}
&\imath A^*_r A_r P-\imath P A^*_r A_r\\
&\qquad =
(xD_x)^* C^\sharp_r(xD_x)+(xD_x)^* xC'_r+xC''_r(xD_x)+x^2C^\flat_r
+E_r+E'_r+F_r,
\end{split}\end{equation}
with
\begin{equation}\begin{split}\label{eq:gend-commutator-tgt-symbol}
&\sigma_{\bl,2s}(C_r^\sharp)=w_r^2\Big(\digamma^{-1}\delta^{-1}a
|\zetab_{n-1}|^{-1}(\hat f^\sharp+\ep^{-2}\delta^{-1}
f^\sharp)
\chi_0'\chi_1\chi_2+a^2 \tilde c^\sharp_r\Big),\\
&\sigma_{\bl,2s+1}(C'_r)=w_r^2\Big(\digamma^{-1}\delta^{-1}a
(\hat f'+\delta^{-1}\ep^{-2}f')
\chi_0'\chi_1\chi_2+a^2 \tilde c'_r\Big),\\
&\sigma_{\bl,2s+1}(C''_r)=w_r^2\Big(\digamma^{-1}\delta^{-1}a
(\hat f''+\delta^{-1}\ep^{-2}f'')
\chi_0'\chi_1\chi_2+a^2 \tilde c''_r\Big),\\
&\sigma_{\bl,2s+2}(C_r^\flat)
=w_r^2\Big(\digamma^{-1}\delta^{-1}|\zetab_{n-1}|a
(4+\hat f^\flat+\delta^{-1}\ep^{-2}f^\flat)
\chi_0'\chi_1\chi_2+a^2 \tilde c^\flat_r\Big),
\end{split}\end{equation}
where $f^\sharp$, $f'$, $f''$ and $f^\flat$ as well as $\hat f^\sharp, \hat f',
\hat f''$ and $\hat f^\flat$ are all smooth
functions on $\Tb^* X\setminus o$,
homogeneous of degree 0 (independent of
$\ep$ and $\delta$). Moreover, $f^\sharp, f', f'', f^\flat$ arise from when
$\omega$ is differentiated in $\chi_0(\cte^{-1}(2-\phi/\delta))$,
while $\hat f^\sharp, \hat f',
\hat f''$ and $\hat f^\flat$ arise when $\tilde\eta$ is differentiated in
$\chi_0(\cte^{-1}(2-\phi/\delta))$, and comprise all such terms with
the exception of part of that arising from the $-\sH_h$
component of $V^\flat|_Y$ (which gives the $4$
on the last line above, modulo a term included in $\hat f^\flat$ and
which vanishes at $\omega=0$).
In addition, as $V^\bullet \rho^2=2\rho V^\bullet\rho$
for any function $\rho$, the terms $f^\bullet$, $\bullet=\sharp,','',\flat$,
have vanishing factors
of $\rho_l$, resp.\ $x$, with the structure of
the remaining factor dictated by the form of $V^\bullet\rho_l$,
resp.\ $V^\bullet x$.
Thus, using \eqref{eq:sH-rho_2} to compute $f^\flat$,
\eqref{eq:sH-tilde-eta} to compute $\hat f^\flat$, we have
\begin{equation*}\begin{split}
f^\sharp&=\sum_{k} \rho_k f^\sharp_{k}+x
f^\sharp_{0},\\
f^\bullet&=\sum_k \rho_k f_{k}^\bullet
+x f_{0}^\bullet,\ \bullet=','',\\
f^\flat&=\sum_{kl}\rho_k\rho_l f^\flat_{kl}
+\sum_{k} \rho_k x f^\flat_{k}+ x^2  f_{0}
+\sum_{k} \rho_k\tilde\eta f^\flat_{k+},\\
\hat f^\flat&= x \hat f^\flat_{0}+\sum_k\rho_k \hat f^\flat_{k}+
\tilde\eta \hat f^\flat_{+},
\end{split}\end{equation*}
with $f^\sharp_{k}$, etc., smooth.
We deduce that
\begin{equation}\begin{split}\label{eq:tgt-prop-quad-terms}
\ep^{-2}\delta^{-1}|f^\sharp|\leq C\ep^{-1},
\ |\hat f^\sharp|\leq C,
\end{split}\end{equation}
while
\begin{equation}\label{eq:tgt-prop-lin-terms}
\ep^{-2}\delta^{-1}|f^\bullet|\leq C\ep^{-1},
\ |\hat f^\bullet|\leq C,
\end{equation}
$\bullet=',''$, and
\begin{equation}\label{eq:tgt-prop-tgt-terms}
\ep^{-2}\delta^{-1}|f^\flat|\leq C\ep^{-1}\delta,\ |\hat f^\flat|\leq C\delta.
\end{equation}
We remark that although thus far we worked with a single $q_0\in K$, the
same construction works with $q_0$ in a neighborhood $\cU_{q'_0}$ of a fixed
$q'_0\in K$, with a {\em uniform} constant $C$. In view of the
compactness of $K$, this suffices (by the rest
of the argument we present below) to give the uniform estimate
of the proposition.

Since \eqref{eq:tgt-prop-quad-terms}-\eqref{eq:tgt-prop-tgt-terms}
are exactly the same (with slightly different notation) as
\cite[Equations~(6.16)-(6.18)]{Vasy:Maxwell}, the rest of the
proof is analogous, except that \cite[Lemma~4.6]{Vasy:Maxwell} is
replaced by Lemma~\ref{lemma:Dt-Dx} here.
Thus,
for a small constant $c_0>0$ to be determined, which we may assume to be less
than $C$, we demand below that
the expressions on the right hand sides of \eqref{eq:tgt-prop-quad-terms}
are bounded by $c_0(\ep\delta)^{-1}$,
those on the right hand sides of
\eqref{eq:tgt-prop-lin-terms} are bounded by $c_0(\ep\delta)^{-1/2}$,
while those on the right hand sides of \eqref{eq:tgt-prop-tgt-terms}
are bounded by $c_0$. This demand is due to the appearance of two, resp.\ one,
resp.\ zero, factors of $xD_x$ in \eqref{eq:gend-commutator-tgt} for
the terms whose principal symbols are affected by these, taking
into account that in view of Lemma~\ref{lemma:Dt-Dx}
we can estimate $\|Q_i v\|$ by
$C_{\cG,K}(\ep\delta)^{1/2}\|D_{y_{n-1}}v\|$
if $v$ is microlocalized to a $\ep\delta$-neighborhood of
$\cG$, which
is the case for us with $v=A_r u$ in terms of support properties of
$a$.

Thus, recalling that $c_0>0$ is to be determined, we require that
\begin{equation}\label{eq:ep-estimate}
(C/c_0)^2\delta\leq\ep\leq 1,
\end{equation}
and
\begin{equation}\label{eq:delta-estimate}
\delta<(c_0/C)^2;
\end{equation}
see \cite[Proposition~6.1]{Vasy:Maxwell} for motivation.
Then with $\ep,\delta$ satisfying \eqref{eq:ep-estimate} and
\eqref{eq:delta-estimate}, hence
$\delta^{-1}>(C/c_0)^2>C/c_0$,
\eqref{eq:tgt-prop-quad-terms}-\eqref{eq:tgt-prop-tgt-terms} give that
\begin{equation}\begin{split}\label{eq:tgt-prop-quad-terms-rev}
\ep^{-2}\delta^{-1}|f^\sharp|\leq c_0\delta^{-1}\ep^{-1},
\ |\hat f^\sharp|\leq c_0\delta^{-1}\ep^{-1},
\end{split}\end{equation}
while
\begin{equation}\label{eq:tgt-prop-lin-terms-rev}
\ep^{-2}\delta^{-1}|f^\bullet|\leq c_0\delta^{-1/2}\ep^{-1/2},
\ |\hat f^\bullet|\leq c_0\delta^{-1/2}\ep^{-1/2},
\end{equation}
$\bullet=',''$, and
\begin{equation}\label{eq:tgt-prop-tgt-terms-rev}
\ep^{-2}\delta^{-1}|f^\flat|\leq c_0,
\ |\hat f^\flat|\leq c_0,
\end{equation}
as desired. One deduces that
\begin{equation}\begin{split}\label{eq:P-comm-tgt}
&\imath A^*_r A_r P-\imath P A^*_r A_r\\
&\quad=\tilde B^*_r\big(C^*x^2 C+x R^\flat x
+(xD_x)^* \tilde R' x+x\tilde R'' (xD_x)
+(xD_x)^*R^\sharp(xD_x)\big)\tilde B_r\\
&\qquad\qquad\qquad+R''_r+E_r+E'_r
\end{split}\end{equation}
with
\begin{equation*}\begin{split}
&R^\flat\in\Psib^0(X),\ \tilde R',\tilde R''\in\Psib^{-1}(X),
\ R^\sharp\in\Psib^{-2}(X),\\
&R''_r\in L^\infty((0,1);\Diffz^2\Psib^{2s-1}(X)),
\ E_r,E'_r\in L^\infty((0,1);\Diffz^2\Psib^{2s}(X)),
\end{split}\end{equation*}
with
$$
\WFb'(E)\subset\tilde\eta^{-1}((-\delta-\ep\delta,-\delta])
\cap\omega^{-1}([0,4\delta^2\ep^2))\subset \tilde\cU
$$
(cf.\ \eqref{eq:dchi_1-supp}),
$\WFb'(E')\cap\dot\Sigma=\emptyset$,
and with
$r^\flat=\sigma_{\bl,0}(R^\flat)$, $\tilde r'=\sigma_{\bl,-1}(\tilde R')$,
$\tilde r''=\sigma_{\bl,-1}(\tilde R'')$,
$r^\sharp\in\sigma_{\bl,-2}(R^\sharp)$,
\begin{equation*}\begin{split}
&|r^\flat|\leq 2c_0+C_2\delta\digamma^{-1},
\ |\zetab_{n-1} \tilde r'|\leq 2c_0\delta^{-1/2}\ep^{-1/2}+C_2\delta\digamma^{-1},\\
&|\zetab_{n-1} \tilde r''|\leq 2c_0\delta^{-1/2}\ep^{-1/2}+C_2\delta\digamma^{-1},
\ |\zetab_{n-1}^2 r^\sharp|\leq  2c_0\delta^{-1}\ep^{-1}
+C_2\delta\digamma^{-1}.
\end{split}\end{equation*}
These are analogues of
the result of the second displayed equation after
\cite[Equation~(7.16)]{Vasy:Propagation-Wave}, as corrected in
\cite{Vasy:Propagation-Wave-Correction}, with the small (at this point
arbitrary) constant $c_0$ replacing some constants given there
in terms of $\ep$ and $\delta$; see \cite[Equation~(6.25)]{Vasy:Maxwell}
for estimates stated in exactly the same form in the form-valued
setting.
The rest of the argument thus proceeds as in
\cite[Proof of Proposition~7.3]{Vasy:Propagation-Wave}, taking into
account \cite{Vasy:Propagation-Wave-Correction}, and using
Lemma~\ref{lemma:Dt-Dx} in place of \cite[Lemma~7.1]{Vasy:Propagation-Wave}.
\end{proof}

Since for $\lambda$ real, $\lambda<(n-1)^2/4$, both forward and backward
propagation is covered by these two results, see
Remarks~\ref{rem:normal-remark} and \ref{rem:tgt-rem},
we deduce our main result on the propagation of singularities:

\begin{thm}
\label{thm:prop-sing}
Suppose that $P=\Box+\lambda$, $\lambda<(n-1)^2/4$,
$m\in\RR$ or $m=\infty$.
Suppose $u\in H^{1,k}_{0,\bl,\loc}(X)$ for some $k\leq 0$.
Then
$$
(\WFb^{1,m}(u)\cap\dot\Sigma)\setminus\WFb^{-1,m+1}(Pu)
$$
is a union of maximally extended generalized broken bicharacteristics
of the conformal metric $\hat g$ in
$$
\dot\Sigma\setminus\WFb^{-1,m+1}(Pu).
$$

In particular, if $P u=0$ then $\WFb^{1,\infty}(u)\subset\dot\Sigma$
is a union of maximally extended generalized broken bicharacteristics
of $\hat g$.
\end{thm}

\begin{proof}
The proof proceeds as in \cite[Proof of Theorem~8.1]{Vasy:Propagation-Wave},
since Propositions \ref{prop:normal-prop}
and \ref{prop:tgt-prop} are complete
analogues of \cite[Proposition~6.2]{Vasy:Propagation-Wave}
and \cite[Proposition~7.3]{Vasy:Propagation-Wave}.
Given the results of the preceding sections of \cite{Vasy:Propagation-Wave},
the argument of \cite[Proof of Theorem~8.1]{Vasy:Propagation-Wave} is
itself only a slight modification of an argument originally
due to Melrose and Sj\"ostrand \cite{Melrose-Sjostrand:I}, as presented
by Lebeau \cite{Lebeau:Propagation} (although we do not need
Lebeau's treatment of corners here).

For the convenience of the reader we give a very sketchy version of the proof.
To start with, propagation of singularities has already been proved in $X^\circ$;
this is the theorem of Duistermaat and H\"ormander \cite{FIOII, Hormander:Existence}.
Now, the theorem can easily be localized -- the global version follows by
a Zorn's lemma argument, see \cite[Proof of Theorem~8.1]{Vasy:Propagation-Wave}
for details. Indeed, in view of the Duistermaat-H\"ormander result
it suffices to show that if
\begin{equation}\label{eq:prop-103}
q_0\in\WFb^{1,m}(u)\setminus\WFb^{-1,m+1}(Pu)\Mand q_0\in \Tb^*_YX
\end{equation}
then
\begin{equation}\label{eq:prop-104}\begin{split}
&\text{there exists a generalized broken bicharacteristic}
\ \gamma:[-\ep_0,0]\to\dot\Sigma,\ \ep_0>0,\\
&\qquad\qquad \gamma(0)=q_0,
\ \gamma(s)\in\WFb^{1,m}(u)\setminus\WFb^{-1,m+1}(Pu),\ s\in[-\ep_0,0],
\end{split}\end{equation}
for the existence of a \GBBsp on $[0,\ep_0]$ can be demonstrated similarly
by replacing the forward propagation
estimates by backward ones, and, directly from
Definition~\ref{def:gen-br-bich}, piecing together the two
\GBB's gives one defined on $[-\ep_0,\ep_0]$. Note that
by microlocal
elliptic regularity, Proposition~\ref{prop:elliptic}, \eqref{eq:prop-103} implies
that $q_0\in\cG\cup\cH$.

Now suppose $q_0\in (\WFb^{1,m}(u)\setminus\WFb^{-1,m+1}(Pu))\cap
\Tb^*_YX\cap\cH$. We use the notation of
Proposition~\ref{prop:normal-prop}. Then $\gamma$ in \eqref{eq:prop-103} is
constructed by taking a sequence $q_n\to q_0$, $q_n\in T^* X^\circ$ with
$\eta(q_n)=-\xibh(q_n)<0$ and \GBBsp $\gamma_n:[-\ep_0,0]\to \dot\Sigma$
with $\gamma_n(0)=q_n$ and with $\gamma_n(s)\in
(\WFb^{1,m}(u)\setminus\WFb^{-1,m+1}(Pu))\cap T^*X^{\circ}$ for $s\in [-\ep_0,0]$.
Once this is done, by compactness of \GBBsp with image in a compact set, see
\cite[Proposition~5.5]{Vasy:Propagation-Wave} and Lebeau's paper
\cite[Proposition~6]{Lebeau:Propagation}, one can extract a uniformly
convergent subsequence, converging to some $\gamma$, giving \eqref{eq:prop-104}.
Now, the $q_n$ arise directly from Proposition~\ref{prop:normal-prop}, by shrinking
$U$ (via shrinking $\delta$ in \eqref{eq:prop-rem-9b}), namely under our assumption
on $q_0$, for each such $U$ there must exist a $q\in \WFb^{1,m}(u)$ in
$U\cap\{\eta<0\}$. The $\gamma_n$ then arise from the theorem of
Duistermaat and H\"ormander, using that $\eta(q_n)<0$ implies that the backward
\GBBsp from $q_n$ cannot meet $Y$ for some time $\ep_0$, uniform in $n$ -- this
is essentially due to $\eta$ being strictly increasing along \GBBsp
microlocally, and $\eta$ vanishing at
$\dot\Sigma\cap\Tb^*_YX$, so as long as $\eta$ is negative, the \GBBsp cannot
hit the boundary.
See \cite[Proof of Theorem~8.1]{Vasy:Propagation-Wave} for more details.

Finally, suppose $q_0\in (\WFb^{1,m}(u)\setminus\WFb^{-1,m+1}(Pu))\cap
\Tb^*_YX\cap\cG$, which is the more technical case.
This part of the argument is present in essentially the same
form in the paper of Melrose and Sj\"ostrand \cite{Melrose-Sjostrand:I};
Lebeau's paper \cite[Proposition~VII.1]{Lebeau:Propagation} gives a very nice
presentation, see \cite[Proof of Theorem~8.1]{Vasy:Propagation-Wave} for an
overview with more details. The rough idea for constructing the \GBBsp
$\gamma$ for \eqref{eq:prop-104}
is to define approximations to it
using Proposition~\ref{prop:tgt-prop}. First, recall
that in Proposition~\ref{prop:tgt-prop}, applied at $q_0$,
$W_0$ is the coordinate projection
(push forward) of $\sH_p$, evaluated at $\hat\pi^{-1}(q_0)$,
to $T^*Y$. Thus, one should think of
the point $\tilde\pi(q_0)-\delta W_0$ in $T^*Y$ as an
$\mathcal{O}(\delta^2)$ approximation
of where a backward \GBBsp should be after `time' (i.e.\ parameter value) $\delta$.
This is used as follows: given $\delta>0$,
Proposition~\ref{prop:tgt-prop}
gives the existence of a point $q_1$ in $\WFb^{1,m}(u)$ which is, roughly speaking,
$\mathcal{O}(\delta^2)$
from $\tilde\pi(q_1)-(\tilde\pi(q_0)-\delta W_0)$, with $x(q_1)$
being $\mathcal{O}(\delta^2)$ as well.
Then, from $q_1$, one can repeat this procedure (replacing $q_0$ by $q_1$ in
Proposition~\ref{prop:tgt-prop}) -- there are
some technical issues corresponding to $q_1$ being in the boundary or not, and also
whether in the former case the backward \GBBsp hits the boundary in time $\delta$.
Taking $\delta=2^{-N}\ep_0$, this gives $2^N+1$ points $q_j$
corresponding to the dyadic
points on the parameter interval $[-\ep_0,0]$. It is helpful to consider this as
analogous to a discrete approximation of solving an ODE without the presence
of the boundary by taking steps of size $2^{-N}\ep_0$. Defining $\gamma_N(s)$
for only these dyadic values, one can then get a subsequence $\gamma_{N_k}$
which converges, as $k\to\infty$,
at $s=2^{-n}j\ep_0$ for all $n\geq 1$ and $0\leq j\leq 2^n$ integers.
(Note that $\gamma_{N_k}(s)$ is defined for these values of $s$ for $k$ sufficiently
large!) One then checks
as in Lebeau's proof that the result is the restriction of a \GBBsp to dyadic parameter
values. Again, we refer to Lebeau's paper \cite[Proposition~VII.1]{Lebeau:Propagation}
and \cite[Proof of Theorem~8.1]{Vasy:Propagation-Wave} for more detail.
\end{proof}

In fact, even if $\im\lambda\neq 0$, we get one-sided statements:

\begin{thm}
\label{thm:prop-sing-im}
Suppose that $P=\Box+\lambda$, $\im\lambda>0$, resp.\ $\im\lambda<0$, and
$m\in\RR$ or $m=\infty$.
Suppose $u\in H^{1,k}_{0,\bl,\loc}(X)$ for some $k\leq 0$.
Then
$$
(\WFb^{1,m}(u)\cap\dot\Sigma)\setminus\WFb^{-1,m+1}(Pu)
$$
is a union of maximally {\em forward extended, resp.\ backward extended}
generalized broken bicharacteristics
of the conformal metric $\hat g$ in
$$
\dot\Sigma\setminus\WFb^{-1,m+1}(Pu).
$$

In particular, if $P u=0$ then $\WFb^{1,\infty}(u)\subset\dot\Sigma$
is a union of maximally extended generalized broken bicharacteristics
of $\hat g$.
\end{thm}

\begin{proof}
The proof proceeds again as for Theorem~\ref{thm:prop-sing}, but now
Propositions \ref{prop:normal-prop}
and \ref{prop:tgt-prop} only allow propagation in one direction.
Thus, if $\im\lambda<0$,
they allow one to conclude that if a point in $\dot\Sigma\setminus
\WFb^{-1,m+1}(Pu)$
is in $\WFb^{1,m}(u)$, then there is another point
in $\WFb^{1,m}(u)$ which is roughly along a {\em backward} \GBBsp segment
emanating from it. Then an actual backward \GBBsp can be constructed
as in the works of
Melrose and Sj\"ostrand \cite{Melrose-Sjostrand:I}, and
Lebeau \cite{Lebeau:Propagation}.
\end{proof}

In the absence of b-wave front set we can easily read off the actual
expansion at the boundary as well.

\begin{prop}\label{prop:asymp}
Suppose that $P=\Box+\lambda$, $\lambda\in\Cx$.
Let  $s_\pm(\lambda)=\frac{n-1}{2}\pm\sqrt{\frac{(n-1)^2}{4}-\lambda}$.
Suppose $u\in H^1_{0,\loc}(X)$, $\WFb^{1,\infty}(u)=\emptyset$ and
$Pu\in\dCI(X)$.
Then
\begin{equation}\label{eq:asymp-exp}
u=x^{s_+(\lambda)}v_+,\ v_+\in\CI(X).
\end{equation}

Conversely, if $\lambda<(n-1)^2/4$,
given any $g_+\in\CI(Y)$, there exists $v_+\in\CI(X)$,
$v_+|_Y=g_+$ such that $u=x^{s_+(\lambda)}v_+$ satisfies $Pu\in\dCI(X)$;
in particular $u\in H^1_{0,\loc}(X)$ and $\WFb^{1,\infty}(u)=\emptyset$.
\end{prop}

This proposition reiterates the importance of the constraint on $\lambda$
in that
$$
x^{(n-1)/2+i\alpha}\notin H^1_{0,\loc}(X)
$$
for $\alpha\in\RR$;
for $\lambda\geq (n-1)^2/4$, the growth or decay relative to $H^1_{0,\loc}(X)$
does not distinguish between the two approximate solutions
$x^{s_\pm(\lambda)}v_\pm$, $v_\pm\in\CI(X)$.

\begin{proof}
For the first part of the lemma,
by Lemma~\ref{lemma:WFb-infty-independence} and the subsequent remark,
under our assumptions we have $u\in \cA^{(n-1)/2}(X)$.
By \eqref{eq:Box-form},
\begin{equation}\label{eq:P-b-normal}
P+\big(((xD_x+\imath(n-1))(xD_x)-\lambda\big)\in x\Diffb^2(X).
\end{equation}
This is, up to a change in overall the sign of the second summand,
$$
(xD_x+\imath(n-1))(xD_x)-\lambda,
$$
the same as the analogous expression
in the de Sitter setting, see the first line of the proof
of Lemma~4.13 of \cite{Vasy:De-Sitter}. Thus, the proof
of that lemma goes through without changes -- the reader needs to keep in mind
that $u\in\cA^{(n-1)/2}(X)$ excludes one of the indicial roots from
appearing in the argument of that lemma. (In the De Sitter setting,
in Lemma~4.13 of \cite{Vasy:De-Sitter}, there
was no a priori weight (relative to which one has conormality) specified.)

The converse again works as in Lemma~4.13 of \cite{Vasy:De-Sitter} using
\eqref{eq:P-b-normal}.
\end{proof}

We can now state the `inhomogeneous Dirichlet problem':

\begin{thm}\label{thm:inhomog-Dirichlet}
Assume (TF) and (PT).
Suppose $\lambda<(n-1)^2/4$, and
$s_+(\lambda)-s_-(\lambda)=2\sqrt{\frac{(n-1)^2}{4}-\lambda}$
is not an integer, $P=P(\lambda)=\Box_g+\lambda$.

Given $v_0\in\CI(Y)$ and $f\in\dCI(X)$, both supported in $\{t\geq t_0\}$,
the problem
$$
Pu=f,\ u|_{t<t_0}=0,\ u=x^{s_-(\lambda)}v_-
+x^{s_+(\lambda)}v_+,\ v_\pm\in\CI(X),\ v_-|_Y=v_0,
$$
has a unique solution

If $s_+(\lambda)-s_-(\lambda)$
is an integer, the same conclusion holds if we replace
$v_-\in\CI(X)$ by $v_-=\CI(X)
+x^{s_+(\lambda)-s_-(\lambda)}\log x\,\CI(X)$.
\end{thm}

\begin{proof}
The proof of Lemma~4.13 of \cite{Vasy:De-Sitter} shows that there
exists $\tilde u$, supported in $t\geq t_0$, such that
$\tilde u=x^{s_-(\lambda)}v_-$, $v_-$ as in the statement of the theorem,
and $P\tilde u\in\dCI(X)$. Now let $u'$ be the solution of
$Pu'=f-P\tilde u$ supported in $\{t\geq t_0\}$, whose
existence follows from Theorem~\ref{thm:well-posed}, and which is
of the form $x^{s_+(\lambda)}v_+$ by Theorem~\ref{thm:prop-sing}
and Proposition~\ref{prop:asymp}.
Then $u=\tilde u+u'$
solves the PDE as stated. Uniqueness follows from the basic well-posedness
theorem, Theorem~\ref{thm:well-posed}.
\end{proof}

Finally we add well-posedness of possibly rough initial data:

\begin{thm}\label{thm:well-posed-precise}
Assume (TF) and (PT).
Suppose $f\in H^{-1,m+1}_{0,\bl,\loc}(X)$ for some $m\in\RR$, and
let $m'\leq m$. Then
\eqref{eq:mixed-problem} has a unique solution in $H^{1,m'}_{0,\bl,\loc}(X)$,
which in fact lies in $H^{1,m}_{0,\bl,\loc}(X)$,
and for all compact $K\subset X$
there exists a compact $K'\subset X$ and a constant
$C>0$ such that
$$
\|u\|_{H^{1,m}_0(K)}\leq C\|f\|_{H^{-1,m+1}_{0,\bl}(K')}.
$$
\end{thm}

\begin{rem}
It should be emphasized that if one only wants to prove this result,
without microlocal propagation, one could use more elementary
energy estimates.
\end{rem}

\begin{proof}
If $m\geq 0$, then by Theorem~\ref{thm:well-posed},
\eqref{eq:mixed-problem} has a unique solution in $H^{1}_{0,\loc}(X)$,
and by propagation of singularities it lies in $H^{1,m}_{0,\bl,\loc}(X)$,
with the desired estimate. Moreover, again by the propagation
of singularities, any solution of
\eqref{eq:mixed-problem} in $H^{1,m'}_{0,\bl,\loc}(X)$ lies
in $H^{1,m}_{0,\bl,\loc}(X)$, so the solution is indeed unique
even in $H^{1,m'}_{0,\bl,\loc}(X)$.

If $m<0$, uniqueness and the stability estimate follow as above. To see
existence, let $T_0<t_0$, and
let $f_j\to f$ such that $f_j\in H^{-1,1}_{0,\bl,\loc}$ and
$\supp f_j\subset\{t>T_0\}$. This can be achieved by taking
$A_r\in\Psibc^{-\infty}(X)$ with properly supported Schwartz kernel
(of sufficiently small support)
such that $\{A_r:\ r\in(0,1]\}$ is a bounded family in $\Psibc^0(X)$,
converging to $\Id$ in $\Psibc^{\ep}(X)$ for $\ep>0$, then with
$f_j=A_{r_j}f$, $r_j\to 0$, we have the desired properties.
By Theorem~\ref{thm:well-posed}, \eqref{eq:mixed-problem} with
$f$ replaced by $f_j$ has a unique solution $u_j\in H^1_{0,\loc}(X)$.
Moreover, by the propagation of singularities, one has a uniform estimate
$$
\|u_k-u_j\|_{H^{1,m}_0(K)}\leq C\|f_k-f_j\|_{H^{-1,m+1}_{0,\bl}(K')},
$$
with $C$ independent of $j,k$. In view of the convergence of the $f_j$
in $H^{-1,m+1}_{0,\bl}(K')$, we deduce the convergence of the
$u_j$ in $H^{1,m}_{0,\bl}(K)$ to some $u\in H^{1,m}_{0,\bl}(K)$,
hence (by uniqueness)
we deduce the existence of $u\in H^{1,m}_{0,\bl,\loc}(X)$ solving
$Pu=f$ with support in $\{t\geq T_0\}$. However, as $\supp f\subset
\{t\geq t_0\}$, uniqueness shows the vanishing of $u$ on $\{t<t_0\}$,
proving the theorem.
\end{proof}

\def\cprime{$'$} \def\cprime{$'$}

\end{document}